\documentclass[a4paper]{article}

\usepackage[T1]{fontenc}            % Better font encoding (supports accented characters, better hyphenation)
\usepackage[utf8]{inputenc}         % Input encoding (UTF-8), needed if using pdflatex
\usepackage{lmodern}                % Modern version of standard Latex font

\usepackage{lipsum}                 % Dummy text for testing
\usepackage{todonotes}              % To-do notes: \todo{...} and \todo[inline]{...}
\usepackage{appendix}               % Extra appendix features
\usepackage{makeidx}                % Create index of terms
\usepackage{comment}                % Block comment large text sections

\usepackage{titlesec}               % Customize (centered) section titles
\titleformat{\section}{\centering\Large\bfseries}{\thesection}{1em}{}
\titleformat{\subsection}{\centering\large\bfseries}{\thesubsection}{1em}{}
\titleformat{\subsubsection}{\centering\normalsize\bfseries}{\thesubsubsection}{1em}{}
\usepackage{changepage}             % Change margins locally (e.g., adjustwidth)
\usepackage{enumerate}              % Enhanced enumerate environment
\usepackage{array}                  % Additional tabular features
\usepackage{multirow}               % Multi-row cells in tables
\usepackage{arydshln}               % Dashed lines in tables
\usepackage{rotating}               % Rotate tables or figures
\usepackage{booktabs}               % For better-looking tables
\usepackage{nth}                    % English ordinal superscripts (\nth{3} → 3rd)
\usepackage{setspace}               % Control line spacing, e.g., \doublespacing
\usepackage{pifont,mdframed}        % Dingbats and framed boxes (e.g., warning boxes)
\usepackage{framed}                 % Allows to put frames around sections
\usepackage{authblk}                % allows to specify affiliation of the authors; via \affil

\usepackage{scalerel,stackengine}
\stackMath
\newcommand\reallywidehat[1]{%
  \savestack{\tmpbox}{\stretchto{%
    \scaleto{%
      \scalerel*[\widthof{\ensuremath{#1}}]{\kern-.6pt\bigwedge\kern-.6pt}%
      {\rule[-\textheight/2]{1ex}{\textheight}}%
    }{\textheight}%
  }{0.5ex}}%
  \stackon[1pt]{#1}{\tmpbox}%
}
\usepackage{color}                  % Colors in text and figures
\usepackage{graphicx}               % Include graphics files
\usepackage{caption}                % Customize figure/table captions
\usepackage{subcaption}             % Subfigures and subtables captions
\usepackage{wrapfig}                % Wrap text around figures

\usepackage{tikz}                   % Draw graphics and diagrams
\usepackage{pgfplots}               % Create plots and graphs
\pgfplotsset{compat=1.18}    

\usetikzlibrary{shapes.geometric,arrows}        % shapes like triangles, diamonds; arrow tips for edges
\usetikzlibrary{calc,decorations.pathreplacing} % coordinate calculations; decorate paths (e.g. braces)

\usepackage{tikz-cd}                % Commutative diagrams with TikZ
\usepackage{tikz-qtree}             % Trees and flowcharts with TikZ
\usetikzlibrary{backgrounds}        % Required for background layer

\usepackage{amsmath}                % Enhanced math features
\usepackage{amssymb}                % Extra math symbols
\usepackage{amsthm}                 % Theorem environments
\usepackage{mathtools}              % Extensions to amsmath
\usepackage{mathrsfs}               % \mathscr font
\usepackage{bbm}                    % Blackboard bold for numbers
\usepackage{nicefrac}               % Slanted fractions (\nicefrac{1}{2})
\usepackage{bm}                     % Bold math symbols (\bm)
\usepackage{stmaryrd}               % Extra symbols (e.g., double brackets)
\usepackage{thmtools}               % Enhanced theorem features, list of theorems
\usepackage{dsfont}                 % Different style blackboard bold fonts
\usepackage{accents}                % for the circ over a math letter

\def\XXint#1#2#3{%
  {\setbox0=\hbox{$#1{#2#3}{\int}$}%
   \vcenter{\hbox{$#2#3$}}\kern-.6\wd0}%
}

\usepackage[pdftoolbar=true,pdfmenubar=true,pdfstartview=FitV,pdfdisplaydoctitle=true,pdfencoding=auto,final]{hyperref}
\hypersetup{
	colorlinks=true,
	allcolors=blue
}

\usepackage[backend=biber,style=numeric,uniquename=init,useprefix=true,hyperref=true]{biblatex}
\AtEveryBibitem{%
	\iffieldundef{journaltitle}{}{\clearfield{eprint}}%
	\iffieldundef{doi}{}{\clearfield{url}\clearfield{urlyear}\clearfield{urlmonth}\clearfield{urlday}}%
	\ifentrytype{online}{}{\clearfield{url}\clearfield{urlyear}\clearfield{urlmonth}\clearfield{urlday}}%
	\iffieldequalstr{eprinttype}{arXiv}{\clearfield{doi}}{}%
	\clearfield{issue}%
	\clearfield{month}\clearfield{day}%
	\clearfield{pagetotal}%
}
\AtEveryCitekey{%
	\iffieldundef{journaltitle}{}{\clearfield{eprint}}%
	\iffieldundef{doi}{}{\clearfield{url}\clearfield{urlyear}\clearfield{urlmonth}\clearfield{urlday}}%
	\ifentrytype{online}{}{\clearfield{url}\clearfield{urlyear}\clearfield{urlmonth}\clearfield{urlday}}%
	\iffieldequalstr{eprinttype}{arXiv}{\clearfield{doi}}{}%
	\clearfield{issue}%
	\clearfield{month}\clearfield{day}%
	\clearfield{pagetotal}%
}

\DeclareFieldFormat[article,incollection]{title}{#1}
\DeclareFieldFormat[article]{volume}{\textbf{#1}\setunit*{\space}}
\renewcommand*{\bibeidpunct}{\addcolon}

\renewbibmacro*{volume+number+eid}{%
	\printfield{volume}%
	\setunit*{\adddot}%
	\printfield{number}%
	\setunit*{\bibeidpunct}%
	\printfield{eid}%
	\setunit{\addcomma\space}%
}

\renewbibmacro{in:}{
	\ifentrytype{article}{}{\printtext{\bibstring{in}\intitlepunct}}%
}

\renewbibmacro*{journal}{%
	\iffieldundef{shortjournal}%
		{%
			\iffieldundef{journaltitle}%
			{}%
			{\printtext[journaltitle]%
				{%
					\printfield[titlecase]{journaltitle}%
					\setunit{\subtitlepunct}%
					\printfield[titlecase]{journalsubtitle}%
				}%
			}%
		}%
		{%
			\printtext[journaltitle]%
				{%
					\printfield[titlecase]{shortjournal}%
				}%
		}%
}

\makeatletter

\DeclareFieldFormat{doi}{%
	\ifhyperref
		{\href{https://doi.org/#1}{\mkbibacro{DOI}\addcolon\space\nolinkurl{#1}}}
		{\mkbibacro{DOI}\addcolon\space\nolinkurl{#1}}}
	
\DeclareFieldFormat{eprint:hdl}{%
	\ifhyperref
		{\href{http://hdl.handle.net/#1}{HDL\addcolon\space\nolinkurl{#1}}}
		{HDL\addcolon\space\nolinkurl{#1}}}
	
\DeclareFieldFormat{eprint:arxiv}{%
	\ifhyperref
		{\href{https://arxiv.org/\abx@arxivpath/#1}{%
			arXiv\addcolon\space\nolinkurl{#1}%
			\iffieldundef{eprintclass}
				{}
				{\addspace\texttt{\mkbibbrackets{\thefield{eprintclass}}}}}}
		{arXiv\addcolon\space\nolinkurl{#1}%
		 \iffieldundef{eprintclass}
			{}
			{\addspace\texttt{\mkbibbrackets{\thefield{eprintclass}}}}}}

\DeclareFieldFormat{eprint:jstor}{%
	\ifhyperref
		{\href{http://www.jstor.org/stable/#1}{JSTOR\addcolon\space\nolinkurl{#1}}}
		{JSTOR\addcolon\space\nolinkurl{#1}}}

\DeclareFieldFormat{eprint:pubmed}{%
	\ifhyperref
		{\href{http://www.ncbi.nlm.nih.gov/pubmed/#1}{PMID\addcolon\space\nolinkurl{#1}}}
		{PMID\addcolon\space\nolinkurl{#1}}}

\DeclareFieldFormat{eprint:googlebooks}{%
	\ifhyperref
		{\href{http://books.google.com/books?id=#1}{Google Books\addcolon\space\nolinkurl{#1}}}
		{Google Books\addcolon\space\nolinkurl{#1}}}
		
\makeatother

\theoremstyle{plain}                        % default style: italic text, bold title
\newtheorem{thm}{Theorem}[section]
\newtheorem{lem}[thm]{Lemma}
\newtheorem{prop}[thm]{Proposition}

\theoremstyle{definition}
\newtheorem{defn}[thm]{Definition}
\newtheorem*{defn*}{Definition}

\newtheorem{exmp}[thm]{Example}
\theoremstyle{remark}
\newtheorem*{rem}{Remark}

\numberwithin{equation}{section}

\usepackage{mdframed}   % already in your preamble
\usepackage{pifont}     % for \ding{43} (pointing finger)

\DeclareMathOperator{\id}{id}                % Identity
\DeclareMathOperator{\supp}{supp}            % Support

\DeclareMathOperator{\ev}{ev}                % Evaluation functional
\DeclareMathOperator{\dd}{\mathrm{d}\!}       % Differential d
\DeclareMathOperator{\Ren}{Ren}               % Renormalized

\DeclareMathOperator{\fA}{\mathfrak{A}}
\DeclareMathOperator{\fB}{\mathfrak{B}}
\DeclareMathOperator{\fC}{\mathfrak{C}}

\DeclareMathOperator{\fL}{\mathfrak{L}}
\DeclareMathOperator{\fM}{\mathfrak{M}}

\DeclareMathOperator{\fP}{\mathfrak{P}}

\DeclareMathOperator{\fV}{\mathfrak{V}}

\DeclareMathOperator{\fX}{\mathfrak{X}}

\DeclareMathOperator{\fq}{\mathfrak{q}}

\DeclareMathOperator{\bE}{\mathbf{E}}

\DeclareMathOperator{\bT}{\mathbf{T}}

\DeclareMathOperator{\bh}{\mathbf{h}}

\DeclareMathOperator{\bx}{\mathbf{x}}
\DeclareMathOperator{\by}{\mathbf{y}}

\DeclareMathOperator{\bmu}{\bm{\mu}}

\DeclareMathOperator{\bbB}{\mathbb{B}}

\DeclareMathOperator{\bbE}{\mathbb{E}}

\DeclareMathOperator{\bbN}{\mathbb{N}}

\DeclareMathOperator{\bbP}{\mathbb{P}}

\DeclareMathOperator{\bbR}{\mathbb{R}}

\DeclareMathOperator{\bbT}{\mathbb{T}}

\DeclareMathOperator{\bbW}{\mathbb{W}}
\DeclareMathOperator{\bbX}{\mathbb{X}}

\DeclareMathOperator{\calB}{\mathcal{B}}
\DeclareMathOperator{\calC}{\mathcal{C}}

\DeclareMathOperator{\calF}{\mathcal{F}}
\DeclareMathOperator{\calG}{\mathcal{G}}

\DeclareMathOperator{\calI}{\mathcal{I}}

\DeclareMathOperator{\calN}{\mathcal{N}}

\DeclareMathOperator{\calP}{\mathcal{P}}

\DeclareMathOperator{\calT}{\mathcal{T}}

\DeclareMathOperator{\calW}{\mathcal{W}}

\DeclareMathOperator{\rmP}{\mathrm{P}}

\DeclareMathOperator{\rmi}{\mathrm{i}}

\DeclareMathOperator{\rmm}{\mathrm{m}}

\DeclareMathOperator{\sB}{\mathscr{B}}

\DeclareMathOperator{\sH}{\mathscr{H}}
\DeclareMathOperator{\sI}{\mathscr{I}}
\DeclareMathOperator{\sJ}{\mathscr{J}}
\DeclareMathOperator{\sK}{\mathscr{K}}

\DeclareMathOperator{\sN}{\mathscr{N}}

\DeclareMathOperator{\sR}{\mathscr{R}}

\usetikzlibrary{shapes.misc, shapes.symbols}

\usetikzlibrary{shapes.misc, shapes.symbols}
\tikzset{
  dot/.style={circle, fill=black, draw=black, solid, inner sep=0pt, minimum size=0.5mm},
  yy/.style={circle, fill=gray!20, draw=black, inner sep=0pt, minimum size=0.8mm},
  >=stealth,
}

\makeatletter
\def\DeclareSymbol#1#2#3{%
  \expandafter\gdef\csname GC@symb@#1\endcsname{%
    \tikz[baseline=#2, scale=0.15]{#3}%
  }%
  \expandafter\gdef\csname GC@symb@#1s\endcsname{%
    \scalebox{0.6}{\tikz[baseline=#2, scale=0.15]{#3}}%
  }%
}
\def\<#1>{\csname GC@symb@#1\endcsname}
\makeatother

\DeclareSymbol{X}{-2.4}{\node[dot] {};}
\DeclareSymbol{1}{0}{\draw[white] (-.4,0) -- (.4,0); \draw (0,0) -- (0,1.2) node[dot] {};}
\DeclareSymbol{2}{0}{\draw (-0.5,1.2) node[dot] {} -- (0,0) -- (0.5,1.2) node[dot] {};}
\DeclareSymbol{3}{0}{%
  \draw (0,0) -- (0,1.2) node[dot] {};
  \draw (-.7,1) node[dot] {} -- (0,0) -- (.7,1) node[dot] {};
}

\DeclareSymbol{31}{-3}{%
  \draw (0,0) -- (0,-1) -- (1,0) node[dot] {};
  \draw (0,0) -- (0,1.2) node[dot] {};
  \draw (-.7,1) node[dot] {} -- (0,0) -- (.7,1) node[dot] {};
}
\DeclareSymbol{30}{-3}{%
  \draw (0,0) -- (0,-1);
  \draw (0,0) -- (0,1.2) node[dot] {};
  \draw (-.7,1) node[dot] {} -- (0,0) -- (.7,1) node[dot] {};
}
\DeclareSymbol{32}{-3}{%
  \draw (0,0) -- (0,-1) -- (1,0) node[dot] {};
  \draw (0,0) -- (0,-1) -- (-1,0) node[dot] {};
  \draw (0,0) -- (0,1.2) node[dot] {};
  \draw (-.7,1) node[dot] {} -- (0,0) -- (.7,1) node[dot] {};
}
\DeclareSymbol{22}{-3}{%
  \draw (0,0.3) -- (0,-1) -- (1,0) node[dot] {};
  \draw (0,0.3) -- (0,-1) -- (-1,0) node[dot] {};
  \draw (-.7,1) node[dot] {} -- (0,0.3) -- (.7,1) node[dot] {};
}
\DeclareSymbol{20}{-3}{%
  \draw (0,0) -- (0,-1);
  \draw (-.7,1) node[dot] {} -- (0,0) -- (.7,1) node[dot] {};
}

\DeclareSymbol{32W}{-3}{%
  \draw[densely dotted, semithick] (-1.5,0) node[dot] {} -- (0,-1.5) node[yy]{} -- (1.5,0) node[dot] {};
  \draw (.8,-2.3) node {\tiny $y$};
  \draw[semithick] (0,0) -- (0,-1.5) node[yy] {};
  \draw[densely dotted, semithick] (0,0) -- (0,1.8) node[dot] {};
  \draw[densely dotted, semithick] (-1,1.5) node[dot] {} -- (0,0) -- (1,1.5) node[dot] {};
}
\DeclareSymbol{32WW}{-3}{%
  \draw[densely dotted, semithick] (-1.5,0) node[dot] {} -- (0,-1.5) node[yy]{} -- (1.5,0) node[dot] {};
  \draw (.8,-2.3) node {\tiny $\bar{y}$};
  \draw[semithick] (0,0) -- (0,-1.5) node[yy] {};
  \draw[densely dotted, semithick] (0,0) -- (0,1.8) node[dot] {};
  \draw[densely dotted, semithick] (-1,1.5) node[dot] {} -- (0,0) -- (1,1.5) node[dot] {};
}

\usepackage[
  a4paper,
  top=3.5cm,
  bottom=3.5cm,
  left=2.5cm,
  right=2.5cm
]{geometry}

\usepackage{fix-cm}   % avoid font warnings
\usepackage{shuffle}  % shuffle symbols
\DeclareFontFamily{U}{shuffle}{}
\DeclareFontShape{U}{shuffle}{m}{n}{ <-8>shuffle7 <8->shuffle10}{}

\DeclareMathOperator{\GRP}{GRP}
\DeclareMathOperator{\PL}{PL}
\DeclareMathOperator{\KL}{KL}
\DeclareMathOperator{\var}{var}
\DeclareMathOperator{\sbH}{\pmb{\sH}}

\newcommand{\vertiii}[1]{{\left\vert\kern-0.25ex\left\vert\kern-0.25ex\left\vert #1 
    \right\vert\kern-0.25ex\right\vert\kern-0.25ex\right\vert}_{\mathbf{E}}}

\numberwithin{equation}{section} % equations numbered by section
\newenvironment{qn}{\par\noindent\begin{minipage}{\linewidth}}{\end{minipage}\par} % suppress page break

\usepackage{orcidlink} % ORCiD links

\makeatletter
\renewcommand{\@maketitle}{
  \begin{center}
    \vskip 0.5em
    {\LARGE\scshape \@title \par} % title in small caps
    \vskip 0.5em
    {\large \@author \par}        % authors
    \vskip 1em
    {\@date \par}                  % date
  \end{center}
  \par\vskip 1em
}
\makeatother
\begin{document}

	\title{Towards Abstract Wiener Model Spaces}
    \author[1]{Gideon Chiusole \thanks{GC gratefully acknowledges support by the doctoral program TopMath and partial funding by the VolkswagenStiftung via a Lichtenberg Professorship awarded to Christian Kuehn. GC would also like to thank Tom Klose and Martin Hairer for helpful discussions and TU Berlin for its hospitality during the summer of 2022, when GC was a visiting student there and this work was started. Email: \href{mailto:g.chiusole@tum.de}{g.chiusole@tum.de}}\,\orcidlink{0009-0005-3930-2924}}
    \author[2,3]{Peter K. Friz\thanks{PKF acknowledges funding by the Deutsche Forschungsgemeinschaft (DFG, German Research Foundation) - CRC/TRR 388 ``Rough Analysis, Stochastic Dynamics and Related Fields'' - Project ID 516748464 and the excellence cluster MATH+, via a MATH+ Distinguished Fellowship. Email: \href{mailto:friz@math.tu-berlin.de}{friz@math.tu-berlin.de}}\,\orcidlink{0000-0003-2571-8388}}
    \affil[1]{Technische Universit\"at M\"unchen, Munich, Germany.}
    \affil[2]{Technische Universit\"at Berlin, Berlin, Germany.}
    \affil[3]{Weierstraß-Institut f\"ur Angewandte Analysis und Stochastik, Berlin, Germany.}
	\date{}

\maketitle

\vspace*{-20pt}

\begin{abstract}
Abstract Wiener spaces are in many ways the decisive setting for fundamental results on Gaussian measures: large deviations (Schilder), quasi-invariance (Cameron--Martin), differential calculus (Malliavin), support description (Stroock--Varadhan), concentration of measure (Fernique), etc. Analogues of these classical results have been derived in the ``enhanced'' context of Gaussian rough paths and, more recently, regularity structures equipped with Gaussian models. The aim of this article is to propose a similar notion \textit{directly} on this enhanced level - an {\em abstract Wiener model space} - that encompasses the aforementioned. More specifically, we focus here on enhanced Schilder type results, Cameron--Martin shifts and Fernique estimates, offering a somewhat unified view on results of Friz--Victoir and Hairer--Weber.
\end{abstract}



\newpage
\tableofcontents

\newpage
\maketitle

\section{Introduction}

The study of abstract Wiener spaces (AWS) was arguably initiated with the insight of L. Gross \cite{grossAbstractWienerSpaces1967} that analysis on Wiener space $(C([0,1]; \bbR), \mu) \label{symb:continuous_functions}$ (i.e. the Banach space of continuous functions, equipped with the distribution $\mu$ of a Brownian motion) does not rely so much on the properties of $C([0,1]; \bbR)$ itself, but rather on a Hilbert subspace $W^{1,2}_0([0,1];\bbR) \label{symb:sobolev_space_starting_at_0}$ (the Hilbert space of square integrable functions starting at $0$ which have a square integrable weak derivative), the Cameron--Martin space. In the paper, Gross starts from an abstract separable Hilbert space $\sH$ and constructs from it a separable Banach space $E$ in which $\sH$ is included along an injection $i$ and on which a Gaussian measure $\mu$ (which is determined by $\sH$) is supported, thus forming a quadruple $(E, \sH, i, \mu)$ called abstract Wiener space. Common examples are given in Table \ref{tab:AWS}. 

{
\renewcommand{\arraystretch}{2}
\begin{table}
\centering
    \begin{tabular}{|l|l|l|}
        \hline
            \textbf{Gaussian Measure $\mu$} & \textbf{Hilbert Space} $\mathscr{H}$ & \textbf{Banach Space} $E$ \\ 
            \hline \hline
            Mult. Gaussian $\sN(0, \Sigma)$ & $\left( \bbR^d, \langle \cdot, \Sigma^{-1} \cdot \rangle_{\bbR^d} \right)$ & $\left( \bbR^d, \langle \cdot, \cdot \rangle_{\bbR^d} \right), \ldots$ \\ 
            \hline
            Brownian Motion & $\left(W_{0}^{1,2}([0,T]), \int_0^T (\cdot)' (\cdot)' \dd \lambda \right)$ & $C [0,T], \mathcal{C}^{0,\frac{1}{2}-\kappa}[0,T] \label{symb:hoelder_space_closure_of_smooth}, \ldots$ \\ 
            \hline
            $\beta$-fractional Brownian Motion & $\left( \dot{H}^{\beta+\frac{1}{2}}, \left\langle \cdot, (-\Delta)^{\beta+\frac{1}{2}} \cdot \right\rangle_{L^2} \right)$ & $\mathcal{C}^{0,\beta-\kappa}, \ldots$ \\ 
            \hline 
            Space-time White Noise & $\left( L^2(\bbT^d), \langle \cdot, \cdot \rangle_{L^2(\bbT^d)} \right)$ & $\mathcal{C}_{\mathfrak{s}}^{0, -\frac{d+2}{2}- \kappa}\left(\bbT^d \right), \ldots$ \\ 
            \hline 
            Dirichlet Gaussian Free Field & $\left( \dot{H}_0^1(U) \label{symb:homogeneous_sobolev_hilbert_dirichlet}, \int_U \langle \nabla \cdot, \nabla \cdot \rangle \dd \lambda^d \right)$ & $\dot{H}^{- \frac{d-2}{2}- \kappa}(U), \ldots$ \\
            \hline  
    \end{tabular}
        \captionsetup{width=0.8\linewidth}
        \caption{Examples of abstract Wiener spaces. $T > 0, \kappa > 0$. See Appendix \ref{app:symbolic_index} for symbols.}
    \label{tab:AWS}
\end{table}
}

Abstract Wiener spaces allow to state, understand, and prove many fundamental theorems on classical Wiener space in a more general language. In particular, in the decades after Gross' foundational paper, the theory became the basis for many results of Gaussian measure theory, generalizing propositions about the classical Wiener measure to general Gaussian measures. For example, in the setting above the following is true. 

\begin{itemize}
\item(Schilder's Large Deviation Principle) The family $(\mu(\varepsilon^{-1} (\cdot)))_{\varepsilon > 0}$ satisfies a large deviation principle (LDP) on $E$ with speed $\varepsilon^2$ and good rate function given by 

\begin{equation}
\sI(x) = 
    \begin{cases} 
        \frac{1}{2} \Vert x \Vert_{\sH}^{2}, & \text{if } x \in \sH, \\
        + \infty, & \text{else} .
    \end{cases} 
\end{equation}

\item (Cameron--Martin Theorem and Formula) For any $x \in E$ the measures $\mu(\cdot)$ and $\mu(\cdot - x)$ are equivalent\footnote{In the sense that either measure is absolutely continuous w.r.t. the other.} if and only if $x \in \sH$. Otherwise they are mutually singular. 

\item (Malliavin Calculus) The distribution of a (non-linear) Wiener functional $\Psi: E \rightarrow \bbR$ has a density with respect to the Lebesgue measure whenever the $\sH$-derivative/Malliavin derivative (not the Fréchet derivative) of $\Psi$ is non-degenerate.

\item (Support Theorem) The topological support of $\mu$ in $E$ is given by the $E$-closure of $\sH$.

\item (Fernique Estimates) The random variable $x \mapsto \Vert x \Vert_E$ has Gaussian tails with decay rate controlled by the values of the $\sH$-norm on the unit sphere in $E$.
\end{itemize}

In particular, abstract Wiener spaces play an important role when studying the solutions to stochastic differential equations, which are after all functionals on a space of generalized paths/fields which is equipped with the distribution of the driving noise. Thus, with the advent of rough paths through \cite{lyonsDifferentialEquationsDriven1998}, \cite{gubinelliControllingRoughPaths2004} and regularity structures through \cite{hairerTheoryRegularityStructures2014} it comes to no surprise that many of the results above have their analogues in the context of rough paths/regularity structures - see Table \ref{tab:classical_vs_enhanced}.

{
\renewcommand{\arraystretch}{2}
\begin{table}
\centering
    \begin{tabular}{|l|l|l|}
        \hline
    \textbf{Theorem/Theory} & \textbf{Classical} & \textbf{RP \& Reg. Structures} \\ \hline \hline
    Large Deviations & e.g. \cite[Sec. 3.4]{deuschelLargeDeviations1989} & \cite{frizMultidimensionalStochasticProcesses2010}, \cite{milletLargeDeviationsRough2006}, \cite{frizLargeDeviationPrinciple2007}, \cite{hairerLargeDeviationsWhiteNoise2015} \\ \hline
    Cameron--Martin & e.g. \cite[Sec. 4.2]{bogachevGaussianMeasures1998} & \cite[Sec. 15.8]{frizMultidimensionalStochasticProcesses2010} \\ \hline
    Malliavin Calculus & e.g. \cite{nualartMalliavinCalculusRelated2006} & \cite{cassMalliavinCalculusRough2011}, \cite{schonbauerMalliavinCalculusDensities2023}, \cite{cannizzaroMalliavinCalculusRegularity2017} \\ \hline
    Support Theory & e.g. \cite[Sec. 9.3]{frizCourseRoughPaths2020} & \cite{choukSupportTheoremSingular2018}, \cite{hairerSupportSingularStochastic2022}, \cite[Sec. 15.8]{frizMultidimensionalStochasticProcesses2010} \\ \hline
    Fernique Estimates & e.g. \cite[Chap. III]{deuschelLargeDeviations1989}, \cite{ledouxIsoperimetryGaussianAnalysis1996} &  \cite{frizGeneralizedFerniqueTheorem2010}, \cite{frizMultidimensionalStochasticProcesses2010}, \cite{hairerLargeDeviationsWhiteNoise2015}, \cite{frizPreciseLaplaceAsymptotics2022} \\ \hline 
    \end{tabular}
        \captionsetup{width=0.8\linewidth}
        \caption{Classical vs. enhanced theory}
    \label{tab:classical_vs_enhanced}
\end{table}
}

The general setting in this ``enhanced'' framework of rough paths \& regularity structures is to consider a ``lift'' $\hat{\fL}$ from a space $E$ of (classical, generalized) paths/fields to a (non-linear subspace of a) direct sum of spaces $E_{\tau}$ of functions/distributions, each of which is usually a closure of smooth two-parameter functions under a Hölder or $p$-variation type norm (often referred to as rough path norm or model norm). This lift assigns to an element $x \in E$ its associated rough path/model $\bx_{\tau} \in E_{\tau}$ associated to a symbol $\tau$. While in principle entirely deterministic, in its natural context of stochastic (partial) differential equations, $E$ carries a probability measure $\mu$, depending on the nature of the noise driving the problem. If the noise is Gaussian, then $\mu$ is Gaussian, which equips $E$ with the structure of an abstract Wiener space (see Section \ref{sec:examples} for serveral instances of this setup). The results in the third column of Table \ref{tab:classical_vs_enhanced} are then derived from properties of that abstract Wiener space and the lift $\hat{\fL}$. \\ 

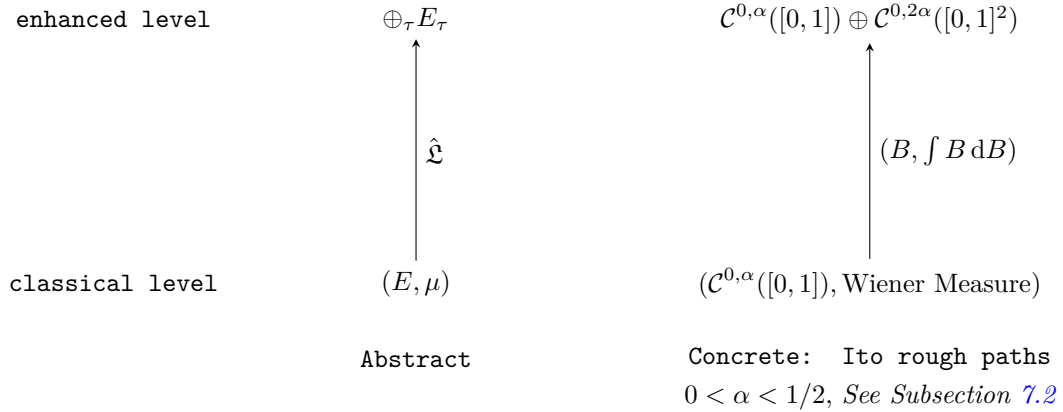
\begin{figure}[ht]
\begin{center}
            \begin{tikzpicture}
        \path
        (-4,0) node(1b) {\texttt{classical level}} 
        (-4,3.5) node(1t) {\texttt{enhanced level}}

        (0,-1) node(z1) {\texttt{Abstract}} 
        (6,-1) node(z2) {\texttt{Concrete: Ito rough paths}}
        (6,-1.5) node(z3) {$0 < \alpha < 1/2$, \textit{See Subsection \ref{subsec:Ito}}}
        
        (0,3.5) node(2t) {$\oplus_{\tau} E_{\tau}$} 
        (0,0) node(2b) {$(E, \mu)$}

        (6,3.5) node(3t) {$\calC^{0, \alpha}([0,1]) \oplus \calC^{0,2 \alpha}([0,1]^2)$} 
        (6,0) node(3b) {$(\calC^{0,\alpha}([0,1]), \text{Wiener Measure})$};
        
        \draw[->] (2b) to node[right] {$\hat{\fL}$} (2t);
        \draw[->] (3b) to node[right] {$(B, \int B \dd B)$} (3t);
        \end{tikzpicture}

        \caption{Diagram of abstract and concrete setting.}
\end{center}
\end{figure}
 
The purpose of this article is to give some ideas about such a construction \textit{directly} on the enhanced level of rough paths/models; i.e. about an \textit{abstract Wiener model space (AWMS)}. We have been guided by the goal of embedding the works \cite{frizLargeDeviationPrinciple2007} and \cite{hairerLargeDeviationsWhiteNoise2015}, that deal with enhanced Gaussian large deviations in the context of rough paths and regularity structures, respectively, into a common abstract framework. Let us emphasise that it is not our primary goal to produce theorems previously unknown, but rather to give a clarifying framework for existing results (similar to the role of abstract Wiener spaces in Gaussian analysis). \\

\begin{rem}[Para-controlled calculus and Polchinski flow]
It is a natural question to what extent our results have correspondences in the context of other approaches to singular SPDEs such as para-controlled calculus and the Polchinski flow. In the case of the paracontrolled approach of \cite{gubinelliParacontrolledDistributionsSingular2015} our results can certainly be reformulated, leveraging the results of Bailleul--Hoshino \cite{bailleulParacontrolledCalculusRegularity2021} and references therein, who established a precise correspondence between modelled distributions in the sense of Hairer and higher-order paracontrolled objects. (Although we have not obtained a general support theorem in our context of abstract Wiener model spaces, the work \cite{hairerSupportSingularStochastic2022} suggests that results in this directions may be possible, in which case \cite{choukSupportTheoremSingular2018} stands as model case in the paracontrolled setting.)

As for the flow-equation (Polchinski) approach to singular SPDEs by Duch \cite{duchFlowEquationApproach2025} the situation is less clear. To the best of our knowledge, there does not yet exist a dictionary or translation between the objects appearing in regularity structures/modelled distributions and those in the flow-equation formalism. For this reason, we do not currently have a precise statement of how our theorems would apply in that context.
\end{rem}

\paragraph{This article is organised as follows:}
In Section \ref{sec:technical_setup} we introduce the building blocks of the theory, which are, roughly speaking the elements assembled in Figure \ref{fig:AWMS_first_appearence}. These consist primarily of (1) elements on the ``classical level'', (2) elements on the ``enhanced level'', and (3) elements relating the two. We also exhibit the central problem with a naive definition of an ``enhanced Cameron--Martin space'' and sketch how it leads to two natural approaches, neatly mirroring the ideas of Stratonovich rough paths and Ito rough paths, respectively. 

With all the technical machinery in place, in Section \ref{sec:defn_and_construction}, we finally give the definition of a (bare) AWMS as well as the stronger notion of an AWMS \emph{with approximation}. In Theorem \ref{thm:top_down_construction} we give a construction of an AMWS from certain data primarily on the basis of the notion of the full lift, following the ``Top-Down'' approach outlined in Section \ref{sec:technical_setup}, which is inspired by the construction of the Ito lift from Brownian motion. In Theorem \ref{thm:bottom_up_construction} we provide a theorem similar in faith, but following the ``Bottom-Up'' approach, which primarily relies on the skeleton lift and is more aligned with the ideas around Gaussian rough paths. 

In Section \ref{sec:LDP}, \ref{sec:Fernique}, and \ref{sec:CM}, we prove a large deviation principle (Theorem \ref{thm:LDP_for_AWMS}), a Fernique estimate (Theorem \ref{thm:exp_int}), and a Cameron--Martin theorem (Theorem \ref{thm:lifted_CM_thm} and \ref{thm:shift_operator_consistency}) for AWMS, respectively. 

Finally, in Section \ref{sec:examples} we show that many examples fit the developed framework and apply the theorems of Sections \ref{sec:LDP}, \ref{sec:Fernique}, and \ref{sec:CM}. The examples considered include Gaussian rough paths (and thus in particular the Stratonovich lift of Brownian motion), the Ito lift of Brownian motion, and regularity structures associated to rough volatility, the $\Phi^4_d$-model, and the parabolic Anderson model. 

\section{Technical Setup} \label{sec:technical_setup}

Let us make a preliminary sketch: in analogy to abstract Wiener spaces and motivated by the preceding paragraphs, an abstract Wiener model space should (at least) consist of some state space $\bE$ on which a measure $\bmu$ is supported, the behaviour of which is controlled by a subspace $\sbH \subseteq \bE$. All of the above should be linked with classical abstract Wiener space theory (in that the latter should be a special case) and with rough paths/regularity structures (in that $\bmu$ should be the distribution of an enhancement $\hat{\fL}$ as alluded to before). The definition we will eventually arrive at is the following.

\begin{defn*}[Abstract Wiener Model Space]
    An \textbf{abstract Wiener model space} is a quintuple consisting of the following data: 
    
    \begin{enumerate}[(1)]
    \item An ambient space $(\calT, \bE, [\cdot], \calN)$,
    
    \item a subset $\sbH \subseteq \bE$, called the \textbf{enhanced Cameron--Martin space}, 
        
    \item a Borel probability measure $\bmu$ on $\bE$, called \textbf{enhanced measure}, such that $\mu := \pi_{\ast} \bmu$ is centred Gaussian on $E$ and $\sH := \pi(\sbH)$ is the Cameron--Martin space associated to $\mu$,
    
    \item an $\sH$-skeleton lift $\fL: \sH \rightarrow \bE$ which is a left inverse of $\pi\vert_{\sbH}$ (that is, $\fL \circ \pi\vert_{\sbH} = \id_{\sbH}$ and thus by (3) $\sbH = \fL(\sH)$), simply called \textbf{skeleton lift}, 
    
    \item a $\mu$-a.s.\! equivalence class represented by a measurable lift $\hat{\fL} \in \calP^{(\leq [\calT])}(E, \mu; \bE)$, called \textbf{full lift}, s.t. $\hat{\fL}_{\ast} \mu = \bmu$ and $\overline{\hat{\fL}_{\tau}} = \fL_{\tau}$ for every $\tau \in \calT$, where $\calP^{(\leq [\calT])}(E, \mu; \bE)$ denotes the $\calT$-inhomogeneous Wiener--Ito chaos in the sense of Definition \ref{defn:graded_WIC} and $\overline{(\cdot)}$ denotes the proxy-restriction in the sense of Definition \ref{defn:proxy_restriction}.
    \end{enumerate}
\end{defn*}

See Definition \ref{defn:AWMS} for further explanation and Figure \ref{fig:AWMS_first_appearence} for a display of the involved data: 

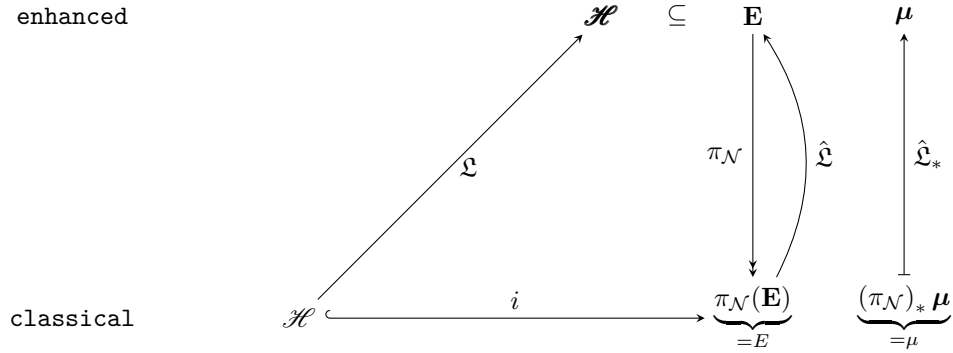
\begin{figure}[ht]
\begin{center}
        \begin{tikzpicture}
        \path
        (-3,0) node(1b) {\texttt{classical}} 
        (-3,4) node(1t) {\texttt{enhanced}}
        
        (4,4) node(1ti) {$\sbH$} 
        (5,4) node(1tsub) {$\subseteq$} 
        (6,4) node(2t) {$\bE$} 
        (8,4) node(3t) {$\bmu$} 
        (0,0) node(1b) {$\sH$} 
        (6,0) node(2b) {$\underbrace{\pi_{\calN}(\bE)}_{= E}$} 
        (8,0) node(3b) {$\underbrace{\left(\pi_{\calN}\right)_{\ast}\bmu}_{= \mu}$};
        
        \draw[right hook->] (1b) to node[above] {$i$} (2b); 
        \draw[->] (1b) to node[right] {$\fL$} (1ti); 
    
        \draw[->] (2b) to [bend right] node[right] {$\hat{\fL}$} (2t);
        \draw[->>] (2t) to node[left] {$\pi_{\calN}$} (2b);
        
        \draw[|->] (3b) to node[right] {$\hat{\fL}_{\ast}$} (3t);
        \end{tikzpicture}
        \caption{Diagram of the definition of AWMS.}
        \label{fig:AWMS_first_appearence}
\end{center}
\end{figure}

The rest of the present section will be spent introducing and motivating the components of the above definition and some technical machinery needed in later sections. Throughout the section we will keep a running example of the simplest nontrivial cases of the theory: $\bbR^d$-valued Ito rough paths of Brownian motion on an interval. 

\subsection{Classical Setup} \label{subsec:enhanced_Gaussian}

The definition of abstract Wiener spaces we use in this paper is the following: 

\begin{defn}[{Abstract Wiener Space, e.g. \cite[Chap. 3.4]{deuschelLargeDeviations1989}}] \label{defn:AWS}
An \textbf{abstract Wiener space} is a quadruple $(E, \sH, i, \mu)$ consisting of 

\begin{enumerate}[(1)]
    \item $E$, a separable (real) Banach space, 
    \item $\sH$, a separable (real) Hilbert space (called the \textbf{Cameron--Martin space}),
    \item $i: \sH \rightarrow E$, a continuous, linear injection\footnote{As is common, we will often tacitly make the identification $i(\sH) \cong \sH$ and/or state $i$ implicitly.},
    \item $\mu$, a (necessarily centred Gaussian) probability measure on $(E, \sB_{E})$ s.t. its characteristic functional $\check{\mu}$ has the form

    \begin{equation}
        \check{\mu}(\ell) := \int_E \exp \left( \rmi \ell(x) \right) \dd \mu(x) = \exp \left( - \frac{1}{2} \Vert i^{\ast}(\ell) \Vert_{\sH}^2 \right), \quad \ell \in E^{\ast},
    \end{equation}    
\end{enumerate}

where $\sB_{E} \label{symb:Borel_sigma}$ denotes the Borel $\sigma$-algebra on $E$ and $i^{\ast}: E^{\ast} \rightarrow \sH^{\ast} \cong \sH$ denotes the adjoint of the injection $i$, i.e. $[i^{\ast}(\ell)](h) = \ell(i(h))$ for every $\ell \in E^{\ast}, h \in \sH$.
\end{defn}

\begin{rem}
    In particular, $\mu$ is a Gaussian Borel measure on $E$ with topological support $\mathrm{supp} \mu = \overline{\sH}^{\Vert \cdot \Vert_E}$, the topological closure of $\sH$ in the topology induced by $E$. In the literature\footnote{E.g. \cite[Def. 3.9.4.]{bogachevGaussianMeasures1998}.} it is often assumed that $\mu$ has full support, so that $\overline{\sH}^{\Vert \cdot \Vert_E} = E$. In the context of abstract Wiener spaces this entails no loss of generality, as one can replace $E$ by $\overline{\sH}^{\Vert \cdot \Vert_E}$ when necessary. Once we talk about the law of enhanced Gaussian processes/models though, such a full support assumption is not reasonable, as the algebraic relations (Chen's relation/ $\Pi$-$\Gamma$-relations) imposed by rough paths and regularity structures typically force the lifted process/random field to live on non-linear subvarieties of $\bE$. We thus do \textit{not} assume full support for $\mu$, with the additional advantage that the same Banach space can serve as the underlying space for different Gaussian measures.\footnote{E.g. let $\mu$ be the distribution of a Brownian bridge on $[0,1]$, then we want to allow $E = C[0,1]$ instead of requiring $E=\{x \in C[0,1]: 0 = x(0) = x(1)\}$.} 
\end{rem}

\begin{exmp}[Running Example]
    The Gaussian measure $\mu$ for our running example is the classical Wiener measure; i.e. the distribution of Brownian motion on the interval $[0,1]$. The associated Cameron--Martin space $\sH$ can be identified with $W^{1,2}([0,1]; \bbR^d)$, the space of absolutely continuous $\bbR^d$-valued paths with square integrable derivative and inner product $\langle x,y \rangle := \int_0^1 \langle x'(t), y'(t)\rangle \dd t$. For the Banach space $E$ on which the measure is supported there are multiple options, among them $C([0,1]; \bbR^d)$, the space of continuous functions on the interval together with the $\sup$-norm; $\calC^{0, \alpha}([0,1]; \bbR^d)$ with $0 < \alpha < 1/2$, the closure w.r.t. the $\alpha$-Hölder norm of the set of smooth functions on $[0,1]$; $\calC^{0, 1-\var}([0,1]; \bbR^d)$, the closure w.r.t. the $1$-variation norm of the set of smooth functions on $[0,1]$. For definiteness we choose $E = \calC^{0, \alpha}([0,1]; \bbR^d)$ here and stick with it throughout the running example. The injection $i$ is simply the inclusion.
\end{exmp}
 
\subsection{Ambient Space} \label{subsec:banach_setup}

Let us start by describing $\bE$ in Figure \ref{fig:AWMS_first_appearence}.

\begin{defn}[Ambient Space] \label{defn:ambient_space}
    An \textbf{ambient space} $(\calT, \bE, [\cdot], \calN)$ consists of 

    \begin{enumerate}[(1)]
        \item a finite set $\calT$, whose elements we call \textbf{symbols},
        \item a separable Banach space $\bE := \bigoplus_{\tau \in \calT} E_{\tau}$, graded on the set $\calT$,
        \item a function $[\cdot]: \calT \rightarrow \bbN_{\geq 1}$ called \textbf{degree}, and
        \item a distinguished subset $\calN \subseteq \calT$ with $[\tau] = 1$ for every $\tau \in \calN$. 
    \end{enumerate}
\end{defn}

The finite set $\calT$ should be thought of as a generating set for the model space\footnote{Not to be confused with the space of models.} of a regularity structure, the distinguished subset $\calN$ as the set of symbols associated to the unlifted noise, and the degree $[\tau]$ of a symbol $\tau \in \calT$ as the number of appearances of the noise in that symbol - see Subsection \ref{subsec:Phi}. Define the sets 

\begin{equation}
[\calT] := \left\{ [\tau] : \tau \in \calT \right\} \quad \text{and} \quad \calT^{(k)} := \left\{ \tau \in \calT : [\tau] = k \right\}, \quad k \geq 1 ,
\end{equation}

and the Banach spaces

\begin{equation}
     E^{(k)} := \bigoplus_{\tau \in \calT^{(k)}} E_{\tau}, \quad \text{giving} \quad \bE = \bigoplus_{\tau \in \calT} E_{\tau} = \bigoplus_{k = 1}^{\max [\calT]} E^{(k)} \label{symb:k_component_of_bE}.
\end{equation}

The projections onto the constituent subspaces are denoted by 

\begin{equation}
    \pi_{\tau}: \mathbf{E} \rightarrow E_{\tau} \quad \text{and} \quad \pi^{(k)}: \mathbf{E} \rightarrow E^{(k)}, \quad \tau \in \calT, k \geq 1 .
    \end{equation}

The space $\bigoplus_{\tau \in \calN} E_{\tau} \label{symb:E1}$ and its associated projection $\sum_{\tau \in \calN} \pi_{\tau}$ will play a distinguished role in the subsequent material and will therefore simply be denoted 

\begin{equation}
    E := E_{\calN} = \bigoplus_{\tau \in \calN} E_{\tau}, \quad \text{and} \quad \pi := \pi_{\calN} = \sum_{\tau \in \calN} \pi_{\tau}. \label{symb:pi}
\end{equation}

We will generally use the shorthand $\bx_{\tau} := \pi_{\tau}(\bx)$ and $\bx^{(k)} := \pi^{(k)}(\bx)$. The spaces $\bE$ and $E^{(k)}$ canonically inherit a (separable) Banach space structure from their summands given by 

\begin{equation}
    \Vert \bx \Vert_{\bE} = \sum_{\tau \in \calT} \Vert \bx_{\tau} \Vert_{E_{\tau}} 
   = \sum_{k \geq 1} \Vert \bx^{(k)}  \Vert_{E^{(k)}}, \quad \bx \in \bE . \label{symb:norm_on_bE}
\end{equation}

Scalar multiplication on $E$, that is $\mathrm{m}_{\varepsilon}: x \mapsto \varepsilon x \label{symb:m_scaling}$, is extended to \textbf{dilation} on $\bE$, by setting

\begin{equation}
    \delta_{\varepsilon}: \bx \mapsto \sum_{\tau \in \calT} \varepsilon^{[\tau]} \bx_{\tau}, \quad \bx \in \bE, \varepsilon \geq 0. \label{eq:definition_of_dilation}
\end{equation}

For $\bx, \by \in \bE$, the Banach distance $\Vert \bx - \by \Vert_{\bE}$ is (locally uniformly\footnote{In the sense that the identity map $\id: (\bE, \Vert \cdot \Vert_{\bE}) \rightarrow (\bE, \vertiii{\cdot})$ and its inverse are locally uniformly continuous.}) equivalent to the \textbf{homogeneous distance $\vertiii{\bx - \by}$} on $\bE$, induced by 

\begin{equation}
\vertiii{\bx} := \sum_{\tau \in \calT} \Vert \bx_{\tau} \Vert_{E_{\tau}}^{\frac{1}{[\tau]}}, \quad \bx \in \bE, \label{symb:homogeneous_distance}
\end{equation}

which (since $0 \leq \frac{1}{[\tau]} \leq 1$ for every $\tau \in \calT$) is also a metric on $\bE$. The raison d'être of the homogeneous distance is its compatibility with the dilation operator in the sense that

\begin{equation}
\vertiii{\delta_{\varepsilon} \bx} =  \varepsilon \vertiii{\bx}, \quad \varepsilon \geq 0 . \label{eq:homogeneity_of_hom_norm}
\end{equation}

\begin{exmp}[Running Example continued]
    Continuing our running example, the set of symbols is given by the union of $\calT^{(1)} := \{1, \ldots, d \}$ (which we identify with $\calN$) and $\calT^{(2)} := \{ ij : 1 \leq i,j \leq d \}$; i.e. the set of indices of the components of the vector valued Brownian motion and of the matrix-valued Ito lift. The graded Banach space $\bE$ has as components $d$ many copies of $\calC^{0, \alpha}([0,1]; \bbR)$ and $d^2$ many copies of $\calC^{0, 2 \alpha}([0,1]^2; \bbR)$, which are closures of smooth functions in Hölder type topologies (see Subsection \ref{subsec:Ito} for their precise definition). Furthermore, we have $[\tau] = 1$ for every $\tau \in \calT^{(1)}$ and $[\tau] = 2$ for every $\calT^{(2)}$. In general it is not the case that $\calT^{(1)} = \calN$ (see e.g. Subsection \ref{subsec:Phi}). In Definition \ref{defn:AWMS} we will see that $E_{\calN}$ can be identified with the Banach space $E$ of the underlying abstract Wiener space, so that $E$ can be thought of as being contained inside of $\bE$ (in the same way that a base space of a vector bundle can be thought of as being contained in the total space). \\
    
    At this point, there is considerable freedom in choosing the involved data since there are no compatibility conditions yet. These will be introduced in Definition \ref{defn:AWMS}. For example, the degree function $[\cdot]$ is significantly restricted by (5) of Definition \ref{defn:AWMS}. Also, we could combine all symbols in $\calN$ into a single symbol, which would correspond to viewing $\calC^{0, \alpha}([0,1]; \bbR^d)$ as a single space instead of as decomposed into a direct sum $\oplus_{0 \leq i \leq d} \calC^{0, \alpha}([0,1]; \bbR)$.
 \end{exmp}

    \subsection{Lifting and Approaches to an Enhanced Cameron--Martin Space} \label{subsec:lifting}

The connection between the classical and the enhanced level is the rough path/model lift. Abstractly we consider the following: 

\begin{defn}[Lift] \label{defn:lift_defn}
    Let $(\calT, \bE, [\cdot], \calN)$ be an ambient space, let $A \subseteq E_{\calN}$ be a subset, and let $f: A \rightarrow \bE$ be a function. We say that $f$ is a \textbf{lift} if 

    \begin{equation}
       \pi_{\calN} \circ f = \id_{A},  \label{eq:lift_defn}
    \end{equation}

    i.e. if $f$ is a right-inverse of $\pi_{\calN} \vert_{f(A)}$. If $f: (A, \sB_A^{\mu}) \rightarrow (\bE, \sB_{\bE})$ is the representative of a $\mu$-a.s.\! equivalence class\footnote{Here and below, $\sB$ denotes the Borel $\sigma$-algebra of some topological space, indicated as subscript, possibly completed with respect to some measure $\mu$, in which case this is indicated as superscript.} for some probability measure $\mu$ on $A$, then \eqref{eq:lift_defn} is only required to hold $\mu$-a.s.\!  
\end{defn} 

Consider a measurable, not necessarily continuous lift from $E$ to $\bE$ in the sense of the above definition; that is, a measurable map $\hat{\fL}: E \rightarrow \bE$ such that $\pi \circ \hat{\fL} = \id_E$ $\mu$-a.s.\! For example, all rough path lifts of a $d$-dimensional Brownian motion ($d \geq 2$), discussed in \cite{frizCourseRoughPaths2020} (It\^o, Stratonovich and ``magnetic''), yield different examples of such a lift as those lifts are well-known to be discontinuous, but measurable and only defined up to $\mu$-a.s.\! equivalence. We are interested in the measure space $(\bE, \bmu)$, where $\bmu :=  \hat{\fL}_{\ast} \mu$ is the distribution of the lift $\hat{\fL}$ w.r.t. $\mu$ and a functional analytic object $\sbH$, that controls the behaviour of $\bmu$. Firstly, we note that $\bmu$ only depends on the $\mu$-equivalence class of $\hat{\fL}$ and thus the same should be true for $\sbH$. Motivated from examples, naively at least, one would like to define $\sbH$ as a subset of $\bE$ such that $\sbH := \hat{\fL} \vert_{\sH}(\sH)$, where $\hat{\fL}|_{\sH}$ denotes ``the restriction'' of the lifting map $\hat{\fL}$ to the Cameron--Martin space $\sH$. However, as is well-known\footnote{See e.g. \cite[Thm. 2.4.7]{bogachevGaussianMeasures1998}.} $\mu(\sH) = 0$, whenever $\dim(\sH) = \infty$, rendering $\sbH$ meaningless in the hinted generality ($\hat{\fL}$ defined up to $\mu$-a.s.\! equivalence.). Therefore, a separate definition of ``the restriction of $\hat{\fL}$ to $\sH$'' will be needed. 

\begin{rem}[Aida--Kusuoka--Stroock] \label{rem:AKS}
The question when some abstract measurable map defined on $E$ admits a meaningful restriction to $\sH$ goes back at least to \cite{aidaSupportWienerFunctionals1993} where the authors introduce the notion of $\mathcal{K}$-regularity. Put in our context, given a measurable $\hat{\fL}: E \rightarrow B$, where $E$ is part of an abstract Wiener space $(E, \sH, i, \mu)$ and $B$ is Polish, they postulate (c.f. \cite[Cor. 1.13]{aidaSupportWienerFunctionals1993}) the existence of a continuous map $\fL: \sH \to B$, s.t. 

\begin{equation}
    \fL \left( \overline P_n (x) \right) \rightarrow \hat{\fL} (x) \quad \text{and} \quad \hat{\fL}\left( \left( \id - \overline{P}_n \right) x + P_n h \right) \rightarrow \fL (h), \quad h \in \sH, \label{eq:AKS}
\end{equation}

in probability w.r.t. $\mu$, which leads to $\supp \hat{\fL}_{\ast} \mu \subseteq \overline{{\fL}(\sH)}^{B}$, $\supp \hat{\fL}_{\ast} \mu \supseteq \overline{{\fL}(\sH)}^{B}$, and hence an abstract support theorem (put in our context) of the form

\begin{equation}
    \supp \hat{\fL}_{\ast} = \overline{{\fL}(\sH)}^{B}. \label{eq:AKS_support}
\end{equation}

(Leaving details of notation to that paper, $P_n(h)$ is basically the projection of $h \in \sH$ to the subspace spanned by the first $n$ basis vectors of some orthonormal basis (ONB) of $\sH$ and $\overline{P}_n$ is its extension to $E$.) We note a recent application via rough paths by Y. Inahama \cite{inahamaSupportTheoremPinned2024}. Note that \eqref{eq:AKS_support} does in general not hold for singular SPDEs; see e.g. \cite{choukSupportTheoremSingular2018} for the generalized parabolic Anderson model (gPAM), \cite{tsatsoulisSpectralGapStochastic2018} for the $\Phi^4_2$-equation, and \cite{hairerSupportSingularStochastic2022} for more general results. 
\end{rem} 

There are two ways of going about ``the restriction of $\hat{\fL}$ to $\sH$'', which will turn out to be (in some sense) consistent: 

\begin{itemize}
    \item (Bottom-Up) One is to utilize the stronger topology of $\sH \subseteq E$, \textit{start} from a continuous lift $\fL$ defined on the subspace $\sH \subseteq E$, and to postulate\footnote{Cf. Remark on p. \pageref{rem:AKS} and \cite{aidaSupportWienerFunctionals1993}.} the existence of a limit $\hat{\fL}$ in probability w.r.t. $\mu$:

    \begin{equation}
        \lim_{m \rightarrow \infty} \fL \circ \Phi_m =: \hat{\fL},
    \end{equation}

    for some suitable approximation scheme $\Phi_m: E \rightarrow \sH$, consisting of bounded linear operators; the prime example being $\Phi_m =  \overline{P}_m$ taken from the left-hand side of \eqref{eq:AKS}, sometimes called Karhunen--Loève approximation\footnote{Other names include spectral-Galerkin approximation or $L^2$-approximation.}, but we also wish to account for piecewise-linear and mollifier approximations, ubiquitous in rough paths and regularity structures.\footnote{See Subsection \ref{subsec:intermediate_space}.}

    By virtue of continuity, if $(\Phi_m)_{m \in \bbN}$ approximates the identity pointwise on $\sH$,\footnote{That is, $\Phi_m(h) \to h$ for all $h \in \sH$. This is plainly the case when $\Phi_m = \overline{P}_m$.} the above limit exists on $\sH$, so that the restriction of $\hat{\fL}$ to $\sH$ is well-defined,\footnote{In the sense that any continuous function which is a representative of a $\mu$-a.s.\! equivalence class has the same restriction to $\sH$.} and agrees with $\fL$. This is close to the strategy pursued in \cite{frizLargeDeviationPrinciple2007} for deriving large deviation principles for Gaussian rough paths with the help of Banach-valued Wiener--Ito chaos, therein defined as the canonical lift $\fL$ of some mollified Gaussian process. 
        
    \item (Top-Down) The other approach is to start from a $\mu$-a.s.\! version of a lift $\hat{\fL}$ defined on all of $E$, make additional structural assumptions about $\hat{\fL}$ and extract a proxy-restriction\footnote{See Definition \ref{defn:proxy_restriction}.} $\fL$ of $\hat{\fL}$ to $\sH$. We want to emphasise here that there is no canonical way of defining $\fL$, so a choice has to be made. This is the strategy pursued in \cite{hairerLargeDeviationsWhiteNoise2015}, building upon ideas going back to \cite{borellTailProbabilitiesGauss1978} and \cite{ledouxNoteLargeDeviations1990}.
\end{itemize}
 
\begin{qn}
\begin{center}
\textbf{Two options:}
\end{center}

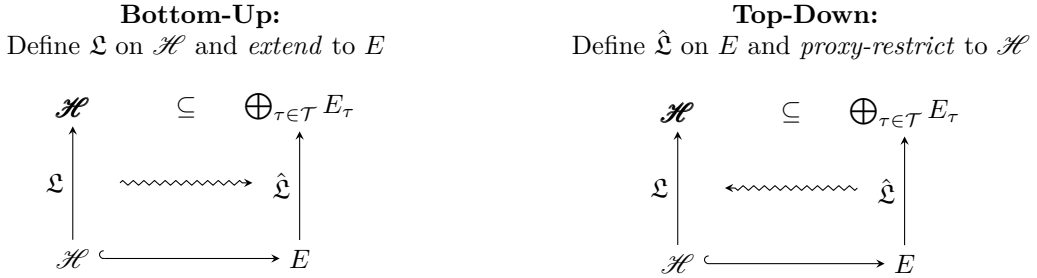
\begin{minipage}[t]{0.5\textwidth}
\begin{center}
\textbf{Bottom-Up:} \\
Define $\fL$ on $\sH$ and \textit{extend} to $E$ %% inspired by Leonard Gross 
\vspace{10pt}

\begin{tikzpicture}

\path
(0,0) node(1) {$\mathscr{H}$} 
(3,0) node(2) {$E$}
(0,2) node(1a) {$\sbH$} 
(3,2) node(2a) {$\bigoplus_{\tau \in \calT} E_{\tau}$}
(1.5,2) node(sub) {$\subseteq$} 
(0.5,1) node(a) {} 
(2.5,1) node(b) {};

\draw[right hook->] (1) to node[above] {} (2);
\draw[->] (1) to node[left] {$\fL$} (1a);
\draw[->] (2) to node[left] {$\hat{\fL}$} (2a);
\draw[->,
line join=round,
decorate, decoration={
    zigzag,
    segment length=4,
    amplitude=.9,post=lineto,
    post length=2pt
}] (a) to node[left] {} (b);
\end{tikzpicture}
\end{center}
\end{minipage}%
\begin{minipage}[t]{0.5\textwidth}
\begin{center}
\textbf{Top-Down:} \\
Define $\hat{\fL}$ on $E$ and \textit{proxy-restrict} to $\sH$ %% inspired by Huroyuki Sato 
\vspace{10pt}

\begin{tikzpicture}

\path
(0,0) node(1) {$\mathscr{H}$} 
(3,0) node(2) {$E$}
(0,2) node(1a) {$\sbH$} 
(3,2) node(2a) {$\bigoplus_{\tau \in \calT} E_{\tau}$}
(1.5,2) node(sub) {$\subseteq$} 
(0.5,1) node(a) {} 
(2.5,1) node(b) {};

\draw[right hook->] (1) to node[above] {} (2);
\draw[->] (1) to node[left] {$\fL$} (1a);
\draw[->] (2) to node[left] {$\hat{\fL}$} (2a);
\draw[->,
line join=round,
decorate, decoration={
    zigzag,
    segment length=4,
    amplitude=.9,post=lineto,
    post length=2pt
}] (b) to node[left] {} (a);
\end{tikzpicture}
\end{center}
\end{minipage}
\end{qn}

\vspace{10pt}

\begin{rem}[Renormalized Bottom-Up] Following \cite{hairerLargeDeviationsWhiteNoise2015}, the Top-Down approach applies to singular SPDEs, whereas the Bottom-Up approach is closer to the existing literature on Gaussian rough paths, starting with \cite{frizLargeDeviationPrinciple2007}. That said, following \cite[Chap. 10]{hairerTheoryRegularityStructures2014}, the natural construction of Gaussian models amounts to having conditions that give convergence in probability w.r.t. $\mu$ of

\begin{equation}
    \lim_{m \rightarrow \infty} \texttt{Ren}^m \circ \fL \circ \Phi_m =: \hat{\fL},
\end{equation}

where $\texttt{Ren}^m$ expresses the action of an element in some renormalization group, $\texttt{Ren}^m \in \calG$ \label{symb:renormalization_group}. It is an important feature of M. Hairer's theory (see \cite{brunedRoughPathPerspective2019} for a rough path perspective) that this group $\calG$ is a finite-dimensional Lie group, essentially due to a stationarity assumption of the underlying noise that is to be preserved by renormalization. It is conceivable that such features can be incorporated in an abstract setup like the one proposed here, e.g. via a group of measure-preserving transformation on $(\bE, \bmu)$.\footnote{Recall a commonly used abstract viewpoint in the analysis of stationary sequences. Instead of processes with a shift-invariant law, one considers a measure space $(\Omega, \calF)$ with $T$-invariant measures, for some measurable transformation $T$, with a measurable inverse $T^{-1}$. In this case the group of transformation is simply $(T^n : n \in \mathbb{Z}) \cong ( \mathbb{Z}, +)$.} Any such investigation should start however by intersecting stationarity with abstract Wiener spaces $(E, \sH, i, \mu)$; we are unaware of a reference but would be surprised if this had not been attempted yet. 
\end{rem}

\subsection{Intermediate Spaces and Skeleton Lifts} \label{subsec:intermediate_space}

The following definition of an intermediate space $\sK$ does not appear in Figure \ref{fig:AWMS_first_appearence} and its significant may not be immediately apparent. The reason for introducing it is to accommodate for the following kind of situation: Let $\sH$ be the Cameron--Martin space of a two-sided Brownian motion $B$ restricted to $[-1,1]$. Then $\sH$ consists precisely of functions $h(t) \equiv  \int_0^t h'(s) \dd s$, with $h' \in L^2([-1,1])$, and in particular $\sH$ is contained in $\left\{ x: [-1,1] \rightarrow \bbR \big\vert \text{cont. and} ~ x(0) = 0 \right\}$. Now consider the distribution of this process as a measure on $E = C([-1,1]; \bbR)$ and a natural approximation to be the piecewise linear approximation $(\Phi^{Q}_m)_{m \in \bbN}$, subordinate to some sequence of partitions $Q = (Q_m)_{m \in \bbN}$. If the partitions happen to not include the point $0$, then in general $\Phi^{Q}_m(E) \nsubseteq \sH$. On the other hand, for $\sK := \calC^{0,1-\var}([-1,1]; \bbR)$ we indeed have $\Phi^{Q}_m(E) \subseteq \sK$, regardless of which points are included in the partitions, and $i(\sH) \subseteq \sK$ and $(\Phi^{Q}_m)_{m \in \bbN}$ are bounded and linear. The situation is similar for convolution with a mollifier as smearing out the values near $0$ will a.s.\! lead to a non-zero value at the origin. We thus detangle the Cameron--Martin space $\sH$ from some other nice space $\sK$, in which the approximations take values and on which the skeleton lift is defined. 

Upon first reading it can be useful to pretend that $\sK \equiv \sH$ (similar to how it can be useful, but incorrect, to pretend that the Dirac distribution is a function).

\begin{defn}[Intermediate Space] \label{defn:intermediate_space}
    Let $(E, \sH, i, \mu)$ be an abstract Wiener space. An \textbf{intermediate space} is a separable Banach space $(\sK, \Vert \cdot \Vert_{\sK})$ contained in $E$ such that $i(\sH) \subseteq \sK$ and

    \begin{equation}
        \sH \xhookrightarrow{i} \sK \subseteq E,
    \end{equation}

    is $\sH-\sK$-continuous. An intermediate space is called \textbf{compact} if the linear injection $i: \sH \hookrightarrow \sK$ is compact (in the sense of bounded linear operators).
\end{defn}

The second part of the Definition \ref{defn:intermediate_space} (also not present in Figure \ref{fig:AWMS_first_appearence}), is that of an adaptation of the lift from $\sH$ to the larger space $\sK$. In the same vein as above, upon first reading it can be useful to pretend that $\fM \equiv \fL$ (and $\sK \equiv \sH$). The importance of both the intermediate space $\sK$ and the associated $\sK$-skeleton lift will become apparent in Theorem \ref{thm:bottom_up_construction}.

\begin{defn}[$\sK$-skeleton lift]\label{defn:skeleton_lift}
    Let $(\calT, \bE, [\cdot], \calN)$ be an ambient space, let $(E, \sH, i, \mu)$ be an abstract Wiener space and let $\sK$ be an intermediate space. A lift $\fM: \sK \rightarrow \bE$ is called \textbf{$\sK$-skeleton lift} if it satisfies the following properties:

    \begin{enumerate}
        \item $\fM: \sK \subseteq E \rightarrow \bE$ is continuous w.r.t. the topology on $\sK$.   
        
        \item $\fM: \sK \subseteq E \rightarrow \bE$ is \textbf{$\calT$-multi-linear} in the following sense: For $\tau \in \calT$, let $\sK^{\otimes_A [\tau]}$ denote the $[\tau]$-fold algebraic tensor product \label{symb:algebraic_tensor}of $\sK$ with itself and let $(\cdot)^{\otimes [\tau]}$ denote the canonical inclusion of $\sK$ into $\sK^{\otimes_A [\tau]}$ via $[\tau]$-fold tensor powering. Then $\fM$ is called \textbf{$\calT$-multi-linear} if for every $\tau \in \calT$ there exists a linear function 
        
            \begin{equation}
                \fM_{\tau}^{\otimes}: \sK^{\otimes_A [\tau]} \rightarrow E_{\tau}, \quad \text{s.t.} \quad \pi_{\tau} \circ \fM = \fM_{\tau}^{\otimes} \circ (\cdot)^{\otimes [\tau]}.
            \end{equation}
    \end{enumerate} 
    
    If $\sK = \sH$, we may refer to $\fM$ as just a skeleton lift.
\end{defn}

\begin{lem}[Homogeneity] \label{lem:homogeneity_of_fL}
    Let $(\calT, \bE, [\cdot], \calN)$ be an ambient space, let $(E, \sH, i, \mu)$ be an abstract Wiener space, let $\sK$ be an intermediate space, and let $\fM: \sK \rightarrow \bE$ be a $\sK$-skeleton lift. Then 

    \begin{equation}
        \fM \circ \rmm_{\varepsilon} = \delta_{\varepsilon} \circ \fM , \quad \varepsilon \geq 0 .
    \end{equation}
    
    We call this property \textbf{homogeneity} of $\fM$. 
\end{lem}
\begin{proof}
    Let $\varepsilon \geq 0$. Then for every $\tau \in \calT$ and $k \in \sK$ we have

    \begin{align}
        \pi_{\tau}(\fM(\rmm_{\varepsilon} k)) &= \pi_{\tau}(\fM(\varepsilon k)) =\fM^{\otimes}_{\tau}\left( (\varepsilon k)^{\otimes [\tau]}) \right) = \fM^{\otimes}_{\tau}(\varepsilon^{[\tau]} k^{\otimes [\tau]}) \\
        &= \varepsilon^{[\tau]}\fM^{\otimes}_{\tau}(k^{\otimes [\tau]}) = \varepsilon^{[\tau]} \pi_{\tau}(\fM(k)) = \pi_{\tau}(\delta_{\varepsilon} (\fM (k))) .
    \end{align}
\end{proof}

\subsection{Proxy-Restriction}

As expected from a theory dealing with Gaussian measures, the Wiener--Ito chaos decomposition will be an indispensable tool. Here we define a variant that takes into account the grading bestowed upon $\bE$. See Appendix \ref{sec:WIC} for an exposition of Banach valued and classical Wiener--Ito chaos.

\begin{defn}[Graded Wiener--Ito Chaos] \label{defn:graded_WIC}
    Let $(\calT, \bE, [\cdot], \calN)$ be an ambient space and let $(E, \sH, i, \mu)$ be an abstract Wiener space. Define the \textbf{$\calT$-inhomogeneous Wiener--Ito chaos} to be 

    \begin{equation}
            \calP^{(\leq [\calT])}(E, \mu; \bE) := \bigoplus_{\tau \in \calT} \calP^{(\leq [\tau])}(E, \mu; E_{\tau}) \subseteq L^2(E, \mu; \bE)
    \end{equation}

    and the \textbf{$\calT$-homogeneous Wiener--Ito chaos} to be

    \begin{equation}
            \calP^{([\calT])}(E, \mu; \bE) := \bigoplus_{\tau \in \calT} \calP^{([\tau])}(E, \mu; E_{\tau}) \subseteq L^2(E, \mu; \bE) ,
    \end{equation}

    where $\calP^{(\leq k)}(E, \mu; B)$ (resp. $\calP^{(k)}(E, \mu; B)$) denote the $k$-th inhomogeneous (resp. homogeneous) $B$-valued Wiener--Ito chaos. There is of course a natural projection

    \begin{equation}
        \Pi_{[\calT]}: \calP^{(\leq [\calT])}(E, \mu; \bE) \rightarrow \calP^{([\calT])}(E, \mu; \bE); \quad \Psi \mapsto \sum_{\tau \in \calT} \Pi_{\tau} \Psi_{\tau},
    \end{equation}

    where $\Pi_{\tau}: \calP^{(\leq [\tau])}(E, \mu; E_{\tau}) \rightarrow \calP^{([\tau])}(E, \mu; E_{\tau})$ is the natural projection onto the $[\tau]$-th homogeneous chaos.
\end{defn}

We now come to one of the key definitions of the theory, that of a proxy-restriction. It provides the main compatibility condition between the full lift $\hat{\fL}$, defined on all of $E$, and the skeleton lift, defined only on $\sH$ (see Definition \ref{defn:AWMS} (5)). In our setting, the proxy-restriction of an element in a finite Wiener--Ito chaos should be thought of as a substitute for ``the restriction of'' (the highest order part of) that element to the associated Cameron--Martin space. It is of course not \emph{the} restriction, as such a notion is is not well defined in this context (cf. the discussion in Subsection \ref{subsec:lifting}), but it provides a good enough substitute to produce the proofs of Theorem \ref{thm:LDP_for_AWMS}, Theorem \ref{thm:exp_int}, Theorem \ref{thm:lifted_CM_thm}.

\begin{defn}[Proxy-Restriction] \label{defn:proxy_restriction}
    Let $(\calT, \bE, [\cdot], \calN)$ be an ambient space, let $(E, \sH, i, \mu)$ be an abstract Wiener space, and let $\Psi \in \calP^{(\leq [\calT])}(E, \mu; \bE)$. For any $\tau \in \calT$ define the \textbf{proxy-restriction} of $\Psi_{\tau}$ to $\sH$ as a map $\overline{\Psi_{\tau}}: \sH \rightarrow E_{\tau}$, defined by\footnote{Since $\Psi$ is a Banach space-valued random variable, all integrals are to be understood as Bochner integrals.}

    \begin{equation}
        \overline{\Psi_{\tau}}(h) := \int_E \left( \Pi_{[\tau]} \Psi_{\tau} \right) \circ T_h \dd \mu, \quad h \in \sH, \label{eq:definition_of_bar_lift}
    \end{equation} 

    where $T_h$ is the classical shift operator $T_h(x) = x+h$ (see also Theorem \ref{thm:classical_CM}). Also define the notation  

    \begin{equation}
        \Psi^{\circ} = \sum_{\tau \in \calT} \underbrace{\Pi_{[\tau]} \Psi_{\tau}}_{=: \Psi_{\tau}^{\circ}}, \quad \overline{\Psi} := \sum_{\tau \in \calT} \overline{\Psi_{\tau}}.
    \end{equation}
\end{defn}

\begin{rem}
    In the literature on white noise analysis e.g. \cite[Sec. 3]{kondratevSpacesWhiteNoise1993} or \cite[Prop. 2.3.]{hidaWhiteNoiseInfinite1993}, Definition \ref{defn:proxy_restriction} is also known as the \textit{$S$-transform} of $\Pi_{[\tau]} \Psi_{\tau}$. In the context of large deviation principles, the concept was already used in \cite{hairerLargeDeviationsWhiteNoise2015} (under the name \textit{homogeneous part}) of $\Psi$ and earlier in \cite{ledouxNoteLargeDeviations1990} (without a dedicated name). Note that if $E_{\tau} = \bbR$, then the proxy-restriction is nothing but a projection onto the homogeneous Wiener--Ito chaos of degree $[\tau]$ followed by an application of the inverse of the Wiener--Ito isometry in the sense of \cite{nualartMalliavinCalculusRelated2006} (and an identification of the symmetric tensor power of $\sH$ with its dual space). 
\end{rem}

\begin{exmp}[Running example continued]
    An AWMS built for Ito Brownian motion includes two types of lift: a full lift $\hat{\fL}$, given by $B \mapsto (B, \int B \dd B)$, i.e. the Ito integral of Brownian motion against itself, and the $\sH$-skeleton lift $\fL$, which is given by the same map, but now interpreted in the sense of Young integration and defined on $\sH = W^{1,2}$. The significance of the proxy-restriction is that $\fL = \overline{\hat{\fL}}$ (see Proposition \ref{prop:proxy_restr_of_Ito}).
\end{exmp}

In the following we collect some basic properties of the proxy-restriction that will be used in the rest of the paper. 

\begin{prop}[Properties of the Proxy-Restriction] \label{prop:proxy_restriction_is_cont}
    Let $C(\sH, \bE)$ denote the space of continuous functions from $\sH$ to $\bE$. Let $(\calT, \bE, [\cdot], \calN)$ be an ambient space, let $(E, \sH, i, \mu)$ be an abstract Wiener space, and let $\Psi \in \calP^{(\leq [\calT])}(E, \mu; \bE)$. Then
    
    \begin{enumerate}[(1)]
        \item If $\pi_{\calN} \circ \Psi = \id_{E}$ $\mu$-a.s.\! (i.e. if $\Psi$ is a lift on $E$ in the sense of Definition \ref{defn:lift_defn}), then $\pi_{\calN} \circ \overline{\Psi} = \id_{\sH}$ (i.e. $\overline{\Psi}$ is a lift on $\sH$ in the sense of Definition \ref{defn:lift_defn}). 
        \item The proxy-restriction $\overline{\Psi}$ is a continuous function on $\sH$. That is, there is an assignment 
        
        \begin{equation}
        \overline{(\cdot)}: \calP^{(\leq [\calT])}(E, \mu; \bE) \rightarrow C(\sH, \bE), \quad \Psi \mapsto \overline{\Psi}. 
        \end{equation}

        \item The assignment $\overline{(\cdot)}$ is well-defined on $\mu$-a.s. equivalence classes and linear.
        \item For any $\Psi \in \calP^{(\leq [\calT])}(E, \mu; \bE)$ the proxy-restriction $\overline{\Psi}$ is $\calT$-multi-linear in the sense of Definition \ref{defn:skeleton_lift}.
    \end{enumerate}

    In particular, as a consequence of (1), (2), and (4), if $\Psi \in \calP^{(\leq [\calT])}(E, \mu; \bE)$ is a lift in the sense of Definition \ref{defn:lift_defn}, then $\overline{\Psi}$ is a $\sH$-skeleton lift in the sense of Definition \ref{defn:skeleton_lift}.
\end{prop}
\begin{proof}  
    (1) Let $h \in \sH$ be arbitrary. Then from 

    \begin{equation}
        \pi \circ \Psi^{\circ} = \pi \circ \Bigg( \sum_{\tau \in \calT} \Pi_{[\tau]} \Psi_{\tau} \Bigg) = \sum_{\substack{\tau \in \calT \\ [\tau] = 1}} \Pi_1 \Psi_{\tau} = \Pi_1 \Bigg( \sum_{\substack{\tau \in \calT \\ [\tau] = 1}} \Psi_{\tau} \Bigg) = \Pi_1 \left( \pi \circ \Psi \right)
    \end{equation}

    we deduce that

    \begin{align}
        \pi \left( \overline{\Psi}(h) \right) &= \pi \left( \bbE \left[ \Psi^{\circ} \circ T_h \right] \right) = \bbE \left[ \pi \circ \Psi^{\circ} \circ T_h \right] = \bbE \left[ \Pi_1 \left( \pi \circ \Psi \right) \circ T_h \right] = \int_E x+h ~ \dd \mu(x) = h,
    \end{align}

    where we made use of the fact that $\Pi_1( \id_E) = \id_E$ and of the assumption $\pi \circ \Psi = \id_{E}$ $\mu$-a.s.\! and therefore also $\mu_h$-a.s.\! by Theorem \ref{thm:classical_CM}. Hence $\pi \circ \overline{\Psi} = \id_{\sH}$. 
    
    (2) To see that the proxy-restriction is continuous, let $h_n \rightarrow h$ in $\sH$. Then 

    \begin{align}
        &\left\Vert \overline{\Psi}(h_n) - \overline{\Psi}(h) \right\Vert_{\bE} = \left\Vert \int_E \Psi^{\circ} \circ T_{h_n} \dd \mu - \int_E \Psi^{\circ} \circ T_{h} \dd \mu \right\Vert_{\bE} \\
        &= \left\Vert \int_E \Psi^{\circ} \cdot (f_{h_n} - f_{h}) \dd \mu \right\Vert_{\bE} \leq  \int_E \left\Vert \Psi^{\circ} \right\Vert_{\bE} \vert f_{h_n} - f_{h} \vert \dd \mu,
    \end{align}

    where

    \begin{equation}
        f_h(x) = \exp \left( \underline{h}(x) - \frac{1}{2} \Vert h \Vert_{\sH}^2 \right), \quad x \in E,
    \end{equation}

    denotes the Radon--Nikodým derivative of $(T_h)_{\ast} \mu$ w.r.t. $\mu$, and $\underline{h}$ denotes the image of $h$ under the identification of $\sH$ with the reproducing kernel Hilbert space of $\mu$ (see Appendix \ref{sec:AWS}). Applying Cauchy--Schwarz yields 

    \begin{equation}
        \left\Vert \overline{\Psi}(h_n) - \overline{\Psi}(h) \right\Vert_{\bE} \leq \underbrace{\left\Vert \Psi^{\circ} \right\Vert_{L^2(E, \mu; \bE)}}_{< \infty} \Vert f_{h_n} - f_{h} \Vert_{L^2(E, \mu; \bbR)}, 
    \end{equation}

    where the first term on the right-hand-side is finite since $\Psi \in L^2(E, \mu; \bE)$ by assumption. Using again the Cauchy--Schwarz inequality we can upper bound the square of the second term on the right-hand-side as 

    \begin{align}
        \int_E \vert f_{h_n} - f_{h} \vert^2 \dd \mu &= \int_E \left\vert \exp \left( \underline{h_n}(x) - \frac{1}{2} \Vert h_n \Vert_{\sH}^2 \right) - \exp \left( \underline{h}(x) - \frac{1}{2} \Vert h \Vert_{\sH}^2 \right) \right\vert^2 \dd \mu(x) \\
        &\leq \int_E \exp(2 \chi(n,x)) \left\vert \left( \underline{h_n} - \underline{h} \right) - \frac{1}{2} \left( \Vert h_n \Vert_{\sH}^2 - \Vert h \Vert_{\sH}^2 \right) \right\vert^2 \dd \mu(x) \\
        &\leq \underbrace{\bbE \left[ \exp(4 \chi(n,\cdot)) \right]^{\frac{1}{2}}}_{(\ast)} \left\Vert \left\vert \left( \underline{h_n} - \underline{h} \right) - \frac{1}{2} \left( \Vert h_n \Vert_{\sH}^2 - \Vert h \Vert_{\sH}^2 \right) \right\vert^2 \right\Vert_{L^2(E, \mu; \bbR)}, \label{eq:cont_of_shift_eq}
    \end{align}

    where we used $f(b) - f(a) = f'(\chi) (b - a)$ for some $\chi \in [a,b]$. In particular, for any $n \geq 0$ and $x \in E$

    \begin{align}
        \exp(4 \chi(n,x)) &\leq \max\big\{ \exp \left( 4 \underline{h_n}(x) - \Vert h_n \Vert_{\sH}^2 \right), \exp \left( 4 \underline{h}(x) - \Vert h \Vert_{\sH}^2 \right)\big\} \\
        &\leq \max\big\{ \exp \left( 4 \underline{h_n}(x) \right), \exp \left( 4 \underline{h}(x) \right)\big\}. \label{eq:upper_bound_on_chi}
    \end{align}

    Recall at this point that for any $h \in \sH$ the random variable $\underline{h} \in L^2(E, \mu; \bbR)$ is centred Gaussian with distribution $\sN(0, \Vert h \Vert_{\sH}^2)$ and that the moment generating function of $\underline{h}$ thus exists on all of $\bbR$ and has the form  

    \begin{equation}
        \bbE \left[ \exp(\lambda \underline{h}) \right] = \exp \left( \Vert h \Vert_{\sH}^2 \frac{\lambda^2}{2} \right).
    \end{equation}

    Hence, using \eqref{eq:upper_bound_on_chi}, $(\ast)$ can be upper bounded by 

    \begin{equation}
        \max \left\{ \exp \left( \Vert h_n \Vert_{\sH}^2 \frac{4^2}{2} \right), \exp \left( \Vert h \Vert_{\sH}^2 \frac{4^2}{2} \right)
        \right\}, 
    \end{equation}

    which is upper bounded uniformly in $n \in \bbN$ since $\sup_{n \in \bbN} \Vert h_n \Vert_{\sH} < \infty$. Finally, the second term goes to $0$ since it can be upper bounded by 

    \begin{align}
        & \left\Vert \vert \underline{h_n} - \underline{h} \vert^2 \right\Vert_{L^2(E, \mu; \bbR)} + \left\Vert \left\vert \frac{1}{2} \left( \Vert h_n \Vert_{\sH}^2 - \Vert h \Vert_{\sH}^2 \right) \right\vert^2 \right\Vert_{L^2(E, \mu; \bbR)} \\
        &= \left\Vert \underline{h_n} - \underline{h} \right\Vert^2_{L^4(E, \mu; \bbR)} + \big\vert \frac{1}{2} \underbrace{ \left( \Vert h_n \Vert_{\sH}^2 - \Vert h \Vert_{\sH}^2 \right)}_{\rightarrow 0} \big\vert^2.
    \end{align}

    Since $h_n \rightarrow h$ in $\sH \cong \sH^{\ast}$ and thus in $L^2(E, \mu; \bbR)$ and all elements lie in $\calP^{(\leq 1)}(E, \mu; \bbR)$, Lemma \ref{lem:equiv_of_norms} shows that the first term approaches $0$ as $n \rightarrow \infty$. Hence $\overline{\Psi}$ is continuous.

    (3) The linearity of the assignment is clear since it is a composition of linear operators. For the well-definedness, assume $\Psi = \Psi'$ $\mu$-a.s.\! Then since $h \in \sH$, the Cameron--Martin Theorem \ref{thm:classical_CM} guarantees that $\overline{\Psi} = \overline{\Psi'}$. 
    
    (4) To see the $\calT$-multi-linearity of $\overline{\Psi}$ let $\tau \in \calT$ be arbitrary. By the vector valued discrete martingale $L^p$-convergence theorem (Proposition \ref{prop:convergence_of_conditional_exp})

    \begin{equation}
        \sum_{\alpha \in A^{\leq [\tau]}_{m}} \bbE \left[ \Psi_{\tau} H_{\alpha} \right] H_{\alpha} = \bbE \left[ \Psi_{\tau} \Big\vert \calF_m \right] \rightarrow \Psi_{\tau}, \quad \text{in} ~ L^2(E, \mu; E_{\tau}),
    \end{equation}

    where $A^{\leq [\tau]}_{m}$ is as defined in \eqref{symb:Akn} on p.~\pageref{symb:Akn}, and $\calF_m$ is as defined in \eqref{eq:sigma_algebra} on p.~\pageref{eq:sigma_algebra}. The convergence thus also holds in $L^1(E, \mu_h; E_{\tau})$ by Proposition \ref{prop:L2L1}. Hence (denoting the mode of convergence by super-script for clarity) we have for any $h \in \sH$

    \begin{align}
        &\overline{\Psi}_{\tau} (h) = \bbE \left[ \left( \Pi_{[\tau]} \Psi_{\tau} \right) \circ T_h \right] = \bbE \left[ \left( \lim_{m \rightarrow \infty}^{L^2(E, \mu; E_{\tau})} \sum_{\alpha \in A^{[\tau]}_{m}} \bbE \left[ \Psi_{\tau} H_{\alpha} \right] H_{\alpha} \right) \circ T_h \right] \\
        &= \bbE \left[ \lim_{m \rightarrow \infty}^{L^1(E, \mu; E_{\tau})} \left( \sum_{\alpha \in A^{[\tau]}_{m}} \bbE \left[ \Psi_{\tau} H_{\alpha} \right] \left( H_{\alpha} \circ T_h \right) \right) \right] = \lim_{m \rightarrow \infty}^{E_{\tau}} \sum_{\alpha \in A^{[\tau]}_{m}} \underbrace{\bbE \left[ H_{\alpha} \circ T_h \right]}_{(\ast)} \bbE \left[ \Psi_{\tau} H_{\alpha} \right]. \label{eq:picking}
    \end{align}

    Now, fix $m \geq 0$ and $\alpha \in A^{[\tau]}_{m}$ and focus on $(\ast)$: applying the Binomial theorem for Hermite polynomials (see Proposition \ref{prop:hermite_polynomial_binomial}) to each factor $h_{\alpha_i}$ of $H_{\alpha}$ yields 
    
    \begin{equation}
        (\ast) = \bbE\left[ \prod_{i \in \bbN} \left( \sum_{l = 0}^{\alpha_{i}} \binom{\alpha_i}{l} h_l \left( \langle e_i, \cdot \rangle \right) \langle e_i, h \rangle^{\alpha_i - l} \right) \right]. 
    \end{equation}
    
    Write $S_{\alpha}:= \times_{i \in \bbN} \{0, \ldots, \alpha_i\}$ \label{page:prod_sum_switching} and note that $S_{\alpha}$ is finite. Switching the sum and the product, and pulling the sum out of the expectation, we obtain
    
    \begin{equation}
        = \sum_{\sigma \in S_{\alpha}} \bbE\left[ \prod_{i \in \bbN} {\alpha_i}{\sigma_i} h_{\sigma_i} \left( \langle e_i, \cdot \rangle \right) \langle e_i, h \rangle^{\alpha_i - \sigma_i} \right].
    \end{equation}
    
    Since the sequence $(\langle e_i,\cdot \rangle)_{i \in \bbN}$ is iid w.r.t. $\mu$ we may also pull the product out of the expectation to obtain
    
    \begin{equation}
        = \sum_{\sigma \in S_{\alpha}} \prod_{i \in \bbN} \binom{\alpha_i}{\sigma_i} \underbrace{\bbE\left[ h_{\sigma_i} \left( \langle e_i, \cdot \rangle \right) \right]}_{= 1_{\{\sigma_i = 0\}}} \langle e_i, h \rangle^{\alpha_i - \sigma_i} = \prod_{i \in \bbN} \langle e_i, h \rangle^{\alpha_i},
    \end{equation}

    where we used the fact that all Hermite polynomials of order $>0$ are centered, and those of order $0$ have expectation $1$. Thus, continuing from Equation \eqref{eq:picking},

    \begin{equation}
        \overline{\Psi}_{\tau} (h) = \lim_{m \rightarrow \infty}^{E_{\tau}} \sum_{\alpha \in A^{[\tau]}_{m}} \prod_{i \in \bbN} \langle e_i, h \rangle^{\alpha_i} \bbE \left[ \Psi_{\tau} H_{\alpha} \right].
    \end{equation}

    Since for a given $\alpha = (\alpha_1, \alpha_2, \ldots) \in A^{[\tau]}_{m}$ the map $h_1 \otimes \ldots \otimes h_{[\tau]} \mapsto \langle e_{i_1'}, h_1 \rangle \cdot \ldots \cdot \langle e_{i_{b}'}, h_{[\tau]} \rangle$, where the $i' \in \bbN$ are those indices such that $\alpha_{i'} \neq 0$ (counted with multiplicity, thus $b \leq [\tau]$), is a linear operator $\sH^{\otimes [\tau]} \rightarrow \bbR$ we may define 
    
    \begin{equation}
        \overline{\Psi}_{\tau}^{\otimes} := \lim_{m \rightarrow \infty}^{E_{\tau}} \sum_{\alpha \in A^{[\tau]}_{m}} \prod_{i \in \bbN} \langle e_i, \cdot \rangle^{\alpha_i} \bbE \left[ \Psi_{\tau} H_{\alpha} \right] .
    \end{equation}

    As a pointwise limit of sums and scalings of the linear operators, this again gives a linear operator $\sH^{\otimes [\tau]} \rightarrow E_{\tau}$.   
\end{proof}

\section{Definition and Constructions of Abstract Wiener Model Spaces} \label{sec:defn_and_construction}

\begin{defn}[Abstract Wiener Model Space] \label{defn:AWMS}
    An \textbf{abstract Wiener model space} is a quintuple $((\calT, \bE, [\cdot], \calN), \sbH, \bmu, \fL, \hat{\fL})$ consisting of the following data: 
    
    \begin{enumerate}[(1)]
    \item An ambient space $(\calT, \bE, [\cdot], \calN)$ (in the sense of Definition \ref{defn:ambient_space}),
    
    \item a subset $\sbH \subseteq \bE$, called the \textbf{enhanced Cameron--Martin space}, 
    
    \item a Borel probability measure $\bmu$ on $\bE$, called \textbf{enhanced measure}, such that $\mu := \pi_{\ast} \bmu$ is centred Gaussian on $E$ and $\sH := \pi(\sbH)$ is the Cameron--Martin space\footnote{Recall that we identify $\sH$ with $i(\sH)$.} associated to $\mu$,
    
    \item an $\sH$-skeleton lift $\fL: \sH \rightarrow \bE$ (in the sense of Definition \ref{defn:skeleton_lift}) which is a left inverse\footnote{That is, $\fL \circ \pi\vert_{\sbH} = \id_{\sbH}$ and thus by (3) $\sbH = \fL(\sH)$.} of $\pi\vert_{\sbH}$, in this context simply called \textbf{skeleton lift}, 
    
    \item a $\mu$-a.s.\! equivalence class represented by a measurable lift\footnote{Recall that since the lift is only assumed to be measurable and represent a $\mu$-a.s.\! equivalence class, the lifting property \eqref{eq:lift_defn} is only assumed to hold $\mu$-a.s.\!} $\hat{\fL} \in \calP^{(\leq [\calT])}(E, \mu; \bE)$, called \textbf{full lift}, s.t. $\hat{\fL}_{\ast} \mu = \bmu$ and
    
    \begin{equation}
        \overline{\hat{\fL}_{\tau}} = \fL_{\tau}, \quad \forall \tau \in \calT, \label{eq:proxy_restriction_equals_skeleton_lift}
    \end{equation}
    
    where $\calP^{(\leq [\calT])}(E, \mu; \bE)$ denotes the $\calT$-inhomogeneous Wiener--Ito chaos in the sense of Definition \ref{defn:graded_WIC} and $\overline{(\cdot)}$ denotes the proxy-restriction in the sense of Definition \ref{defn:proxy_restriction}. \\

    The abstract Wiener space $(E, (\sH, i), \mu)$ will be referred to as the \textbf{underlying AWS} of the AWMS. 
    \end{enumerate}
\end{defn}

Note that the above definition is not minimal in the sense that fixing $\hat{\fL}$ leaves only one choice for $\fL$, namely $\fL = \overline{\hat{\fL}}$. Indeed, this point will become very relevant in Subsections \ref{subsec:TD_constr} and \ref{subsec:bottom_up_construction}. The fact that the definition is not minimal is reminiscent of the classical theory, where fixing $E$ and $(\sH, i)$ leaves only one choice for $\mu$ (cf. discussion before Theorem \ref{thm:characterization_of_enhanced_measure} below). 

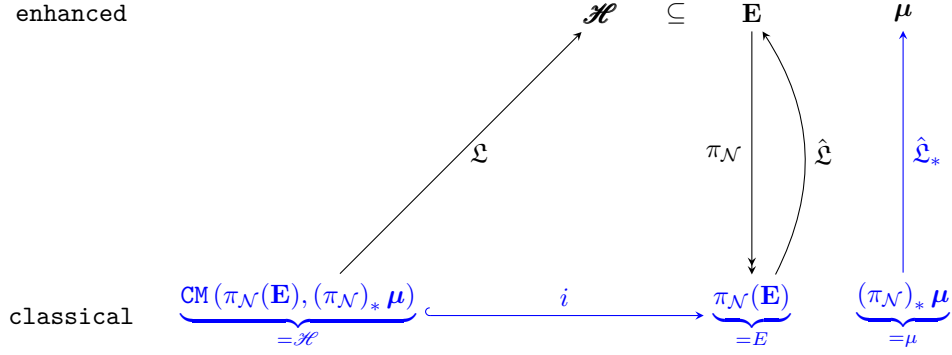
\begin{figure}[ht]
\begin{center}
        \begin{tikzpicture}
        \path
        (-3,0) node(1b) {\texttt{classical}} 
        (-3,4) node(1t) {\texttt{enhanced}}
        
        (4,4) node(1ti) {$\sbH$} 
        (5,4) node(1tsub) {$\subseteq$} 
        (6,4) node(2t) {$\bE$} 
        (8,4) node(3t) {$\bmu$} 
        (0,0) node(1b) {\textcolor{blue}{$\underbrace{\texttt{CM} \left( \pi_{\calN}(\bE), \left(\pi_{\calN}\right)_{\ast} \bmu \right)}_{= \sH}$}} 
        (6,0) node(2b) {\textcolor{blue}{$\underbrace{\pi_{\calN}(\bE)}_{= E}$}} 
        (8,0) node(3b) {\textcolor{blue}{$\underbrace{\left(\pi_{\calN}\right)_{\ast}\bmu}_{= \mu}$}};
        
        \draw[right hook->, blue] (1b) to node[above] {\textcolor{blue}{$i$}} (2b); 
        \draw[->] (1b) to node[right] {$\fL$} (1ti); 
    
        \draw[->] (2b) to [bend right] node[right] {$\hat{\fL}$} (2t);
        \draw[->>] (2t) to node[left] {$\pi_{\calN}$} (2b);
        
        \draw[->, blue] (3b) to node[right] {\textcolor{blue}{$\hat{\fL}_{\ast}$}} (3t);
        \end{tikzpicture}
        \caption{Diagram of the definition of AWMS. By $\texttt{CM}(E, \mu)$ 
        \label{symb:CM_space} we denote the Cameron--Martin space associated to $(E, \mu)$. The lower level corresponds to the data which belongs to the classical theory, while the upper level corresponds to data in the enhanced setting. The two lifts $\fL$ and $\hat{\fL}$ provide a connection between the two. Black symbols represent data which needs to be chosen in the definition an AWMS, while blue symbols can be defined from that choice. Arrows of the form $\twoheadrightarrow$ represent projections, while arrows of the form $\hookrightarrow$ represent inclusions.}
        \label{fig:AWMS}
\end{center}
\end{figure}

\begin{rem}[AWS as AWMS]
    As a first observation, note that every AWS provides an example of an AWMS by supplementing $(E, H, i, \mu)$ with the ambient space given by $\calT = \{ \ast \}, \quad \bE = E, \quad [\ast] = 1, \quad \calN = \calT$, the enhanced Cameron--Martin space $\sbH = i(\sH)$ and the enhanced measure, $\sH$-skeleton lift and full lift $\bmu = \mu$, $\fL = i$, $\hat{\fL} = \id_{E}$.
\end{rem}

Recall that if $(E_1, \sH_1, i_1, \mu_1)$ and $(E_2, \sH_2, i_2, \mu_2)$ are two abstract Wiener spaces such that 

\begin{equation}
    E_1 = E_2, \quad \text{and} \quad (\sH_1, i_1) = (\sH_2, i_2),
\end{equation}

then $\mu_1 = \mu_2$ on $E_1 = E_2$. One quick way of seeing this is noting that the characteristic functional $\check{\mu}_j$ of $\mu_j$ is determined by $(\sH_j, i_j)$ where $j = 1,2$: for every $\ell \in E_1^{\ast} = E_2^{\ast}$ 

\begin{align}
    \check{\mu}_1(\ell) = \exp \left( - \frac{1}{2} \Vert i_1^{\ast}(\ell) \Vert_{\sH_1}^2 \right) = \exp \left( - \frac{1}{2} \Vert i_2^{\ast}(\ell) \Vert_{\sH_2}^2 \right) = \check{\mu}_2(\ell) .
\end{align}

Since the characteristic functional characterizes a measure on an separable Banach space the conclusion follows. A similar statement is true for AWMS: 

\begin{thm} \label{thm:characterization_of_enhanced_measure}
    Let $((\calT_1, \bE_1, [\cdot]_1, \calN_1), \sbH_1, \bmu_1, \fL_1, \hat{\fL}_1)$, and $((\calT_2, \bE_2, [\cdot]_2, \calN_2), \sbH_2, \bmu_2$, $\fL_2, \hat{\fL}_2)$ be two abstract Wiener model spaces. If 

    \begin{equation}
        (\calT_1, \bE_1, [\cdot]_1, \calN_1) = (\calT_2, \bE_2, [\cdot]_2, \calN_2), \quad \sbH_1 = \sbH_2, \quad \text{and} \quad \hat{\fL}_1 = \hat{\fL}_2,
    \end{equation}

    then $\bmu_1 = \bmu_2$ on $\bE_1 = \bE_2$.
\end{thm}
\begin{proof}
    Let $j = 1,2$ and write $\pi_j := \pi_{\calN_j}$. By definition of the enhanced Cameron--Martin space and the enhanced measure, $\sH_j := \pi_j(\sbH_j)$ is the Cameron--Martin space of $\mu_j := (\pi_j)_{\ast} \bmu_j$. By the classical result above 

    \begin{equation}
        \mu_1 = (\pi_1)_{\ast} \bmu_1 = (\pi_2)_{\ast} \bmu_2 = \mu_2, \quad \text{on} ~ E_1 = E_2 .
    \end{equation}

    Hence, by the definition of the full lift and the enhanced measure

    \begin{equation}
        \bmu_1 = \left(\hat{\fL}_1 \right)_{\ast}(\mu_1) = \left(\hat{\fL}_2 \right)_{\ast}(\mu_2) = \bmu_2, \quad \text{on} ~ \bE_1 = \bE_2 .
    \end{equation}
\end{proof}

\begin{rem} \label{rem:replacing_condition}
    Replacing the condition $\hat{\fL}_1 = \hat{\fL}_2$ by $\fL_1 = \fL_2$ is not sufficient. For example, as will be shown in Subsection \ref{subsec:Ito}, the skeleton lifts of the Ito-enhancement and the Stratonovich-enhancement coincide, but their full lifts differ (by a bracket term) and therefore the enhanced measures also differ. 
\end{rem}

\subsection{Top-Down Construction} \label{subsec:TD_constr}

Recall that a classical AWS is over-determined in the sense that given $E$ and $\mu$ there is a unique (up to isometric isomorphism) choice of $\sH$ and $i$ such that $(E, \sH, i, \mu)$ is an AWS. Similarly, given $\sH$ and a choice of measurable norm (in the sense of \cite{grossAbstractWienerSpaces1967}) $\Vert \cdot \Vert_E$, there is a unique (up to precomposition with an isometric isomorphism of $\sH$) $i$ and a unique Borel probability measure $\mu$ on $E:=\overline{\sH}^{\Vert \cdot \Vert_E}$ such that $(E, \sH, i, \mu)$ is an AWS. In much the same way, an AWMS in the sense of Definition \ref{defn:AWMS} can be constructed from strictly less data than is required in the definition. 

Recall the discussion in the end of Subsection \ref{subsec:lifting}. The following construction is along the lines of the ``Top-Down'' philosophy alluded to at that point. It proceeds by assuming (in particular) the full lift as part of the given data and constructing other elements of Definition \ref{defn:AWMS} from it. 

\begin{thm}[Top-Down Construction] \label{thm:top_down_construction}
    Let $(\calT, \bE, [\cdot], \calN)$ be an ambient space, $\bmu$ a Borel probability measure on $\bE$ s.t. $\mu := \pi_{\ast} \bmu$ is centred Gaussian on $E := E_{\calN}$, a $\mu$-a.s.\! equivalence class represented by a measurable lift $\hat{\fL} \in \calP^{(\leq [\calT])}(E, \mu; \bE)$ s.t. $\hat{\fL}_{\ast} \mu = \bmu$. \\

    Then the following data constitutes an abstract Wiener model space which does not depend on the representative of the $\mu$-a.s.\! equivalence class of $\hat{\fL}$: 

    \begin{enumerate}[(1)]
        \item the ambient space $(\calT, \bE, [\cdot], \calN)$
        \item the subset $\sbH := \overline{\hat{\fL}} (\sH)$
        \item the enhanced measure $\bmu$
        \item the skeleton lift $\fL := \overline{\hat{\fL}}$
        \item the full lift $\hat{\fL}$
    \end{enumerate} 
    
We will refer to a triple $((\calT, \bE, [\cdot], \calN), \bmu, \hat{\fL})$ satisfying the above assumptions as \textbf{Top-Down data}.
The construction may be summarized by Figure \ref{fig:top_down}.
\end{thm}
\begin{proof}
    Since $\mu := \pi_{\ast} \bmu$ is a centred Gaussian measure on a separable Banach space $E$, there exists a separable Hilbert space $\sH := \texttt{CM}(E, \mu)$ and a linear injection $i: \sH \hookrightarrow E$ such that $\mu$ is an extension of the canonical cylinder measure associated to $\sH$ i.e. $(E, \sH, i, \mu)$ is an abstract Wiener space.\footnote{See e.g. \cite[Chap. 2 \& 3]{bogachevGaussianMeasures1998} for more details on this construction.} \\ 
    
\begin{enumerate}[(1)]
    \item $(\calT, \bE, [\cdot], \calN)$ is an ambient space by definition. 
    
    \item By Proposition \ref{prop:proxy_restriction_is_cont}, $\sbH$ is a well defined subset of $\bE$. 
    
    \item By definition $\bmu$ is a Borel probability measure on $\bE$ s.t. $\mu := \pi_{\ast} \bmu$ is centred Gaussian on $E$. By Proposition \ref{prop:proxy_restriction_is_cont}, since $\hat{\fL} \in \calP^{(\leq [\calT])}(E, \mu; \bE)$ is a measurable lift in the sense of Definition \ref{defn:lift_defn}, $\fL$ is an $\sH$-skeleton lift. Hence $\pi(\sbH) = \pi(\fL(\sH)) = \sH$ is indeed the Cameron--Martin space of $(E, \mu)$. 

    \item That $\fL$ is an $\sH$-skeleton lift was shown in (2). To see that $\fL$ is also a left inverse of $\pi\vert_{\sbH}$ observe that $\sbH$ is the image of $\sH$ under the map $\overline{\hat{\fL}}$, which makes $\overline{\hat{\fL}}$ surjective. Thus it has a left- and right-inverse, which both have to coincide with its left-inverse $\pi \vert_{\sbH}$.

    \item Both, $\hat{\fL} \in \calP^{(\leq [\calT])}(E, \mu; \bE)$ and $\fL = \overline{\hat{\fL}}$, are satisfied by assumption. 
\end{enumerate}    

The fact that the above construction does not depend on the representative of the $\mu$-a.s.\! equivalence class of $\hat{\fL}$ is a consequence of Proposition \ref{prop:proxy_restriction_is_cont}(3).
\end{proof}

\begin{figure}[ht]
\begin{center}
    \begin{tikzpicture}
        \path
        (-3,0) node(1b) {\texttt{classical}} 
        (-3,4) node(1t) {\texttt{enhanced}}

        (4,4) node(1ti) {\textcolor{blue}{$\sbH$}} 
        (5,4) node(1tsub) {\textcolor{blue}{$\subseteq$}}

        (6,4) node(2t) {$\bE$} 
        (8,4) node(3t) {$\bmu$} 
        (0,0) node(1b) {\textcolor{blue}{$\underbrace{\texttt{CM}(\pi_{\calN}(\bE), \left(\pi_{\calN}\right)_{\ast} \bmu)}_{= \sH}$}} 
        (6,0) node(2b) {\textcolor{blue}{$\underbrace{\pi_{\calN}(\bE)}_{= E}$}} 
        (8,0) node(3b) {\textcolor{blue}{$\underbrace{\left(\pi_{\calN}\right)_{\ast}\bmu}_{= \mu}$}};

        \draw[right hook->, blue] (1b) to node[above] {\textcolor{blue}{$i$}} (2b); 
        
        \draw[->, blue] (1b) to node[left] {\textcolor{blue}{$\fL = \overline{\hat{\fL}}~~$}} (1ti); 
                
        \draw[->] (2b) to [bend right] node[right] {$\hat{\fL}$} (2t);
        \draw[->>] (2t) to node[left] {$\pi_{\calN}$} (2b);
        
        \draw[->, blue] (3b) to node[right] {\textcolor{blue}{$\hat{\fL}_{\ast}$}} (3t);
    \end{tikzpicture}
    \caption{Summary of the Top-Down construction. Symbols in black are assumed (Top-Down Data) while symbols in blue are constructed.}
    \label{fig:top_down}
\end{center}
\end{figure}
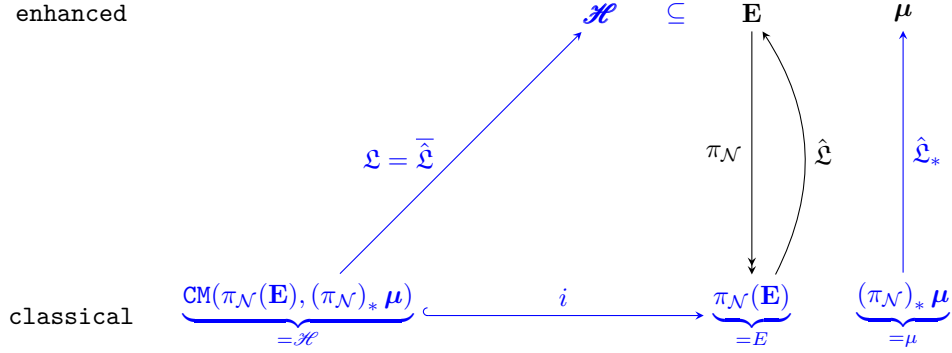

\subsection{Abstract Wiener Model Spaces with Approximation} \label{subsec:AWMS_with_approx}

In the present and the subsequent subsection we are going to flesh out the ideas of the ``Bottom-Up'' philosophy sketched in Subsection \ref{subsec:lifting}; that is, of building an AWMS primarily from the data of a skeleton lift, as opposed to primarily from the data of a full lift as in Subsection \ref{subsec:TD_constr}. We start by properly formulating what we mean by an approximation in this context. 

\begin{defn}[Admissible Approximation] \label{defn:admissible_approximation}
    Let $(\calT, \bE, [\cdot], \calN)$ be an ambient space, let $(E, \sH, i, \mu)$ be an abstract Wiener space, let $\sK$ be an intermediate space, and let $\fM: \sK \rightarrow \bE$ be a $\sK$-skeleton lift. A sequence of bounded linear operators $(\Phi_m)_{m \in \bbN}: E \rightarrow \sK$ is called an \textbf{admissible approximation} if it satisfies the following properties:

    \begin{enumerate}[(1)]
        \item Approximation of identity in $E$ and $\sK$: For $m \rightarrow \infty$ we have 
    
        \begin{align}
        \Vert \Phi_m(x) &- x \Vert_E \rightarrow 0, \quad \text{for} ~ \mu-a.e. ~ x \in E, \label{eq:A2.I}\\
        \Vert \Phi_m(h) &- h \Vert_{\sK} \rightarrow 0, \quad \forall h \in \sH. \label{eq:A2.II}
        \end{align}

        \item Existence of a limit in probability: The limit $X:= \lim_{m \rightarrow \infty} \fM \circ \Phi_m$ exists in probability w.r.t. $\mu$; i.e. 
        
        \begin{equation}
        \lim_{m \rightarrow \infty} \mu \left( \vertiii{\fM \circ \Phi_m - X} > \eta \right) = 0, \quad \forall \eta > 0 . \label{eq:A2.III}
        \end{equation}

        \item Compatibility with $\fM$: For every $\tau \in \calT$, $h \in \sH$ and $a \in \{x, h\}^{[\tau]}$
        
        \begin{equation}
            x \mapsto \fM_{\tau}^{\otimes} (\Phi_{m}(a_1) \otimes \ldots \otimes \Phi_{m}(a_{[\tau]})), \label{eq:sK_compatibility}
        \end{equation}
        
        converges in probability w.r.t. $\mu$ in $E_{\tau}$ for $m \rightarrow \infty$.
    \end{enumerate}
\end{defn}

Note that condition \eqref{eq:A2.II} has to be satisfied only for $h \in \sH$ and only in the $\sK$-norm. Also note in \eqref{eq:sK_compatibility} that $a = (x, \ldots, x)$ is nothing but \eqref{eq:A2.III}. Since $\Vert \cdot \Vert_{\bE}$ and $\vertiii{\cdot}$ are equivalent metrics, condition \eqref{eq:A2.III} may equivalently be stated in terms of $\Vert \cdot \Vert_{\bE}$.

\begin{rem}
    Given a fixed lift $\fM$, the challenge is to find an intermediate space $\sK$ and an admissible approximation $(\Phi_m)_{m \in \bbN}$ such that on the one hand $\sK$ is large enough (with a topology which is weak enough) such that each $\Phi_m$ maps into $\sK$ (cf. the discussion in Subsection \ref{subsec:intermediate_space}) and Condition \eqref{eq:A2.II} is satisfied, but which on the other hand is small enough (with a topology which is strong enough) such that $\fM$ is well-defined and continuous on $\sK$. 
\end{rem}

\begin{exmp}[Running example continued]    
    In the case of Stratonovich Brownian motion the skeleton lift is given by iterated Young integration of paths (see Subsection \ref{subsec:GaussianRP}), $\sH$ is given by the classical Cameron--Martin space of Brownian motion, while a suitable choice for $\sK$ is $\calC^{0, 1-\var}$ (see Subsection \ref{subsubsec:GRP_with_approx}).     
    
    Two natural choices for $(\Phi_m)_{m \in \bbN}$ are the Karhunen--Loève approximation and the piecewise linear approximation (see Subsection \ref{subsubsec:GRP_with_approx}). Under mild assumptions on the ambient space one can show that in the former case the choice of $\sH$ is sufficient as an intermediate space, while in the later case the larger space $\calC^{0, 1-\var}$ is sufficient - see Proposition \ref{prop:KL_is_approx} and \ref{prop:PL_is_approx}.    
\end{exmp}

The Bottom-Up construction, assuming more structure (in the form of an admissible approximation), naturally leads to a stronger notion of AWMS, leading in turn to stronger theorems (see e.g. Theorem \ref{thm:shift_operator_consistency}). This is the content of the following definition. 

\begin{defn}[Abstract Wiener Model Space with Approximation] \label{defn:AWMS_with_approx}
    An \textbf{abstract Wiener model space with approximation} is an AWMS $((\calT, \bE, [\cdot], \calN), \sbH, \bmu, \fL, \hat{\fL})$ together with two additional pieces of data, called an \textbf{approximation scheme}:

    \begin{enumerate}[(1)]
    \setcounter{enumi}{5}
        \item an intermediate space $\sH \hookrightarrow \sK \subseteq E$ (as in Definition \ref{defn:intermediate_space}) and a $\sK$-skeleton lift $\fM: \sK \rightarrow \bE$ such that $\fM\vert_{\sH} = \fL$. Here, $\sH$ denotes the Cameron--Martin space associated to $\pi_{\ast} \bmu$.\footnote{Note that if $\sK = \sH$ this implies $\fM = \fL$.}
        
        \item an admissible approximation $(\Phi_m)_{m \in \bbN}$ (in the sense of Definition \ref{defn:admissible_approximation}) such that the limit in probability postulated in (2) of Definition \ref{defn:admissible_approximation} coincides with $\hat{\fL}$;\footnote{Cf. Remark on pr. \pageref{rem:AKS}.} i.e. $\hat{\fL} = \lim_{m \rightarrow \infty} \fM \circ \Phi_m$ in probability w.r.t. $\mu$.       
    \end{enumerate}
\end{defn}

If we want to emphazise that an AWMS is not an AWMS \emph{with} approximation, we call it \textbf{bare}.

\subsection{Bottom-Up Construction} \label{subsec:bottom_up_construction}

We now turn to the statement and proof of the second large construction of this article, although much of the heavy lifting in the proof will be outsourced to Lemma \ref{lem:agreement_for_bottom_up}.

\begin{thm}[Bottom-Up Construction] \label{thm:bottom_up_construction}
    Let $(\calT, \bE, [\cdot], \calN)$ be an ambient space, 
    with $E := E_{\calN}$ s.t. $(E, \sH, i, \mu)$ is an abstract Wiener space. 
    Let $\sK$ be an intermediate space, 
    let $\fM: \sK \rightarrow \bE$ be a $\sK$-skeleton lift, and 
    let $(\Phi_m)_{m \in \bbN}$ be an admissible approximation
    s.t.
    
    \begin{equation}
        \fL \circ \Phi_m \in \calP^{(\leq [\calT])}(E, \mu; \bE), \quad \forall m \in \bbN. \label{eq:WIC_assumption_BU}
    \end{equation}
    
    Then the following data constitutes an abstract Wiener model space with approximation: 

        \begin{enumerate}[(1)]
            \item the ambient space $(\calT, \bE, [\cdot], \calN)$ 
            \item the subset $\sbH := \fM\vert_{\sH}(\sH)$
            \item the enhanced measure $\bmu := \hat{\fL}_{\ast} \mu$ (with $\hat{\fL}_{\ast}$ defined in (5))
            \item the skeleton lift $\fL := \fM\vert_{\sH}$
            \item the full lift $\hat{\fL} := \lim_{m \rightarrow \infty} \fM \circ \Phi_m$
            \item the intermediate space $\sK$ with the $\sK$-skeleton lift $\fM$
            \item the admissible approximation $(\Phi_m)_{m \in \bbN}$
        \end{enumerate}

        Furthermore, if $\sK$ is a compact intermediate space, then $\fL: \sH \rightarrow \bE$ is compact.\footnote{By a compact (non-linear) function between metric spaces we mean one which sends bounded sets to pre-compact sets.}

        For future reference, we will refer to a quadruple $((\calT, \bE, [\cdot], \calN), (E, \sH, i, \mu), (\sK, \fM), (\Phi_m)_{m \in \bbN})$ satisfying the assumptions of Theorem \ref{thm:bottom_up_construction} above as \textbf{Bottom-Up data}. The construction may be summarized by Figure \ref{fig:bottom_up}.
\end{thm}
\begin{proof}
    \begin{enumerate}[(1)]
        \item $(\calT, \bE, [\cdot], \calN)$ is an ambient space by definition. 
        \item $\sbH$ is a well defined subset of $\bE$.
        \item Since the approximation $(\Phi_m)_{m \in \bbN}$ is admissible, the sequence $(\fM \circ \Phi_m)_{m \in \bbN}$ converges in probability to a measurable function $E \rightarrow \bE$. By assumption this limit is the full lift $\hat{\fL}$. 
        Thus, since $\mu$ is a Borel probability measure, so is $\bmu$. There exists a subsequence $(\fM \circ \Phi_{m_k})_{k \in \bbN} \subseteq (\fM \circ \Phi_m)_{m \in \bbN}$ and a $\mu$-nullset $N_{\Phi} \subseteq E$ s.t.

        \begin{equation}
            (\fM \circ \Phi_{m_k})(x) \rightarrow \hat{\fL}(x), \quad x \in N_{\Phi}^c.
        \end{equation}
    
        Thus, due to $\pi \circ \fM = \id_{\sK}$ and \eqref{eq:A2.I}, for $\mu$-a.e. $x \in E$
    
        \begin{align}
            &\left\Vert \left(\pi  \circ \hat{\fL} \right)(x) - x \right\Vert_E = \left\Vert \pi \left( \lim_{k \rightarrow \infty} \fM (\Phi_{m_k}(x)) \right) - x \right\Vert_E \\
            &= \lim_{k \rightarrow \infty} \left\Vert \pi(\fM (\Phi_{m_k}(x))) - x \right\Vert_E = \lim_{k \rightarrow \infty} \Vert \Phi_{m_k}(x) - x \Vert_E = 0 .
        \end{align}
    
        That is, $\hat{\fL}$ is a measurable lift and in particular 

        \begin{equation}
            \pi_{\ast}(\bmu) = \pi_{\ast}(\hat{\fL}_{\ast}(\mu)) = (\pi \circ \hat{\fL})_{\ast}(\mu) = \mu. 
        \end{equation}

       That $\pi(\sbH)$ is the Cameron--Martin space of $(E, \mu)$ follows from the fact that $\fM$ and therefore $\fM \vert_{\sH}$ is a lift.
        
        \item The fact that $\fL$ is an $\sH$ skeleton lift follows from Definitions \ref{defn:intermediate_space} and \ref{defn:admissible_approximation}, while the $\calT$-multi-linearity of $\fL$ can be seen by defining $\fL^{\otimes}_{\tau} = \fM^{\otimes}_{\tau} \vert_{\sH^{\otimes_A [\tau]}}$ for any $\tau \in \calT$. To see that $\fL$ is also a left inverse of $\pi\vert_{\sbH}$ observe that $\sbH$ is the image of $\sH$ under the map $\fM \vert_{\sH}$, which makes $\fL$ surjective onto $\sbH$. Thus it has a left- and right-inverse on $\sbH$, which both have to coincide with its left-inverse $\pi \vert_{\sbH}$.

        \item The fact that $\hat{\fL}$ is a measurable lift was shown in \textit{(3)}, while $\hat{\fL}_{\ast} \mu = \bmu$ is true by definition. To see the graded Wiener--Ito chaos assumption, let $\tau \in \calT$ be arbitrary. Then since the convergence $\fM \circ \Phi_{m_k} \rightarrow \hat{\fL}$ and thus $\pi_{\tau} \circ \fM \circ \Phi_{m_k} \rightarrow \pi_{\tau} \circ \hat{\fL}$ is in probability w.r.t. $\mu$ and $\calP^{(\leq [\tau])}(E, \mu; E_{\tau})$ is closed under convergence in probability w.r.t. $\mu$ (see Lemma \ref{lem:complete_in_probability}) we obtain $\pi_{\tau} \circ \hat{\fL} \in \calP^{(\leq [\tau])}(E, \mu; E_{\tau})$ for each $\tau \in \calT$.
        
        The proof of Property \eqref{eq:proxy_restriction_equals_skeleton_lift} is considerably more involved and will thus be done separately in Lemma \ref{lem:agreement_for_bottom_up}.        

        \item By assumption $\sK$ is an intermediate space and $\fM$ is a $\sK$-skeleton lift. 

        \item By assumption $(\Phi_m)_{m \in \bbN}$ is an admissible approximation. \\ 
    \end{enumerate}
    
    To see the statement about compactness, let $A \subseteq \sbH$ be bounded. Then since $i$ is a compact linear operator, $i(A)$ is pre-compact in $\sK$ and hence, since $\fM$ is continuous on $\sK$, $i(A)$ is pre-compact in $\bE$.
\end{proof}

\begin{figure}[ht]
\begin{center}
        \begin{tikzpicture}
        \path
        (-2,0) node(1b) {\texttt{classical}} 
        (-2,4) node(1t) {\texttt{enhanced}}
        
        (4,4) node(1t) {\textcolor{blue}{$\sbH$}} 
        (5,4) node(1tsub) {\textcolor{blue}{$\subseteq$}}
        (6,4) node(2t) {$\bE$} 
        (10,4) node(3t) {\textcolor{blue}{$\bmu = \hat{\fL}_{\ast} \mu$}} 
        (0,0) node(1b) {$\sH$} 
        (4,0) node(1bK) {$\sK$} 
        (5,0) node(1bs) {$\subseteq$} 
        (6,0) node(2b) {$E$} 
        (10,0) node(3b) {$\mu$};
        
        \draw[right hook->] (1b) to node[above] {$i$} (1bK); 
        
        \draw[->] (2b) to [bend left] node[below] {$(\Phi_m)_{m \in \mathbb{N}}$} (1bK); 
        \draw[right hook->, blue] (1b) to node[left] {\textcolor{blue}{$\fL = \fM \vert_{\sH}~$}} (1t); 
        \draw[right hook->] (1bK) to node[left] {$\fM$} (2t); 
    
        \draw[->, blue] (2b) to [bend right] node[right] {\textcolor{blue}{$\underbrace{\lim_{m \rightarrow \infty} \fL \circ \Phi_m}_{= \hat{\fL}}$}} (2t);
        \draw[->>] (2t) to node[left] {$\pi_{\calN}$} (2b);
        
        \draw[->, blue] (3b) to node[right] {\textcolor{blue}{$\hat{\fL}_{\ast}$}} (3t);
        \end{tikzpicture}
        \caption{Summary of the Bottom-Up construction. Symbols in black are assumed (Bottom-Up Data) while symbols in blue are constructed.}
        \label{fig:bottom_up}
    \end{center}
\end{figure}

\begin{lem}[Property \eqref{eq:proxy_restriction_equals_skeleton_lift} for Bottom-Up Construction] \label{lem:agreement_for_bottom_up}
    In the context of Theorem \ref{thm:bottom_up_construction}, $\overline{\hat{\fL}_{\tau}} = \fL_{\tau}$ for every $\tau \in \calT$. 
\end{lem}
\begin{proof}
    We split the proof into several parts: \\

    1) Let $\tau \in \calT$ and $h \in \sH$ be arbitrary. By assumption, $\fM \circ \Phi_m \rightarrow \hat{\fL}$ w.r.t. $\mu$ and thus also $\fM_{\tau} \circ \Phi_m \rightarrow \hat{\fL}_{\tau}$ in probability w.r.t. $\mu$. By sequential completeness of $\calP^{(\leq [\tau])}(E, \mu; E_{\tau})$ (see Lemma \ref{lem:complete_in_probability}) and Lemma \ref{lem:equiv_of_norms} we obtain $\Vert \fM_{\tau} \circ \Phi_m - \hat{\fL}_{\tau} \Vert_{L^2(E, \mu; E_{\tau})} \rightarrow 0$ and thus $\Vert (\fM_{\tau} \circ \Phi_m)^{\circ} - \hat{\fL}_{\tau}^{\circ} \Vert_{L^2(E, \mu; E_{\tau})} \rightarrow 0$ since $\Pi_{[\tau]}$ is $L^2$-continuous. By Proposition \ref{prop:L2L1} we further obtain $\Vert (\fM_{\tau} \circ \Phi_m)^{\circ} - \hat{\fL}_{\tau}^{\circ} \Vert_{L^1(E, \mu_h; E_{\tau})} = \Vert (\fM_{\tau} \circ \Phi_m)^{\circ} \circ T_h - \hat{\fL}_{\tau}^{\circ} \circ T_h \Vert_{L^1(E, \mu; E_{\tau})} \rightarrow 0$. In conclusion we have the following: (we indicate the mode of convergence as a super-script for clarity.)
    
    \begin{align}
        \overline{\hat{\fL}_{\tau}}(h) &= \bbE \left[ \hat{\fL}_{\tau}^{\circ} \circ T_h \right] = \bbE \left[ \left( \lim_{m \rightarrow \infty}^{\mu} \fM_{\tau} \circ \Phi_m \right)^{\circ} \circ T_h \right] = \bbE \left[ \left(\lim_{m \rightarrow \infty}^{L^2(E, \mu; E_{\tau})} \fM_{\tau} \circ \Phi_m \right)^{\circ} \circ T_h \right] \\
        &= \bbE \left[ \left(\lim_{m \rightarrow \infty}^{L^2(E, \mu; E_{\tau})} (\fM_{\tau} \circ \Phi_m)^{\circ} \right) \circ T_h \right] = \bbE \left[ \left(\lim_{m \rightarrow \infty}^{L^1(E, \mu_h; E_{\tau})} (\fM_{\tau} \circ \Phi_m)^{\circ} \right) \circ T_h \right] \\
        &= \bbE \left[ \lim_{m \rightarrow \infty}^{L^1(E, \mu; E_{\tau})} \left( (\fM_{\tau} \circ \Phi_m)^{\circ} \circ T_h \right) \right] = \lim_{m \rightarrow \infty}^{E_{\tau}} \bbE \left[ (\fM_{\tau} \circ \Phi_m)^{\circ} \circ T_h \right]. \label{eq:lim_h}
    \end{align}

    2) From now on fix $m \in \bbN$ and recall that, as a push-forward of a Gaussian measure by a bounded linear map, $\mu_m:= (\Phi_m)_{\ast} \mu$ is a Gaussian measure on $\sK$. Denote by $\sH_m \subseteq \sK$ its Cameron--Martin space. If $\mu_m$ is degenerate on $\sK$, i.e. if there exists a non-zero $f \in \sK^{\ast}$ s.t. $\bbE_{\mu_m} \left[ \vert f \vert^2 \right] = 0$, then there exists a closed linear subspace\footnote{In the case where $\mu_m$ is not degenerate on $\sK$ the above is still true with $E_m = \sK$.} $E_m \subseteq \sK$ on which $\mu_m$ is non-degenerate and in which $\sH_m$ is $\sK$-dense. In particular, $\Phi_m$ takes values in $E_m$ $\mu$-a.s.\! Therefore $((E_m, \Vert \cdot \Vert_{\sK}), \sH_m, \mu_m)$ constitutes an abstract Wiener space. Let $(e^m_k)_{k \in \bbN}$ be an ONB of $\sH_m$ contained in $E_m^{\ast}$. Then by classical abstract Wiener space theory\footnote{See e.g. \cite[App. D.3]{frizMultidimensionalStochasticProcesses2010}.} %I think one can set up this argument without reducing to $E_m$, but safety first. 

    \begin{equation}
        y = \sum_{k =1}^{\infty} \langle e^m_k, y \rangle_{\sH_m} e^m_k, \quad \text{for} \quad \mu_m \text{-a.e.} ~ y \in E_m, \quad \forall y \in \sH_m,
    \end{equation}

    where the convergence is w.r.t. $\Vert \cdot \Vert_{\sK}$ and hence 

    \begin{equation}
        \Phi_m(x) = \sum_{k =1}^{\infty} \langle e^m_k, \Phi_m(x) \rangle_{\sH_m} e^m_k, \quad \text{for} \quad \mu \text{-a.e.} ~ x \in E, \quad \forall x \in \sH, \label{eq:consistency_convergencee_also_on_H}
    \end{equation}

    where the convergence is w.r.t. $\Vert \cdot \Vert_{\sK}$. Here we used the fact that $\Phi_m(\sH) \subseteq \sH_m$, proven in Proposition \ref{prop:containment_of_image_CM}. \\
    
    For the remainder of this part of the proof, let $x \in E$ be fixed such that the above converges. Since $\fL$ is continuous on $\sK$ we have
    
    \begin{align}
        (\fM_{\tau} \circ \Phi_m)(x) &= \fM_{\tau} \left( \lim_{n \rightarrow \infty}^{\sK} \sum_{k = 1}^{n} \langle e^m_k, \Phi_m(x) \rangle_{\sH_m} e^m_k \right) \\
        &= \lim_{n \rightarrow \infty}^{E_{\tau}} \fM_{\tau} \left(  \sum_{k = 1}^{n} \langle e^m_k, \Phi_m(x) \rangle_{\sH_m} e^m_k \right). \label{eq:limit_in_consist}
    \end{align}
    
    To streamline notation, for every $k, n \geq 0$ define
    
    \begin{align}
        A^{k}_n &:= \left\{ \alpha: \bbN \rightarrow \bbN_0 : \sum_{i \in \bbN} \alpha_i = k, \alpha_i = 0 ~ \text{for every} ~ i > n \right\}, \quad A^{\leq k}_n := \bigcup_{i = 1}^k A^i_n, \label{symb:Akn} \\
        \underline{e}_{\alpha}^m &:= \bigotimes_{i \in \bbN} (e_i^{m})^{\otimes \alpha_i} \in \sH_m^{\otimes [\tau]}, \quad \alpha \in A^{[\tau]}_n, 
    \end{align}
    
    \begin{equation}
        \quad \left\langle \otimes_{i=1}^{[\tau]} v_i, \otimes_{j=1}^{[\tau]} w_j \right\rangle_{\sH_m^{\otimes [\tau]}} := \prod_{i = 1}^{[\tau]} \langle v_i, w_i \rangle_{\sH_m}, \quad v_1, \ldots, v_{[\tau]}, w_1, \ldots, w_{[\tau]} \in \sH_m .
    \end{equation}

    Using the $\calT$-multi-linearity of the lift we may rewrite \eqref{eq:limit_in_consist} as 
    
    \begin{align}
        \lim_{n \rightarrow \infty} \sum_{\alpha \in A^{[\tau]}_n} \prod_{i \in \bbN} \langle e^m_i, \Phi_m(x) \rangle^{\alpha_i}_{\sH_m} \fM_{\tau}^{\otimes} \left( \bigotimes_{i \in \bbN} (e_i^{m})^{\otimes \alpha_i} \right)
        = \lim_{n \rightarrow \infty} \sum_{\alpha \in A^{[\tau]}_n} \langle \underline{e}_{\alpha}^m, \Phi_m(x)^{\otimes [\tau]} \rangle_{\sH^{\otimes [\tau]}_m} \fL^{\otimes}_{\tau} \left( \underline{e}_{\alpha}^m \right) . \label{eq:consistency_split}
    \end{align}

    Note that $A^{[\tau]}_n$ is a finite set and that only finitely many indices in the product give factors differing from $1$. For every $\alpha \in A^{[\tau]}_n$ the functional $x \mapsto \langle \underline{e}_{\alpha}^m, \Phi_m(x)^{\otimes [\tau]} \rangle_{\sH^{\otimes [\tau]}_m}$ is a $[\tau]$-fold product of bounded linear functionals $x \mapsto \langle e^m_i, \Phi_m(x) \rangle_{\sH_m}$ on $E$, and thus lies in $\calP^{(\leq [\tau])}(E, \mu; \bbR)$ by Lemma \ref{lem:polynomial_in_WIC}. The above limit is a $\mu$-a.s.\! limit and therefore a limit probability w.r.t. $\mu$ in $E_{\tau}$ of elements in $\calP^{(\leq [\tau])}(E, \mu; E_{\tau})$. Thus, by Lemma \ref{lem:equiv_of_norms} the sum converges in $L^2(E, \mu; E_{\tau})$. Therefore, applying the projection onto the $[\tau]$-th component of the Wiener--Ito chaos to \eqref{eq:consistency_split} and pulling out the limit via $L^2$-continuity we obtain
    
    \begin{align}
        (\fM_{\tau} \circ \Phi_m)^{\circ}(x) &= \Pi_{[\tau]} \lim_{n \rightarrow \infty} \sum_{\alpha \in A^{[\tau]}_n} \langle \underline{e}_{\alpha}^m, \Phi_m(x)^{\otimes [\tau]} \rangle_{\sH^{\otimes [\tau]}_m} \fM_{\tau}^{\otimes} \left( \underline{e}_{\alpha}^m \right) \\
        &= \lim_{n \rightarrow \infty} \sum_{\alpha \in A^{[\tau]}_n} \Pi_{[\tau]} \langle \underline{e}_{\alpha}^m, \Phi_m(x)^{\otimes [\tau]} \rangle_{\sH^{\otimes [\tau]}_m} \fM_{\tau}^{\otimes} \left( \underline{e}_{\alpha}^m \right) . \label{eq:consistency_2100}
    \end{align}

    3) Fix now $n \in \bbN$ and $\alpha \in A^{[\tau]}_n$. Let $h_i$ denote the $i$-th (monic) Hermite polynomial and let  

    \begin{equation}
        H^m_{\alpha}(x) := \prod_{i \in \bbN} h_{\alpha_i} \left( \langle e^m_i, x \rangle_{\sH_m} \right), \quad x \in E, \label{symb:hermite}
    \end{equation}
    
    denote the multi-dimensional Hermite polynomial with multi-index $\alpha$ associated\footnote{Recall that while the Hermite polynomial $h_i$ is defined without reference to any underlying measure, the definition of the multi-dimensional Hermite polynomials $H_{\alpha}$ involves an ONB of the Cameron--Martin space of the underlying Gaussian measure $\mu$ and thus depends on that measure.} to $\mu_m$ and note that 
    
    \begin{equation}
        \prod_{i \in \bbN} \langle e^m_i, \Phi_m(x) \rangle^{\alpha_i}_{\sH_m} = \langle \underline{e}_{\alpha}^m, \Phi_m(x)^{\otimes [\tau]} \rangle_{\sH^{\otimes [\tau]}_m}
    \end{equation}

    is the leading monomial of $H^m_{\alpha}(\Phi_m(\cdot))$. Thus we conclude that the second term on the right-hand side of 

    \begin{equation}
        \langle \underline{e}_{\alpha}^m, \Phi_m(x)^{\otimes [\tau]} \rangle_{\sH^{\otimes [\tau]}_m} = \underbrace{H^m_{\alpha}(\Phi_m(x))}_{\in \calP^{([\tau])}(E, \mu; \bbR)} + \underbrace{\left( \langle \underline{e}_{\alpha}^m, \Phi_m(x)^{\otimes [\tau]} \rangle_{\sH^{\otimes [\tau]}_m} - H^m_{\alpha}(\Phi_m(x)) \right)}_{\in \calP^{(< [\tau])}(E, \mu; \bbR)} \label{eq:consistency_splitting}
    \end{equation}

    is a polynomial of bounded linear functionals of degree strictly less than $\sum_{i \in \bbN} \alpha_i = [\tau]$ and therefore must lie in $\calP^{(< [\tau])}(E, \mu; \bbR)$ by Lemma \ref{lem:polynomial_in_WIC}. The first term, on the other hand, has the form 

    \begin{equation}
        H^m_{\alpha}(\Phi_m(x)) = \prod_{i \in \bbN} h_{\alpha_i} \left( \langle e^m_i, \Phi_m(x) \rangle_{\sH_m} \right). 
    \end{equation}

    Since $(e^m_i)_{i \in \bbN}$ is an ONB of $\sH_m$, the functionals $(\langle e^m_i, \cdot \rangle_{\sH_m})_{i \in \bbN}$ are iid $\sim \sN(0,1) \label{symb:normal_distr}$ w.r.t. $\mu_m$ and thus orthonormal in $L^2(E_m, \mu_m; \bbR)$. Therefore $(\langle e^m_i, \Phi_m(\cdot) \rangle_{\sH_m})_{i \in \bbN}$ is iid $\sim \sN(0,1)$ w.r.t. $\mu$ and thus orthonormal in $L^2(E, \mu; \bbR)$. Denote by $\fC_{\mu}: E^{\ast} \rightarrow E$ the covariance operator associated to $\mu$. Then by the above $(\fC_{\mu}\langle e^m_i, \Phi_m(\cdot) \rangle_{\sH_m})_{i \in \bbN}$ forms an orthonormal system in $\sH$, which can be completed to an ONB of $\sH$. Hence, since the definition of Wiener--Ito chaos is independent of the choice of ONB of $\sH$, we conclude that $H^m_{\alpha}(\Phi_m(\cdot)) \in \calP^{([\tau])}(E, \mu; \bbR)$. Thus according to \eqref{eq:consistency_splitting}

    \begin{equation}
       \Pi_{[\tau]} \left\langle \underline{e}_{\alpha}^m, \Phi_m(x)^{\otimes [\tau]} \right\rangle_{\sH^{\otimes [\tau]}_m} \fL^{\otimes}_{\tau} \left( \underline{e}_{\alpha}^m \right) = H_{\alpha}(\Phi_m(x)) \fL^{\otimes}_{\tau} \left( \underline{e}_{\alpha}^m \right)
    \end{equation}
    
    and hence, by inserting back into \eqref{eq:consistency_2100}, 

    \begin{equation}
        (\fM_{\tau} \circ \Phi_m)^{\circ}(x) = \lim_{n \rightarrow \infty} \sum_{\alpha \in A^{[\tau]}_n} \underbrace{H_{\alpha}(\Phi_m(x)) \fL^{\otimes}_{\tau} \left( \underline{e}_{\alpha}^m \right)}_{\in \calP^{([\tau]}(E, \mu; E_{\tau})}, \quad \mu-a.s.\! \label{eq:consistency_pre_4}
    \end{equation}

    Since in \eqref{eq:consistency_pre_4} all elements of the sequence lie in $\calP^{([\tau])}(E, \mu; E_{\tau})$, Lemma \ref{lem:equiv_of_norms} shows that the convergence is not only $\mu$-a.s.\! in $E_{\tau}$, but also in $L^2(E, \mu; E_{\tau})$. \\

    4) Inserting \eqref{eq:consistency_pre_4} into \eqref{eq:lim_h} yields

    \begin{equation}
        \bbE\left[ \left( \fM_{\tau} \circ \Phi_m \right)^{\circ}(\cdot + h) \right] = \bbE \left[ \lim_{n \rightarrow \infty} \sum_{\alpha \in A^{[\tau]}_n} H^m_{\alpha}(\Phi_m(\cdot + h)) \fM_{\tau}^{\otimes} \left( \underline{e}_{\alpha}^m \right)\right]. \label{eq:consistency_4}
    \end{equation}

    By the remarks at the end of 3) and Proposition \ref{prop:L2L1} we can pull the limit outside of the expectation, giving
    
    \begin{align}
        &= \lim_{n \rightarrow \infty} \sum_{\alpha \in A^{[\tau]}_n} \bbE\left[ H^m_{\alpha}(\Phi_m(\cdot + h)) \right] \fM_{\tau}^{\otimes} \left( \underline{e}_{\alpha}^m \right) \\
        &= \lim_{n \rightarrow \infty} \sum_{\alpha \in A^{[\tau]}_n} \bbE\left[ \prod_{i \in \bbN} h_{\alpha_i}(\langle e^m_i, \Phi_m(\cdot + h) \rangle_{\sH_m}) \right] \fM_{\tau}^{\otimes} \left( \underline{e}_{\alpha}^m \right) \\
        &= \lim_{n \rightarrow \infty} \sum_{\alpha \in A^{[\tau]}_n} \underbrace{\bbE\left[ \prod_{i \in \bbN} h_{\alpha_i} \big( \langle e^m_i, \Phi_m(\cdot) \rangle_{\sH_m} + \langle e^m_i, \Phi_m(h) \rangle_{\sH_m} \big) \right]}_{(\ast)} \fM_{\tau}^{\otimes} \left( \underline{e}_{\alpha}^m \right) . \label{eq:consistency_post_4}
    \end{align}

    For the rest of this part of the proof fix $n \geq 0$ and $\alpha \in A^{[\tau]}_{n}$ and focus on $(\ast)$: applying the Binomial theorem for Hermite polynomials (see Proposition \ref{prop:hermite_polynomial_binomial}) to $h_{\alpha_i}$ yields 
    
    \begin{equation}
        (\ast) = \bbE\left[ \prod_{i \in \bbN} \left( \sum_{l = 0}^{\alpha_{i}} \binom{\alpha_i}{l} h_l \left( \langle e^m_i, \Phi_m(\cdot) \rangle_{\sH_m} \right) \langle e^m_i, \Phi_m(h) \rangle_{\sH_m}^{\alpha_i - l} \right) \right]. 
    \end{equation}
    
    Write $S_{\alpha}:= \times_{i \in \bbN} \{0, \ldots, \alpha_i\}$ and note that the cardinality of $S_{\alpha}$ is finite. We switch the sum and the product to obtain
    
    \begin{align}
        &= \bbE\left[ \sum_{\sigma \in S_{\alpha}} \prod_{i \in \bbN} \binom{\alpha_i}{\sigma_i} h_{\sigma_i} \left( \langle e^m_i, \Phi_m(\cdot) \rangle_{\sH_m} \right) \langle e^m_i, \Phi_m(h) \rangle_{\sH_m}^{\alpha_i - \sigma_i} \right] \\
        &= \sum_{\sigma \in S_{\alpha}} \bbE\left[ \prod_{i \in \bbN} \binom{\alpha_i}{\sigma_i} h_{\sigma_i} \left( \langle e^m_i, \Phi_m(\cdot) \rangle_{\sH_m} \right) \langle e^m_i, \Phi_m(h) \rangle_{\sH_m}^{\alpha_i - \sigma_i} \right] .
    \end{align}
    
    Since the sequence $(\langle e^m_i,\cdot \rangle_{\sH_m})_{i \in \bbN}$ is iid w.r.t. $\mu_m = (\Phi_m)_{\ast} \mu$ we may pull the product out of the expectation to obtain
    
        \begin{align}
        &= \sum_{\sigma \in S_{\alpha}} \prod_{i \in \bbN} \binom{\alpha_i}{\sigma_i} \underbrace{\bbE\left[ h_{\sigma_i} \left( \langle e^m_i, \Phi_m(\cdot) \rangle_{\sH_m} \right) \right]}_{= 1_{\{\sigma_i = 0\}}} \langle e^m_i, \Phi_m(h) \rangle_{\sH_m}^{\alpha_i - \sigma_i} \\  
        &= \prod_{i \in \bbN} \langle e^m_i, \Phi_m(h) \rangle_{\sH_m}^{\alpha_i} \\
        &= \langle \underline{e}_{\alpha}, \Phi_m(h)^{\otimes [\tau]} \rangle_{\sH_m^{\otimes [\tau]}}, \label{eq:consistency_pre_5}
    \end{align}

    where from the first to the second line we used the fact that all Hermite polynomials of order $>0$ are centered, and those of order $0$ have expectation $1$. Finally, insert \eqref{eq:consistency_pre_5} back into \eqref{eq:consistency_post_4} to obtain 

    \begin{equation}
        \bbE\left[ \left( \fM_{\tau} \circ \Phi_m \right)^{\circ}(\cdot + h) \right] = \lim_{n \rightarrow \infty} \sum_{\alpha \in A^{[\tau]}_n} \langle \underline{e}_{\alpha}, \Phi_m(h)^{\otimes [\tau]} \rangle_{\sH_m^{\otimes [\tau]}} \fM_{\tau}^{\otimes} \left( \underline{e}_{\alpha}^m \right) = \left( \fM_{\tau} \circ \Phi_m \right)(h) , \label{eq:consistency_pre_pre_5}
    \end{equation}

    where in the last equality we used the second quantifier in  \eqref{eq:consistency_convergencee_also_on_H}. \\
    
    5) In conclusion, using \eqref{eq:lim_h}, then \eqref{eq:consistency_pre_pre_5}, and then \eqref{eq:A2.II}, for every $h \in \sH$
    
    \begin{equation}
        \overline{\hat{\fL}_{\tau}}(h)= \lim_{m \rightarrow \infty} \bbE \left[ \left( \fM_{\tau} \circ \Phi_m \right)^{\circ}(\cdot + h) \right] = \lim_{m \rightarrow \infty} \left( \fM_{\tau} \circ \Phi_m \right)(h) = \fM_{\tau}(h) = \fL_{\tau}(h) .
    \end{equation}

    Hence $\fL_{\tau} = \overline{\hat{\fL}_{\tau}}$ for every $\tau \in \calT$, which concludes the proof. 
\end{proof}

\section{Large Deviations} \label{sec:LDP}

As alluded to in the introduction, one of our goals is to derive results on large deviations for the family of measures $((\delta_{\varepsilon})_{\ast} \bmu)_{\varepsilon > 0}$ associated to an abstract Wiener model space. 

\begin{rem}
    In the entirety of the paper, all LDPs will be assumed to have speed $\varepsilon^2$ without further comment. 
\end{rem}

\begin{thm}[LDP for AWMS] \label{thm:LDP_for_AWMS}
    Let $((\calT, \bE, [\cdot], \calN), \sbH, \bmu, \fL, \hat{\fL})$ be an abstract Wiener model space. Then the family of measures $((\delta_{\varepsilon})_{\ast} \bmu)_{\varepsilon > 0} = (\bmu_{\varepsilon})_{\varepsilon > 0}$ satisfies an LDP on $\bE$ with good rate function $\sJ: \bE \rightarrow [0, \infty]$ given by 
    
        \begin{equation}
            \sJ(\bx) = \begin{cases}
            \frac{1}{2} \Vert \pi (\bx) \Vert^2_{\sH} & \quad \bx \in \sbH \\
            + \infty & \quad \text{else}. 
            \end{cases} \label{eq:LDP_for_AWMS}
        \end{equation}
\end{thm}

\begin{rem}[Form of the rate function]
    The major insight, already understood by \cite{hairerLargeDeviationsWhiteNoise2015} in their setup, which becomes apparent in the proof of Lemma \ref{lem:LDP_of_each_approx}, is the fact that, firstly, all contributions from components in the Wiener--Ito chaos expansion of degree less than $[\tau]$, as well as, secondly, all contributions from those components of degree $[\tau]$, but not of leading order, are ``scaled away''. This is how the specific form of the rate function of Theorem \ref{thm:LDP_for_AWMS} and thus the definition of the proxy-restriction arises. 
    
    The first part is taken care of by projecting the full lift into homogeneous chaos: $\hat{\fL} \mapsto \hat{\fL}^{\circ}$, while the second part is a consequence of integrating a perturbation of $\hat{\fL}^{\circ}$ by a shift operator: $\hat{\fL}^{\circ} \mapsto \bbE [\hat{\fL}^{\circ} \circ T_h]$.
\end{rem}

The proof presented here follows a Freidlin--Wentzell type strategy\footnote{See for instance \cite[Thm. 1.4.25]{deuschelLargeDeviations1989}.} and is heavily inspired by the proof of \cite[Thm. 3.5]{hairerLargeDeviationsWhiteNoise2015}. Let us give a rough sketch before we start: 

\begin{enumerate}
    \item Define an approximation $\hat{\fL}_m := \bbE \left[ \hat{\fL} \Big\vert \calF_m\right]$ of $\hat{\fL}$ by conditioning on basis elements of an ONB of $\sH$. 
    \item Define $\bmu^m := (\hat{\fL}_m)_{\ast} \mu$ and show that $\bmu$ and $(\bmu^m)_{m \in \bbN}$ are exponentially equivalent. This is done in Lemma \ref{lem:exp_equiv_after}, which is similar in spirit to \cite[Lem. 4]{frizLargeDeviationPrinciple2007} and \cite[Lem 3.9]{hairerLargeDeviationsWhiteNoise2015}.
    \item Show that for every $m \in \bbN$ the sequence of measures $(\bmu^m_{\varepsilon})_{\varepsilon > 0} := ( (\delta_{\varepsilon})_{\ast} \bmu^m)_{\varepsilon > 0}$ satisfies an LDP where the rate function $\sJ_{\!m}$ only depends on $\overline{\hat{\fL}_m}$. This is done in Lemma \ref{lem:LDP_of_each_approx}, which is similar in spirit to \cite[Lem 3.7]{hairerLargeDeviationsWhiteNoise2015}.
    \item Finally, show by hand that the rate functions $\sJ_{\!m}$ approximate \eqref{eq:LDP_for_AWMS} in the appropriate sense. This is done in Lemma \ref{lem:relation_among_rate_functions}, which is similar in spirit to \cite[Lem. 3.8]{hairerLargeDeviationsWhiteNoise2015}.
\end{enumerate}

Let $(e_k)_{k \in \mathbb{N}}$ be an ONB of $\sH$ contained in $E^{\ast}$ and define

\begin{equation}
    \calF_m := \sigma \big( \langle e_k, \cdot \rangle: 1 \leq k \leq m \big), \quad m \in \bbN, \label{eq:sigma_algebra}
\end{equation}

i.e. the $\sigma$-algebra on $E$ generated by the random variables $\{ \langle e_k, \cdot \rangle: 1 \leq k \leq m\}$. Furthermore, for each $m \in \bbN, \tau \in \calT$ let $\rmP_m:E \rightarrow E$ \label{symb:projection_onto_finite_basis} be the projection defined by 

    \begin{equation}
        \rmP_m(x) = \sum_{k = 1}^{m} \langle e_k, x \rangle e_k, \quad x \in E, 
    \end{equation}

and let $\sH_m := \rmP_m(\sH)$, $\mu^m := (\rmP_m)_{\ast} \mu$, and  

\begin{equation}
    \hat{\fL}_{\tau} := \pi_{\tau} \circ \hat{\fL}, \quad \hat{\fL}_m := \bbE \left[\hat{\fL} \Big\vert \calF_m \right], \quad \hat{\fL}_{m,\tau} := \pi_{\tau} \circ \bbE \left[ \hat{\fL} \Big\vert \calF_m \right] = \bbE \left[ \pi_{\tau} \circ \hat{\fL} \Big\vert \calF_m \right]. \label{eq:definition_of_fl_m}
\end{equation}

\begin{prop}[{AWMS for each $m \in \bbN$}] \label{prop:proxy_restriction_of_conditioning}
    Consider the context of Theorem \ref{thm:LDP_for_AWMS} and let $m \in \bbN$. Then the ambient space $(\calT, \bE, [\cdot], \calN)$, the enhanced measure $\bmu^m := (\hat{\fL}_m)_{\ast} \mu$, and the full lift $\hat{\fL}_m$ constitute Top-Down data in the sense of Theorem \ref{thm:top_down_construction} and induce an AWMS such that the underlying AWS is $(E, (\sH_m, \langle \cdot, \cdot \rangle_{\sH}), i\vert_{\sH_m}, \mu_m)$. In particular $\fL_m: \sH_m \rightarrow \bE$, obtained from the Top-Down construction, is injective and furthermore $\overline{\hat{\fL}_m}: \sH \rightarrow \bE$, the proxy-restriction of $\hat{\fL}_m$ to $\sH$, satisfies $\fL_m \circ \rmP_m = \overline{\hat{\fL}_m}$. 
\end{prop}
\begin{proof}
    Recall that by standard AWS theory $\id_E = \lim_{n \rightarrow \infty}^{L^2(E, \mu; E)} \sum_{k = 1}^n \langle e_i, \cdot \rangle e_i$. Then for any $\alpha \in A_m^{\leq \infty}$ (with $A_m^{\leq \infty}$ as defined in \eqref{symb:Akn} on p.~\pageref{symb:Akn}) we have 

    \begin{equation}
        \bbE [\id_E H_{\alpha}] = \sum_{k = 1}^{m} \bbE[\langle e_k, \cdot \rangle H_{\alpha}] e_k + \sum_{k = m+1}^{\infty} \underbrace{\bbE[\langle e_k, \cdot \rangle H_{\alpha}]}_{\substack{= 0 \\ \text{since} ~ k > m}} e_k =  \sum_{k = 1}^{m} \bbE[\langle e_k, \cdot \rangle H_{\alpha}] e_k = \bbE [\rmP_m H_{\alpha}] \label{eq:exp_of_identity_against_Hermite}
    \end{equation}

    and   
    
    \begin{equation}
        H_{\alpha} \circ \rmP_m = \prod_{i \in \bbN} h_{\alpha_i} \left( \left\langle e_i, \sum_{k = 1}^{m} \langle e_k, \cdot \rangle e_k \rangle \right\rangle \right) = \prod_{1 \leq i \leq m} h_{\alpha_i} \left( \left\langle e_i, \cdot \right\rangle \right) = H_{\alpha} . \label{eq:Hermite_after_projection}
    \end{equation}

    Therefore

    \begin{equation}
        \hat{\fL}_m \circ \rmP_m = \sum_{\tau \in \calT} \sum_{\alpha \in A^{\leq[\tau]}_{m}} \bbE [\hat{\fL}_{\tau} H_{\alpha}] \underbrace{H_{\alpha} \circ \rmP_m}_{= H_{\alpha}} = \hat{\fL}_m \label{eq:fL_m_after_projection}
    \end{equation} 
    
    and, using the fact that $\hat{\fL}$ is a lift on $E$,

    \begin{align}
        &\pi \circ \hat{\fL}_m \circ \rmP_m = \pi \left( \sum_{\tau \in \calT} \sum_{\alpha \in A^{\leq[\tau]}_{m}} \bbE [\hat{\fL}_{\tau} H_{\alpha}] (H_{\alpha} \circ \rmP_m) \right) \\
        &= \sum_{\tau \in \calT} \sum_{\alpha \in A^{\leq[\tau]}_{m}} \bbE [(\pi \circ \hat{\fL}_{\tau}) H_{\alpha}] (H_{\alpha} \circ \rmP_m) = \sum_{\alpha \in A^{\leq 1}_{m}} \underbrace{\bbE [\id_E H_{\alpha}]}_{=\bbE [\rmP_m H_{\alpha}]} \underbrace{H_{\alpha} \circ \rmP_m}_{= H_{\alpha}} = \rmP_m . \label{eq:projection_after_fL_m}
    \end{align}    

    Therefore, using \eqref{eq:fL_m_after_projection} and \eqref{eq:projection_after_fL_m} we obtain 

    \begin{equation}
        \pi_{\ast} \bmu^m = \left( \pi \circ \hat{\fL}_m \right)_{\ast} \mu = \left( \pi \circ \hat{\fL}_m \circ \rmP_m \right)_{\ast} \mu = (\rmP_m)_{\ast} \mu = \mu^m 
    \end{equation}

    and, using \eqref{eq:fL_m_after_projection},

    \begin{equation}
        \bmu^m = (\hat{\fL}_m)_{\ast} \mu = (\hat{\fL}_m \circ \rmP_m)_{\ast} \mu = (\hat{\fL}_m)_{\ast} \mu^m .
    \end{equation}

    Hence we have Top-Down data and the underlying AWS is then necessarily given as claimed. In particular, $\fL_m: \sH_m \rightarrow \bE$ is an $\sH_m$-skeleton lift and thus injective.

   Finally, let $h \in \sH$. Then via \eqref{eq:Hermite_after_projection}

    \begin{equation}
        \int_E H_{\alpha} \circ T_{\rmP_m(h)} \dd \mu = \int_E H_{\alpha}(\cdot + \rmP_m(h)) \dd \mu = \int_E \underbrace{H_{\alpha}(\rmP_m(\cdot + h))}_{= H_{\alpha}(\cdot + h)} \dd \mu = \int_E H_{\alpha} \circ T_{h} \dd \mu,
    \end{equation}

    and therefore 

    \begin{align}
        \fL_m ( \rmP_m(h)) &= \sum_{\tau \in \calT} \sum_{\alpha \in A^{[\tau]}_{m}} \bbE [\hat{\fL}_m H_{\alpha}] \int_E H_{\alpha} \circ T_{\rmP_m(h)} \dd \mu \\
        &= \sum_{\tau \in \calT} \sum_{\alpha \in A^{[\tau]}_{m}} \bbE [\hat{\fL}_m H_{\alpha}] \int_E H_{\alpha} \circ T_{h} \dd \mu = \overline{\hat{\fL}_m} (h) .
        \end{align}
\end{proof}

\begin{lem}[{Exponential Equivalence}] \label{lem:exp_equiv_after}
   In the context of Theorem \ref{thm:LDP_for_AWMS}

   \begin{equation}
       \limsup_{m \rightarrow \infty} \limsup_{\varepsilon \downarrow 0} \varepsilon^2 \log \mu \left( \vertiii{\delta_{\varepsilon} \circ \hat{\fL} - \delta_{\varepsilon} \circ \hat{\fL}_m} > \eta \right) = - \infty, \quad \eta > 0.
   \end{equation}
\end{lem}
\begin{proof}
    Let $N := \max\{ [\tau] : \tau \in \calT\}$ and define

    \begin{equation}
        \alpha_m := \left\Vert \vertiii{\hat{\fL} - \hat{\fL}_m} \right\Vert_{L^{2N}(E, \mu; \bbR)} .
    \end{equation}

    Via the triangle inequality we obtain 

    \begin{align}
    \alpha_m = \left\Vert \vertiii{\hat{\fL} - \hat{\fL}_m} \right\Vert_{L^{2N}(E, \mu; \bbR)} \leq \sum_{\tau \in \calT} \Vert \pi_{\tau} (\hat{\fL} - \hat{\fL}_m) \Vert_{L^{\frac{2N}{[\tau]}}(E, \mu; E_{\tau})}^{\frac{1}{[\tau]}}. \label{eq:sum_martingale}
    \end{align}
    
    Since $\hat{\fL}_m = \bbE [\hat{\fL} \vert \calF_m]$, the vector valued discrete martingale $L^p$-convergence theorem (Proposition \ref{prop:convergence_of_conditional_exp}) shows that each summand in \eqref{eq:sum_martingale} converges to $0$ (note that $p = \frac{2N}{[\tau]} \in [1, \infty)$). Since the indexing set of the sum is finite we obtain $\lim_{m \rightarrow \infty} \alpha_m = 0$. 
    
    By equivalence of the $p$-norms in the homogeneous distance (see Lemma \ref{lem:equiv_of_norm_homogeneous})
    
    \begin{equation}
    \left\Vert \vertiii{\hat{\fL} - \hat{\fL}_m} \right\Vert_{L^q(E, \mu; \bbR)} \leq C'(N) \sqrt{q} \alpha_m, \quad \forall q \geq 2N. 
    \end{equation}
    
    Let $\eta > 0$. Then via the Chebychev inequality with $\cdot \mapsto (\cdot)^q$ and \eqref{eq:homogeneity_of_hom_norm} we can estimate 
    
    \begin{align*}
    &\mu\left( \vertiii{\delta_{\varepsilon} \circ \hat{\fL} - \delta_{\varepsilon} \circ \hat{\fL}_m} > \eta \right) = \mu\left( \vertiii{\hat{\fL} - \hat{\fL}_m} > \frac{\eta}{\varepsilon}
    \right) \\
    &\lesssim \left( \frac{\eta}{\varepsilon} \right)^{-q} \sqrt{q}^{q} \alpha_m^{q} = \exp \left( q \log \left( \frac{\varepsilon}{\eta} \alpha_m \sqrt{q} \right) \right) . 
    \end{align*}
    
    Choosing $q := 1/\varepsilon^2$ and for $\frac{1}{\sqrt{2N}} \geq \varepsilon$
    
    \begin{equation}
        \varepsilon^2 \log \mu\left( \vertiii{\delta_{\varepsilon} \circ \hat{\fL} - \delta_{\varepsilon} \circ \hat{\fL}_m} > \eta \right) \leq \varepsilon^2 q \log \left( \frac{\varepsilon}{\eta} \alpha_m \sqrt{q} \right) = \log \left( \frac{\alpha_m}{\eta} \right) .
    \end{equation}
    
    Taking the $\limsup_{\varepsilon \downarrow 0}$, subsequently $\lim_{m \rightarrow \infty}$, and recalling $\alpha_m \rightarrow 0$ completes the proof. 
\end{proof}

\begin{lem}[{LDP for each Approximation}] \label{lem:LDP_of_each_approx}
    Fix $m \in \bbN$, let $\fL_m := \overline{\hat{\fL}_m}$ and let $\sbH_m = \fL_m(\sH_m)$ denote the enhanced Cameron--Martin space constructed in Proposition \ref{prop:proxy_restriction_of_conditioning}. In the context of Theorem \ref{thm:LDP_for_AWMS} the family of measures $(\bmu^m_{\varepsilon})_{\varepsilon > 0} = ( (\delta_{\varepsilon} \circ \hat{\fL}_m )_{\ast}\mu )_{\varepsilon > 0}$ satisfies an LDP with good rate function 
    
    \begin{equation}
        \sJ_{\!m}(\bx) = \inf \left\{ \frac{1}{2} \Vert h \Vert_{\sH}^2 : h \in \sH, \fL_m(h) = \bx \right\} = \begin{cases}
            \frac{1}{2} \Vert \pi (\bx) \Vert_{\sH}^2 & \bx \in \sbH_m \\
            + \infty & \text{else}
        \end{cases}, \label{eq:LDP_of_each_approx}
    \end{equation}

    on $\bE$, where $\hat{\fL}_m$ is as defined in \eqref{eq:definition_of_fl_m} and $\overline{(\cdot)}$ denotes the proxy-restriction as in Definition \ref{defn:proxy_restriction}.
\end{lem}

\begin{proof}[Proof of Lemma \ref{lem:LDP_of_each_approx}]
    Let $m \in \mathbb{N}$ be arbitrary. We want to use Theorem \ref{thm:HW_extended_contraction_principle}. That is, we consider 
    
    \begin{enumerate}[(i)]
        \item the spaces $(E, \Vert \cdot \Vert_E)$ and $(\bE, \vertiii{\cdot})$,
        \item the sequence of probability measures $(\mu_{\varepsilon})_{\varepsilon > 0}$, satisfying an LDP on $E$ with good rate function $\sI$, and
        \item the family of functions $(\Psi_{\varepsilon, m})_{\varepsilon \geq 0}: (E, \Vert \cdot \Vert_E) \rightarrow (\bE, \vertiii{\cdot})$ to be defined below.    
    \end{enumerate}

    For (iii), recall that the measures we want to derive an LDP for are\footnote{Since $\hat{\fL}_m$ is only a representative of a $\mu$-equivalence class and $\mu$ and $\mu_{\varepsilon}$ are mutually singular for any $\varepsilon > 0$ (see e.g. \cite[Rem. 2.10]{dapratoIntroductionInfinitedimensionalAnalysis2006}), the expression $\hat{\fL}_m \circ \rmm_{\varepsilon^{-1}}$ is a priori not well defined, i.e. $[\hat{\fL}_m \circ \rmm_{\varepsilon^{-1}}]_{\mu} \label{symb:equiv_class}$ is not independent of the choice of representative for $[\hat{\fL}_m]_{\mu}$. However, $[\hat{\fL}_m \circ \rmm_{\varepsilon^{-1}}]_{\mu_{\varepsilon}}$ is independent of such a choice, and therefore (as also \eqref{eq:pushforward_rewriting} shows), the measure $\left( \delta_{\varepsilon} \circ \hat{\fL}_m \circ \rmm_{\varepsilon^{-1}} \right)_{\ast} \mu_{\varepsilon}$ is well defined i.e. independent of the $\mu$-representative of $\hat{\fL}_m$.} 
    
    \begin{align}
        \left( \delta_{\varepsilon} \circ \hat{\fL}_m \right)_{\ast} \mu &= \left( \delta_{\varepsilon} \circ \hat{\fL}_m \circ \rmm_{\varepsilon^{-1}} \circ \rmm_{\varepsilon} \right)_{\ast} \mu = \left( \delta_{\varepsilon} \circ \hat{\fL}_m \circ \rmm_{\varepsilon^{-1}} \right)_{\ast} \mu_{\varepsilon}, \quad \varepsilon > 0. \label{eq:pushforward_rewriting}
    \end{align}

    Thus, in order to apply Theorem \ref{thm:HW_extended_contraction_principle} we need versions of the measurable functions $(\delta_{\varepsilon} \circ \hat{\fL}_m \circ \rmm_{\varepsilon^{-1}})_{\varepsilon > 0}$ which are continuous on neighborhoods of $\sH = \{ x \in E : \sI(x) < \infty \}$. For this, consider the following computation: 
    
    Let $\varepsilon > 0$ be arbitrary and recall the notation $A^k_n$ and $A^{\leq k}_n$ from \eqref{symb:Akn} on p.~\pageref{symb:Akn}. Then using Proposition \ref{prop:conditional_expectation_of_finite_WIC} for the conditional expectation

   \begin{align}
       \delta_{\varepsilon} \circ \bbE \left[ \hat{\fL} \Big\vert \calF_m \right] \circ \rmm_{\varepsilon^{-1}} 
       &= \delta_{\varepsilon} \left( \sum_{\tau \in \calT} \bbE \left[ \hat{\fL}_{\tau} \Big\vert \calF_m \right] \circ \rmm_{\varepsilon^{-1}} \right) \label{eq:cond_exp1}\\
       &= \delta_{\varepsilon} \left( \sum_{\tau \in \calT} \sum_{\alpha \in A^{\leq [\tau]}_m} \bbE \left[ \hat{\fL}_{\tau} H_{\alpha} \right] \left( H_{\alpha} \circ \rmm_{\varepsilon^{-1}} \right) \right) \label{eq:cond_exp2}\\
       &= \sum_{\tau \in \calT} \sum_{0 \leq k \leq [\tau]} \varepsilon^{[\tau] - k}
       \sum_{\alpha \in A^{k}_m} \bbE \left[ \hat{\fL}_{\tau} H_{\alpha} \right] \underbrace{\varepsilon^k \left( H_{\alpha} \circ \rmm_{\varepsilon^{-1}} \right)}_{=:I(\varepsilon)} \label{eq:TD_LDP_I} \\
        &= \underbrace{\sum_{\tau \in \calT} \sum_{\alpha \in A^{[\tau]}_m} \bbE \left[ \hat{\fL}_{\tau} H_{\alpha} \right] \varepsilon^{[\tau]} \left( H_{\alpha} \circ \rmm_{\varepsilon^{-1}} \right)}_{=:II(\varepsilon)} \\
        &+  \underbrace{\sum_{\tau \in \calT} \sum_{0 \leq k < [\tau]} \varepsilon^{[\tau] - k}
       \sum_{\alpha \in A^{k}_m} \bbE \left[ \hat{\fL}_{\tau} H_{\alpha} \right] \varepsilon^k \left( H_{\alpha} \circ \rmm_{\varepsilon^{-1}} \right)}_{=:III(\varepsilon)} =: \Psi_{\varepsilon,m}. \label{eq:definition_of_Psi_eps_m}
   \end{align}

   From \eqref{eq:cond_exp1} to \eqref{eq:cond_exp2} we chose a specific representative of the conditional expectation, which is continuous on all of $E$. That specific continuous representative shall define $\Psi_{\varepsilon,m}$, giving (iii) for $\varepsilon > 0$. \\

   In order to compute the limit of $\Psi_{\varepsilon, m}$ as $\varepsilon \rightarrow 0$ (and hence show (iv)) consider $I(\varepsilon)$ in \eqref{eq:TD_LDP_I}. For $\alpha \in A^{k}_m$ and $i \geq 0$ expand $h_{\alpha_i}(x) = \sum_{l = 0}^{\alpha_i} c_l x^l$ and recall that by our convention $c_{\alpha_i} = 1$. Then

   \begin{align}
       \varepsilon^k \left( H_{\alpha} \circ \rmm_{\varepsilon^{-1}} \right) (x) &= \varepsilon^k \prod_{i \in \bbN} h_{\alpha_i} (\langle e_i, \varepsilon^{-1} x \rangle) = \prod_{i \in \bbN} \varepsilon^{\alpha_i} h_{\alpha_i} (\langle e_i, \varepsilon^{-1} x \rangle) \\
       &= \prod_{i \in \bbN} \varepsilon^{\alpha_i} \sum_{l = 0}^{\alpha_i} c_l \langle e_i, \varepsilon^{-1} x \rangle^l = \prod_{i \in \bbN} \sum_{l = 0}^{\alpha_i} \varepsilon^{\alpha_i - l} c_l \langle e_i, x \rangle^l \rightarrow \prod_{i \in \bbN} \langle e_i, x \rangle^{\alpha_i},
   \end{align}

   for $\varepsilon \rightarrow 0$. Since all involved products and sums are finite and all functions are continuous, the convergence is uniformly on bounded subsets of $E$. This further implies that $III(\varepsilon) \rightarrow 0$ for $\varepsilon \rightarrow 0$, uniformly on bounded subsets of $E$ and 

   \begin{equation}
       II(\varepsilon) \rightarrow \sum_{\tau \in \calT} \sum_{\alpha \in A^{[\tau]}_m} \bbE \left[ \hat{\fL}_{\tau} H_{\alpha} \right] \prod_{i \in \bbN} \langle e_i, \cdot \rangle^{\alpha_i}, \quad \varepsilon \rightarrow 0,
   \end{equation}

   also uniformly on bounded subsets of $E$. In conclusion

   \begin{equation}
       \Psi_{\varepsilon,m} \rightarrow \sum_{\tau \in \calT} \sum_{\alpha \in A^{[\tau]}_m} \bbE \left[ \hat{\fL}_{\tau} H_{\alpha} \right] \prod_{i \in \bbN}  \langle e_i, \cdot \rangle^{\alpha_i} =: \Psi_{0,m}, \quad \varepsilon \rightarrow 0, \label{eq:definition_of_Psi_eps_0}
   \end{equation}

    uniformly on bounded subsets of $E$. In particular, for every $C \in \bbR$ the convergence is uniform on $B^E(r_C) := \{x \in E: \frac{1}{2} \Vert x \Vert_E^2 \leq r_C\}$, where $r_C > 0$ is large enough such that $B^E(r_C)$ is a neighborhood of $\{x \in E: \frac{1}{2} \Vert x \Vert_{\sH}^2 \leq C\}$. Such an $r_C > 0$ exists since $\Vert \cdot \Vert_E \lesssim \Vert \cdot \Vert_{\sH}$. This shows (iv) (and gives (iii) for $\varepsilon = 0$). \\

    Hence, \eqref{eq:pushforward_rewriting} and Theorem \ref{thm:HW_extended_contraction_principle} imply that the family $\left(\left(\delta_{\varepsilon} \circ \hat{\fL}_m \right)_{\ast} \mu \right)_{\varepsilon > 0}$ satisfies an LDP with good rate function 

   \begin{equation}
       \sJ_{\!m}(\bx) = \inf \{ \sI(x) : x \in E, \Psi_{0,m}(x) = \bx \} . \label{eq:rate_Psi_0}
   \end{equation}

    The final step is to show that this agrees with \eqref{eq:LDP_of_each_approx}. Substituting the generalized Schilder rate function 
    
    \begin{equation}
    \sI(x) = \begin{cases}
    \frac{1}{2} \Vert x \Vert^2_{\sH} & x \in \sH \\
    + \infty & \text{else}.
    \end{cases} \label{eq:Schilder_rate}
    \end{equation}
    
    (see e.g. \cite[Thm. 3.4.12]{deuschelLargeDeviations1989}) into \eqref{eq:rate_Psi_0} yields 
    
   \begin{equation}
       \sJ_{\!m}(\bx) = \inf \left\{ \frac{1}{2} \Vert h \Vert_{\sH}^2 : h \in \sH, \Psi_{0,m}(h) = \bx \right\}  \label{eq:rate_subst}
   \end{equation}

    Hence, only the values of $\Psi_{0,m}$ on elements of the Cameron--Martin space $\sH$ are relevant. Indeed, for any $h \in \sH$ we have 

    \begin{align}
        \overline{\hat{\fL}_m}(h) &= \sum_{\tau \in \calT} \overline{\bbE \left[ \hat{\fL}_{\tau} \vert \calF_m \right]} (h) = \sum_{\tau \in \calT} \overline{\sum_{\alpha \in A^{\leq [\tau]}_m} \bbE \left[ \hat{\fL}_{\tau} H_{\alpha} \right]  H_{\alpha}} (h) \\
        &= \sum_{\tau \in \calT}  \sum_{\alpha \in A^{\leq [\tau]}_m} \bbE \left[ \hat{\fL}_{\tau} H_{\alpha} \right] \int_E ( \Pi_{[\tau]} H_{\alpha}) \circ T_h \dd \mu = \sum_{\tau \in \calT}  \sum_{\alpha \in A^{[\tau]}_m} \bbE \left[ \hat{\fL}_{\tau} H_{\alpha} \right] \int_E H_{\alpha} \circ T_h \dd \mu , \label{eq:LDP_for_each_pre_pre_end}
   \end{align}

    where in the last line we used that for any $\alpha \in A^{\leq [\tau]}_m$ we have that 
    
    \begin{equation}
        \Pi_{[\tau]} H_{\alpha} = \begin{cases}
            H_{\alpha} & \vert \alpha \vert := \sum_{i \in \bbN} \alpha_i = [\tau] \\
            0 & \vert \alpha \vert \neq [\tau] ,
        \end{cases}
    \end{equation}
    
    since $H_{\alpha} \in \calP^{(\vert \alpha \vert)}(E, \mu; \bbR)$ by definition. Fix now $\tau \in \calT$ and $\alpha \in A^{\leq [\tau]}_m$. Then by virtue of the Binomial theorem for Hermite polynomials (see Proposition \ref{prop:hermite_polynomial_binomial})

    \begin{align}
        \int_E H_{\alpha} \circ T_h \dd \mu &= \bbE \left[ \prod_{i \in \bbN} h_{\alpha_i}(\langle e_i, \cdot + h \rangle) \right] = \bbE \left[ \prod_{i \in \bbN} \sum_{l = 0}^{\alpha_i} \binom{\alpha_i}{l} h_l(\langle e_i, \cdot \rangle) \langle e_i, h \rangle^{\alpha_i - l} \right] . \label{eq:LDP_TD_switching_prod_and_sum}
    \end{align}

    (The same way as on p.~\pageref{page:prod_sum_switching}), define $S_{\alpha} := \times_{i \in \bbN} \{0, \ldots, \alpha_i\}$ and note that the cardinality of $S_{\alpha}$ is finite. Switch the product and the sum in \eqref{eq:LDP_TD_switching_prod_and_sum} and pull out the finite sum to obtain

    \begin{equation}
        =\bbE \left[ \sum_{\sigma \in S_{\alpha}} \prod_{i \in \bbN} \binom{\alpha_i}{\sigma_i} h_{\sigma_i}(\langle e_i, \cdot \rangle) \langle e_i, h \rangle^{\alpha_i - \sigma_i} \right] = \sum_{\sigma \in S_{\alpha}} \bbE \left[\prod_{i \in \bbN} \binom{\alpha_i}{\sigma_i} h_{\sigma_i}(\langle e_i, \cdot \rangle)\langle e_i, h \rangle^{\alpha_i - \sigma_i} \right] .
    \end{equation}

    The product can be pulled outside of the integral since the sequence $(\langle e_i, \cdot \rangle)_{i \in \bbN}$ is independent w.r.t. $\mu$, giving
    
    \begin{equation}
        = \sum_{\sigma \in S_{\alpha}} \prod_{i \in \bbN} \binom{\alpha_i}{\sigma_i} \underbrace{\bbE \left[h_{\sigma_i}(\langle e_i, \cdot \rangle)\right]}_{= 1_{\sigma_i = 0}} \langle e_i, h \rangle^{\alpha_i - \sigma_i} = \prod_{i \in \bbN} \langle e_i, h \rangle^{\alpha_i}, \label{eq:LDP_for_each_pre_end}
    \end{equation}
    
   and therefore, by inserting \eqref{eq:LDP_for_each_pre_end} back into \eqref{eq:LDP_for_each_pre_pre_end},

   \begin{equation}
        \fL_m(h) = \overline{\hat{\fL}_m}(h) = \sum_{\tau \in \calT} \sum_{\alpha \in A_m^{[\tau]}} \bbE \left[ \hat{\fL}_{\tau} H_{\alpha} \right] \prod_{i \in \bbN} \langle e_i, h \rangle^{\alpha_i} = \Psi_{0,m}(h), \quad h \in \sH. \label{eq:LDP_TD_equality_for_Psi_0m}
   \end{equation}
  
    Thus, substituting this into \eqref{eq:rate_subst}, we conclude that the family $\left( \left( \delta_{\varepsilon} \circ \hat{\fL}_m \right)_{\ast} \mu \right)_{\varepsilon > 0}$ satisfies an LDP with good rate function 
     
   \begin{equation}
    \sJ_{\!m}(\bx) = \inf \left\{ \frac{1}{2} \Vert h \Vert_{\sH}^2 : h \in \sH, \fL_m(h) = \bx \right\} = \inf \left\{ \frac{1}{2} \Vert h \Vert_{\sH}^2 : h \in \sH_m, \fL_m(h) = \bx \right\},
    \end{equation}

    where in the last equality we used $\fL_m \circ \rmP_m = \fL_m$, shown in Proposition \ref{prop:proxy_restriction_of_conditioning}. Since by Proposition \ref{prop:proxy_restriction_of_conditioning} $\fL_m$ is an $\sH_m$-skeleton lift (and thus injective on $\sH_m$), the infimum above is either over the empty set or over a set with a single element: $\pi(\bx)$. Hence, using $\fL_m(\sH_m) = \sbH_m$, we may rewrite 

    \begin{equation}
        \sJ_{\!m}(\bx) = \begin{cases}
            \frac{1}{2} \Vert \pi (\bx) \Vert_{\sH}^2 & \bx \in \sbH_m \\
            + \infty & \text{else}.
        \end{cases}
    \end{equation}   
\end{proof}

\begin{lem}[{Relation Among Rate Functions}] \label{lem:relation_among_rate_functions}
    Let $\sJ: \bE \rightarrow [0,\infty]$ be defined by 
    
   \begin{equation}
    \sJ(\bx) = \inf \left\{ \frac{1}{2} \Vert h \Vert_{\sH}^2 : h \in \sH, \overline{\hat{\fL}}(h) = \bx \right\} = \begin{cases}
            \frac{1}{2} \Vert \pi (\bx) \Vert_{\sH}^2 & \bx \in \sbH \\
            + \infty & \text{else}
        \end{cases}  \label{eq:relation_among_rate_functions}
   \end{equation}
    
    and let $(\sJ_{\!m})_{m \in \mathbb{N}}$ be the sequence of good rate functions resulting from Lemma \ref{lem:LDP_of_each_approx}. Then (i) $\sJ$ is a good rate function and (ii) for any closed set $A \subseteq \bE$ we have 

    \begin{equation}
        \lim_{\eta \downarrow 0} \liminf_{m \rightarrow \infty} \inf_{ \bx \in A_{\eta}} \sJ_{\!m}(\bx) = \inf_{\bx \in A} \sJ(\bx),
    \end{equation}

    where $A_{\eta} := \{ \bx \in \bE : \inf_{\by \in A} \vertiii{\by - \bx} \leq \eta \}$ is the $\eta$-fattening of $A$. 
\end{lem}
\begin{proof}
    Extend $\overline{\hat{\fL}}$ (which is defined on $\sH$) to a measurable function on $E$ by defining

    \begin{equation}
        \overline{\hat{\fL}_{\infty}}(x) := \begin{cases}
            \overline{\hat{\fL}}(x) & x \in \sH \\
            0 & \text{else}
        \end{cases}
    \end{equation}

    which is a measurable function $E \rightarrow \bE$, and let $\Psi_{0,m}$ be as defined in \eqref{eq:definition_of_Psi_eps_0}. Then using \eqref{eq:LDP_TD_equality_for_Psi_0m}, the vector valued discrete martingale $L^p$-convergence theorem (Proposition \ref{prop:convergence_of_conditional_exp}), and Lemma \ref{lem:L2_convergence_implies_pointwise_convergence_of_homogeneous_part}, we deduce that 

    \begin{align}
        \vertiii{\overline{\hat{\fL}_{\infty}}(h) - \Psi_{0,m}(h)} = \vertiii{\overline{\hat{\fL}}(h) - \overline{\hat{\fL}_m}(h)} \rightarrow 0
    \end{align}

    uniformly on bounded sets of $\sH$ and thus uniformly on sub-level sets of $\sI$. Hence \cite[Lem. 2.1.4]{deuschelLargeDeviations1989} implies (i) and (ii) follows from the proof of \cite[Lem. 2.1.4]{deuschelLargeDeviations1989}.

    To see the equality in \eqref{eq:relation_among_rate_functions}, recall that $\fL = \overline{\hat{\fL}}$ is an $\sH$-lift and thus injective. Hence the set over which the infimum is taken either consists only of $\pi(\bx)$ or is empty. This, together with $\sbH = \fL(\sH)$ gives the equality.   
\end{proof}

\begin{proof}[Proof of Theorem \ref{thm:LDP_for_AWMS}]
    \textit{Upper bound for closed sets}: Let $A \subseteq \bE$ be closed. Then for every $\varepsilon >0, m \in \bbN, \eta > 0$

    \begin{equation}
       \mu \left( \delta_{\varepsilon} \circ \hat{\fL} \in A \right) \leq \mu \left( \delta_{\varepsilon} \circ \hat{\fL}_m \in A_{\eta} \right) + \mu \left( \vertiii{\delta_{\varepsilon} \circ \hat{\fL} - \delta_{\varepsilon} \circ \hat{\fL}_m} \geq \eta \right),
    \end{equation}
    
    where $A_{\eta} := \{ \bx \in \bE : \inf_{\by \in A} \vertiii{\by - \bx} \leq \eta \}$ is the $\eta$-fattening of $A$. By Lemma \ref{lem:log_equiv}

    \begin{align}
        \limsup_{\varepsilon \rightarrow 0} \varepsilon^2 \log \mu \left( \delta_{\varepsilon} \circ \hat{\fL} \in A \right) &\leq \limsup_{\varepsilon \rightarrow 0} \varepsilon^2 \log \mu \left( \delta_{\varepsilon} \circ \hat{\fL}_m \in A_{\eta} \right) \\ 
        &\vee \limsup_{\varepsilon \rightarrow 0} \varepsilon^2 \log \mu \left( \vertiii{\delta_{\varepsilon} \circ \hat{\fL} - \delta_{\varepsilon} \circ \hat{\fL}_m} \geq \eta \right) .
    \end{align}
    
    By Lemma \ref{lem:LDP_of_each_approx}, since $A_{\eta}$ is closed, for any $m \in \mathbb{N}$

    \begin{equation}
        \limsup_{\varepsilon \rightarrow 0} \varepsilon^2 \log \mu \left( \delta_{\varepsilon} \circ \hat{\fL}_m \in A_{\eta} \right) \leq - \inf_{x \in A_{\eta}} \sJ_{\!m}(\bx) .
    \end{equation}

    If $\inf_{\bx \in A} \sJ(\bx) > - \infty$, then by Lemma \ref{lem:exp_equiv_after} there exists an $m(A) \in \mathbb{N}$ s.t. 

    \begin{equation}
        \limsup_{\varepsilon \rightarrow 0} \varepsilon^2 \log \mu \left( \vertiii{\delta_{\varepsilon} \circ \hat{\fL} - \delta_{\varepsilon} \circ \hat{\fL}_m} \geq \eta \right) \leq - \inf_{\bx \in A} \sJ(\bx), \quad m \geq m(A), \label{eq:using_exp_equiv_after}
    \end{equation}

    while if $\inf_{\bx \in A} \sJ(\bx) = - \infty$ \eqref{eq:using_exp_equiv_after} holds anyway. Therefore in conclusion

    \begin{align}
        \limsup_{\varepsilon \rightarrow 0} \varepsilon^2 \log \mu \left( \delta_{\varepsilon} \circ \hat{\fL} \in A \right) &\leq \lim_{\eta \downarrow 0} \liminf_{m \rightarrow \infty} \left( - \inf_{x \in A_{\eta}} \sJ_{\!m}(\bx)\right) \vee \left( - \inf_{\bx \in A} \sJ(\bx) \right) \\
        &= \left( - \lim_{\eta \downarrow 0} \limsup_{m \rightarrow \infty} \inf_{x \in A_{\eta}} \sJ_{\!m}(\bx) \right) \vee \left( - \inf_{\bx \in A} \sJ(\bx) \right) \\
        &\leq \left( - \lim_{\eta \downarrow 0} \liminf_{m \rightarrow \infty} \inf_{x \in A_{\eta}} \sJ_{\!m}(\bx) \right) \vee \left( - \inf_{\bx \in A} \sJ(\bx) \right) \\
        &= - \inf_{\bx \in A} \sJ(\bx),
    \end{align}

    where the last equality is due to Lemma \ref{lem:relation_among_rate_functions}(ii). \\

    \textit{Lower bound on open sets}: Let $U \subseteq \bE$ be open and let $\bx \in U$ be arbitrary. Choose $\eta > 0$ s.t. $\overline{B}(\bx, 2 \eta) := \{ \by \in \bE : \vertiii{\bx - \by} \leq 2 \eta \} \subseteq U$. Then for any $m \in \mathbb{N}$ we have 

    \begin{align}
    \mu \left( \vertiii{\bx - \delta_{\varepsilon} \circ \hat{\fL}_m} < \eta \right) &\leq \mu \left( \left\{ \vertiii{\bx - \delta_{\varepsilon} \circ \hat{\fL}_m} < \eta \right\} \cap \left\{ \delta_{\varepsilon} \circ \hat{\fL} \in U \right\} \right) \\
    &+ \mu \left( \left\{\vertiii{\bx - \delta_{\varepsilon} \circ \hat{\fL}_m} < \eta \right\} \cap \left\{ \delta_{\varepsilon} \circ \hat{\fL} \not\in U \right\} \right) \\
    &\leq \mu \left(\delta_{\varepsilon} \circ \hat{\fL} \in U \right) + \mu \left( \vertiii{\delta_{\varepsilon} \circ \hat{\fL} - \delta_{\varepsilon} \circ \hat{\fL}_m} \geq \eta \right).
    \end{align}

    Thus by Lemma \ref{lem:log_equiv} we obtain 

    \begin{align}
        \liminf_{\varepsilon \rightarrow 0} \varepsilon^2 \log \mu \left( \vertiii{\bx - \delta_{\varepsilon} \circ \hat{\fL}_m} < \eta \right) &\leq \liminf_{\varepsilon \rightarrow 0} \varepsilon^2 \log \mu \left( \delta_{\varepsilon} \circ \hat{\fL} \in U \right) \\
        &\vee \liminf_{\varepsilon \rightarrow 0} \varepsilon^2 \log  \mu \left( \vertiii{\delta_{\varepsilon} \circ \hat{\fL} - \delta_{\varepsilon} \circ \hat{\fL}_m} \geq \eta \right).
    \end{align}

    If $\liminf_{\varepsilon \rightarrow 0} \varepsilon^2 \log \mu(\delta_{\varepsilon} \circ \hat{\fL} \in U ) > - \infty$, then by Lemma \ref{lem:exp_equiv_after} there exists an $m(U) \in \bbN$ such that for every $m \geq m(U)$ 

    \begin{align}
        &\liminf_{\varepsilon \rightarrow 0} \varepsilon^2 \log \mu \left( \delta_{\varepsilon} \circ \hat{\fL} \in U \right) \vee \liminf_{\varepsilon \rightarrow 0} \varepsilon^2 \log  \mu \left( \vertiii{\delta_{\varepsilon} \circ \hat{\fL} - \delta_{\varepsilon} \circ \hat{\fL}_m} \geq \eta \right) \\
        &\leq \liminf_{\varepsilon \rightarrow 0} \varepsilon^2 \log \mu \left( \delta_{\varepsilon} \circ \hat{\fL} \in U \right) \vee \limsup_{\varepsilon \rightarrow 0} \varepsilon^2 \log  \mu \left( \vertiii{\delta_{\varepsilon} \circ \hat{\fL} - \delta_{\varepsilon} \circ \hat{\fL}_m} \geq \eta \right) \\       
        &= \liminf_{\varepsilon \rightarrow 0} \varepsilon^2 \log \mu \left( \delta_{\varepsilon} \circ \hat{\fL} \in U \right), \label{eq:1}
    \end{align}

    while if $\liminf_{\varepsilon \rightarrow 0} \varepsilon^2 \log \mu \left( \delta_{\varepsilon} \circ \hat{\fL} \in U \right) = - \infty$ \eqref{eq:1} holds anyway. Thus we only need a lower bound on $\liminf_{\varepsilon \rightarrow 0} \varepsilon^2 \log \mu \left( \vertiii{\bx - \delta_{\varepsilon} \circ \hat{\fL}_m} < \eta \right)$, asymptotically as $m \rightarrow \infty$. By Lemma \ref{lem:LDP_of_each_approx}, applied to the open set $B(\bx, \eta) := \{ \by \in \bE : \vertiii{\bx - \by} < \eta \}$, we obtain 

    \begin{equation}
          - \inf_{\by \in \overline{B}(\bx, \frac{\eta}{2})} \sJ_{\!m}(\by) \leq - \inf_{\by \in B(\bx, \eta)} \sJ_{\!m}(\by) \leq \liminf_{\varepsilon \rightarrow 0} \varepsilon^2 \log \mu \left( \vertiii{\bx - \delta_{\varepsilon} \circ \hat{\fL}_m} < \eta \right), \label{eq:intermediate_lower_bound}
    \end{equation}

    and by Lemma \ref{lem:relation_among_rate_functions} applied to the closed set $\{\bx\}$ we have 

    \begin{align}
        \lim_{\eta \downarrow 0} \limsup_{m \rightarrow \infty} \left(- \inf_{\by \in \overline{B}(\bx, \frac{\eta}{2})} \sJ_{\!m}(\by) \right) &= - \lim_{\eta \downarrow 0} \liminf_{m \rightarrow \infty} \inf_{\by \in \overline{B}(\bx, \frac{\eta}{2})} \sJ_{\!m}(\by) = - \sJ(\bx) .  \label{eq:applying_consistency_of_rate_function}
    \end{align}

    That is, taking the limit superior as $m \rightarrow \infty$ and then the limit $\eta \rightarrow 0$ in \eqref{eq:intermediate_lower_bound}, and combining \eqref{eq:applying_consistency_of_rate_function} with \eqref{eq:1} yields 

    \begin{equation}
          - \sJ(\bx) \leq \liminf_{\varepsilon \rightarrow 0} \varepsilon^2 \log \mu \left( \delta_{\varepsilon} \circ \hat{\fL} \in U \right) .
    \end{equation}

    Taking the supremum over all $\bx \in U$ on the left hand side yields the result. Thus the family of measures $(\bmu_{\varepsilon})_{\varepsilon > 0}$ satisfies an LDP with good rate function $\sJ$. 
\end{proof}
Regarding the above proof, see also \cite[Thm. 3.5]{hairerLargeDeviationsWhiteNoise2015}.

\begin{rem}
    While Definition \ref{defn:skeleton_lift} requires $\calT$-multi-linearity of $\fL$, this not strictly necessary in order to derive an LDP. What is needed is the weaker property of homogeneity (as in Proposition \ref{lem:homogeneity_of_fL}).
\end{rem}

\section{Fernique Estimate} \label{sec:Fernique}

Let $d \geq 1$ and let $X:(\Omega, \bbP) \rightarrow \bbR^d$ be a multivariate normally distributed $d$-dimensional random vector. Then the distribution of $X$ famously has Gaussian tails, i.e. there exist a constant $\eta_0 > 0$ such that

\begin{equation}
    \bbP \left( \Vert X \Vert \geq t \right) \lesssim \exp\left( - \eta_0 t^2 \right), \quad \forall t \geq 0.
\end{equation}

This exceptionally good integrability property is critical in Gaussian analysis, guaranteeing among other things the existence of momenta of all orders without being compactly supported. In the infinite dimensional setting a similar result holds: the celebrated theorem of X. Fernique.  

\begin{thm}[{Fernique's Theorem, see e.g. \cite[Thm. 2.8.5]{bogachevGaussianMeasures1998} or \cite[Thm. 1.3.24]{deuschelLargeDeviations1989}}] \label{thm:fernique}
    Let $(E, \sH, i, \mu)$ be an abstract Wiener space and let

    \begin{equation}
      \eta_0 := \inf \left\{ \frac{1}{2} \Vert h \Vert_{\sH}^2 : h \in \sH, \Vert h \Vert_E = 1 \right\}.
    \end{equation} 

    Then for any $\eta < \eta_0$

    \begin{equation}
        \mu \left( x \in E: \Vert x \Vert_E \geq t \right) \lesssim
 \exp\left( - \eta t^2 \right), \quad \forall t \geq 0,
    \end{equation}

    and in particular the random variable $x \mapsto \Vert x \Vert_E$ has Gaussian tails in the sense that 

    \begin{equation}
        \bbE_{\mu}\! \left[ \exp \left( \eta \Vert x \Vert_E^2 \right) \right] < \infty.
    \end{equation}
\end{thm}

As a consequence of the results in Section \ref{sec:LDP} one can show a similar statement for the enhanced measure $\bmu$ of an abstract Wiener model space; despite the fact that $\bmu$ is (expect for trivial cases) not Gaussian. Before that, recall the following lemma.

\begin{lem} \label{lem:decay_to_exp_int}
    Let $(\Omega, \bbP)$ be a probability space, let $X: \Omega \rightarrow \bbR$ be a random variable and let $\eta_0, t_0 > 0$ be s.t. 

    \begin{equation}
        \bbP \left( \vert X \vert \geq t \right) \lesssim e^{- \eta_0 t^2 }, \quad \forall t > t_0.
    \end{equation}

    Then $X$ has Gaussian tails in the sense that 

    \begin{equation}
        \bbE_{\bbP}\! \left[ \exp \left( \eta \vert X \vert^2 \right) \right] < \infty, \quad \forall \eta < \eta_0 .
    \end{equation}
\end{lem}
\begin{proof}
    Let $\eta < \eta_0$ be arbitrary. Then, using the layer cake representation, 

    \begin{align}
        \bbE_{\bbP}\! \left[ \exp \left( \eta \vert X \vert^2 \right) \right] 
        &= \underbrace{\int_0^{t_0} \bbP \left\{ e^{\eta \vert X \vert^2} \geq s \right\} \dd s}_{ < \infty} + \int_{t_0}^{\infty} \underbrace{\bbP \left\{ e^{\eta \vert X \vert^2} \geq s \right\}}_{ = \bbP \left\{ \vert X \vert \geq \sqrt{\frac{\log (s)}{\eta} } \right\}} \dd s .
    \end{align}

    By assumption, the right most term can be upper bounded up to a constant by 

    \begin{equation}
    \int_{t_0}^{\infty} e^{- \eta_0 \left( \sqrt{\frac{\log(s)}{\eta} }\right)^2 } \dd s = \int_{t_0}^{\infty} s^{- \frac{\eta_0}{\eta}} \dd s. 
    \end{equation}

    Since $1 < \frac{\eta_0}{\eta}$ the integral is finite, which proves the claim.
\end{proof}

\begin{thm}[Fernique Estimate for AWMS] \label{thm:exp_int}
        Let $((\calT, \bE, [\cdot], \calN), \sbH, \bmu, \fL, \hat{\fL})$ be an abstract Wiener model space. Let 
        
        \begin{equation}
            \eta_0 := \inf \left\{ \frac{1}{2} \Vert \pi (\bh) \Vert_{\sH}^2 : \bh \in \sbH, \vertiii{\bh} =1 \right\} .
        \end{equation}
        
        Then $0 < \eta_0$ and for any $0 \leq \eta < \eta_0$

        \begin{equation}
            \bmu \left( \bx \in \bE: \vertiii{\bx} \geq t \right) \lesssim \exp\left( - \eta t^2 \right), \quad \forall t \geq 0, \label{eq:exp_int}
        \end{equation}
        
        and in particular the random variable $\bx \mapsto \vertiii{\bx}$ has Gaussian tails in the sense that 

        \begin{equation}
        \bbE_{\bmu}\! \left[ \exp \left( \eta \vertiii{\bx}^2 \right) \right] = \int_{\bE} e^{\eta \vertiii{\bx}^2} \dd \bmu(\bx) < \infty.
        \end{equation}
\end{thm}
\begin{proof}
    In order to see that $\eta_0 > 0$, recall that $\fL$ is continuous on $\sH$ and that $\fL(0) = 0$ by homogeneity. Now assume $\eta_0 = 0$ and choose a minimizing sequence, i.e. a sequence $(\bh_n)_{n \in \bbN} \subseteq \sbH$ such that $\vertiii{\bh_n} = 1$ and $\frac{1}{2} \Vert \pi (\bh_n) \Vert_{\sH}^2 \rightarrow 0$. Then by continuity of $\fL$ and the fact that $\fL$ is a left-inverse of $\pi\vert_{\sbH}$ we conclude 
    
    \begin{equation}
        \vertiii{\bh_n} = \vertiii{\fL(\pi (\bh_n))} \rightarrow 0, \quad n \rightarrow \infty.
    \end{equation}
    
    A contradiction. \\
    In order to show \eqref{eq:exp_int}, applying the contraction principle along the map $\bx \mapsto \vertiii{\bx}$ to the LDP with rate function $\sJ$, derived in Theorem \ref{thm:LDP_for_AWMS}, gives

    \begin{equation}
        \limsup_{\varepsilon \downarrow 0} \varepsilon^2 \log \bmu \left( \vertiii{\delta_{\varepsilon} \bx} \geq 1 \right) \leq - \inf_{\lambda \in [1, \infty)} \inf_{\substack{\bx \in \bE \\ \vertiii{\bx} = \lambda}} \sJ(\bx),
    \end{equation} 

    where 

    \begin{align}
        \inf_{\lambda \in [1, \infty)} \inf_{\substack{\bx \in \bE \\ \vertiii{\bx} = \lambda}} \sJ(\bx) &= \inf \left\{ \sJ(\bx) : \vertiii{\bx} \geq 1 \right\} \\
        &= \inf \left\{ \frac{1}{2} \Vert \pi (\bh) \Vert_{\sH}^2 : \bh \in \sbH, \vertiii{\bh} \geq 1 \right\} \\
        &= \inf \left\{ \frac{1}{2} \Vert \pi (\bh) \Vert_{\sH}^2 : \bh \in \sbH, \vertiii{\bh} = 1 \right\} \\
        &= \eta_0.
    \end{align}

   Let $\delta > 0$ be arbitrary. Then

    \begin{equation}
        \limsup_{\varepsilon \downarrow 0} \varepsilon^2 \log \bmu \left( \vertiii{\delta_{\varepsilon} \bx} \geq 1 \right) \leq - \eta_0 < - (\eta_0 - \delta),
    \end{equation}

    and thus there exists a $\varepsilon_0(\delta) > 0$ s.t.  

    \begin{equation}
        \bmu \left( \vertiii{\delta_{\varepsilon} \bx} \geq 1 \right) \leq e^{\frac{- (\eta_0 - \delta)}{\varepsilon^2}}, \quad \forall \varepsilon < \varepsilon_0(\delta) .
    \end{equation}

    Set $\ell = \frac{1}{\varepsilon}$ and $\ell_0 = \frac{1}{\varepsilon_0(\delta)}$. Then for every $\ell > \ell_0$ (and thus $\varepsilon < \varepsilon_0(\delta)$)

    \begin{equation}
        \bmu \left( \vertiii{\bx} \geq \ell \right) = \bmu \left( \varepsilon \vertiii{\bx} \geq 1 \right) = \bmu \left( \vertiii{\delta_{\varepsilon} \bx} \geq 1 \right) \leq e^{\frac{- (\eta_0 - \delta)}{\varepsilon^2}} = e^{- (\eta_0 - \delta) \ell^2} .
    \end{equation}

    Thus for any $\delta$ we may choose a constant $C$ and set $\eta = \eta_0 - \delta$ so that \eqref{eq:exp_int} is satisfied. 
    
    By Lemma \ref{lem:decay_to_exp_int} we conclude that 

    \begin{equation}
        \bbE \left[ \exp\left( \beta \vertiii{\bx}^2
        \right)\right] < \infty, \quad \beta < \eta .
    \end{equation}

    Since $\delta > 0$ was arbitrary, the result follows.
\end{proof}

\section{Cameron--Martin Theorem and Formula} \label{sec:CM}

Let $\mu = \sN(0, \Sigma)$ be a (possibly degenerate) Gaussian measure on $\bbR^d$ and for a fixed $x \in \bbR^d$ consider the shifted measure 

\begin{equation}
    \mu_x(A) := \mu(A - x), \quad A \in \sB_{\bbR^d}. \label{symb:shifted_measure}
\end{equation}

The subspace of directions into which $\mu$ can be shifted in the above sense and remain equivalent\footnote{In the sense that both Radon--Nikodým derivatives $\frac{\dd \mu_x}{\dd \mu}$ and $\frac{\dd \mu}{\dd \mu_x}$ exist; in symbols $\mu \approx \mu_x$.} can easily be guessed: it consists of the vectors in the support of $\mu$, which coincides precisely with the Cameron--Martin space of $\mu$ in $\bbR^d$: 

\begin{equation}
    \texttt{CM}(\bbR^d, \mu) = \supp(\mu) . \label{symb:support}
\end{equation}

 Reducing considerations to the subspace $\supp(\mu) \subseteq \bbR^d$ yields a non-degenerate Gaussian measure $\mu$ on $\supp(\mu) \subseteq \bbR^d$. Thus for any $x \in \texttt{CM}(\bbR^d, \mu)$ there exists a density of $\mu$ and $\mu_x$ w.r.t. the Lebesgue measure on $\supp(\mu) \subseteq \bbR^d$ which is strictly positive. This in turn implies that $\mu$ and $\mu_x$ are equivalent\label{symb:equivalence_of_measures}. Furthermore, there exists an explicit formula for $\frac{\dd \mu_x}{\dd \mu}$. On infinite dimensional spaces similar results hold, albeit with the subtlety that \eqref{symb:support} is only an inclusion from left to right and that there is no analogue of the Lebesgue measure. This is the content of the following well-known theorem of R.H. Cameron and W.T. Martin.

\begin{thm}[{Classical Cameron--Martin Theorem, see e.g. \cite[Prop. 2.4.2. \& Prop. 2.4.5.(i)]{bogachevGaussianMeasures1998}}] \label{thm:classical_CM}
    Let $(E, \sH, i, \mu)$ be an abstract Wiener space and define for any $x \in E$ the shift operator $T_x: E \rightarrow E$ by 

    \begin{equation}
        T_x(y) = y + x, \quad y \in E . \label{symb:ordinary_shift}
    \end{equation}

    Then 

    \begin{equation}
        \underbrace{(T_h)_{\ast} \mu}_{=: \mu_h} \approx \mu \quad \Leftrightarrow \quad h \in \sH .
    \end{equation}

    If $h \in \sH$, then the Radon--Nikodým of $\mu_h$ w.r.t. $\mu$ has the form 

    \begin{equation}
        f_h(x) := \frac{\dd \mu_h}{\dd \mu}(x) = \exp \left( \underline{h}(x) - \frac{1}{2} \Vert h \Vert_{\sH}^2 \right), \quad x \in E, \label{symb:f_h}
    \end{equation}

    where $\underline{h}$ is the image of $h$ under the identification of $\sH$ with the reproducing kernel Hilbert space of $\mu$. 
\end{thm}

As similar relation between the Cameron--Martin space, equivalence of shifted measures, and densities holds for abstract Wiener model spaces. This is the content of the rest of this section. 

\begin{defn}[Lifted Shift Operator] \label{defn:lifted_shift}
    Let $((\calT, \bE, [\cdot], \calN), \sbH, \bmu, \fL, \hat{\fL})$ be an AWMS. A function
    
    \begin{equation}
        \bT: \sbH \times \supp \bmu \rightarrow \bE; \quad (\bh, \bx) \mapsto \bT_{\bh}(\bx),
    \end{equation}

     which is measurable in the second entry is called \textbf{lifted shift operator} if for any $\bh \in \sbH$

    \begin{equation}
        \left( \bT_{\bh} \circ \hat{\fL} \right)(\bx) = \left( \hat{\fL} \circ T_{\pi(\bh)} \right)(\bx), \label{eq:lifted_shift_defn}
    \end{equation}

    for $\bmu$-a.e. $\bx \in \bE$.
\end{defn}

\begin{thm}[Lifted Cameron--Martin Theorem] \label{thm:lifted_CM_thm}
    Let $((\calT, \bE, [\cdot], \calN), \sbH, \bmu, \fL, \hat{\fL})$ be an AWMS.
    
    \begin{enumerate}[(i)]
        \item Let $h \in \sH$ and let $\bmu_h := (\hat{\fL} \circ T_h)_{\ast} \mu$. Then $\bmu_h \approx \bmu$ and 

        \begin{equation}
            \frac{\dd \bmu_h}{\dd \bmu}(\bx) = \exp \left( \underline{h}(\pi (\bx)) - \frac{1}{2} \Vert h \Vert_{\sH}^2 \right).
        \end{equation}

        \item Let $\bT$ be a lifted shift operator in the sense of Definition \ref{defn:lifted_shift}, then $\bmu_{\bh} := \left( \bT_{\bh} \circ \hat{\fL} \right)_{\ast} \mu$ coincides with $\bmu_{\pi(\bh)}$ as defined in (i) and 

        \begin{equation}
            \frac{\dd \bmu_{\bh}}{\dd \bmu}(\bx) = \exp \left( \underline{\pi(\bh)}(\pi(\bx)) - \frac{1}{2} \Vert \pi(\bh) \Vert_{\sH}^2 \right). \label{eq:density_for_lifted_shift}
        \end{equation}
    \end{enumerate}
\end{thm}
\begin{proof} 
    (i) Let $h \in \sH$ be arbitrary and let $A \subseteq \bE$ be measurable. Then 

    \begin{align}
        \bmu_{h} (A) = \mu \left( \hat{\fL} \circ T_h \in A \right) = \mu_h \left( \hat{\fL} \in A \right) .
    \end{align}

    The classical Cameron--Martin Theorem \ref{thm:classical_CM} implies that the last term is $0$ if and only if $\mu \left( \hat{\fL} \in A \right) = \bmu (A)$ is also $0$. Hence $\bmu_h \approx \bmu$. Regarding the density, let again $\bh \in \sbH$ be arbitrary and $A \subseteq \bE$ measurable. Then 

    \begin{align}
        \bmu_{h}(A) = \mu \left( \hat{\fL} \circ T_h \in A \right) = \int_E 1_{\left\{ \hat{\fL} \in A \right\}} f_h(x) \dd \mu(x) = \int_{\bE} 1_{\left\{ \bx \in A \right\}} f_{h} \left( \pi(\bx) \right) \dd \bmu(\bx) .
    \end{align}

    The function $f_{h} \circ \pi$ is measurable, non-negative, and integrates to $1$ w.r.t. $\bmu$. Hence, from the above and \eqref{symb:f_h} of the classical Cameron--Martin Theorem \ref{thm:classical_CM} we conclude that 
    
    \begin{equation}
        \exp \left( \underline{h}(\pi\bx) - \frac{1}{2} \Vert h \Vert_{\sH}^2 \right) = f_{h}(\pi(\bx)) = \frac{\dd \bmu_{h}}{\dd \bmu}(\bx), \quad \text{for} \bmu \text{-a.e.} \bx \in \bE.
    \end{equation}

    (ii) Let $\bh \in \sbH$. Then using \eqref{eq:lifted_shift_defn} we obtain

    \begin{equation}
        \bmu_{\bh} = \left( \bT_{\bh} \circ \hat{\fL} \right)_{\ast} \mu = \left( \hat{\fL} \circ \bT_{\pi(\bh)} \right)_{\ast} \mu = \bmu_{\pi(\bh)} ,
    \end{equation}

    and \eqref{eq:density_for_lifted_shift} follows.
\end{proof}

For abstract Wiener models spaces with approximation, there is a canonical choice of lifted shift operator: 

\begin{thm} \label{thm:shift_operator_consistency}
   Let $((\calT, \bE, [\cdot], \calN), \sbH, \bmu, \fL, \hat{\fL}, (\sK, \fM), (\Phi_m)_{m \in \bbN})$ be an abstract Wiener model space with approximation. For any $\bh \in \sbH$ define $\bT_{\bh}: \supp \bmu \rightarrow \bE$ via

    \begin{equation}
        \bT_{\bh} (\bx) = \sum_{\tau \in \calT} \sum_{a \in \{\pi(\bx),\pi(\bh)\}^{[\tau]}} \lim_{m \rightarrow \infty} \fM_{\tau}^{\otimes}(\Phi_{m}(a_1) \otimes \ldots \otimes \Phi_{m}(a_{[\tau]})), \label{eq:shift_defn}
    \end{equation}

    as a limit in $\bmu$-probability. Then $\bT$ is a lifted shift operator in the sense of Definition \ref{defn:lifted_shift}, i.e. for any $\bh \in \sbH$

    \begin{equation}
        \left( \bT_{\bh} \circ \hat{\fL} \right)(\bx) = \left( \hat{\fL} \circ T_{\pi(\bh)} \right)(\bx)
    \end{equation}

    on the complement of a $\bmu$-nullset (depending on $\bh)$. 
\end{thm}
\begin{proof}
    The limits in \eqref{eq:shift_defn} exist in $\bmu$-probability due the assumption of compatibility of the approximation $(\Phi_m)_{m \in \bbN}$ with the $\sK$-skeleton lift $\fM$ (see Definition \ref{defn:admissible_approximation}).

    Let $\bh \in \sbH$ be arbitrary and denote $h:= \pi (\bh)$. On the one hand, using the fact that $\hat{\fL}$ is a lift almost surely, for $\mu$-a.e. $x \in E$ the equality 

    \begin{align}
        \left( \bT_{\bh} \circ \hat{\fL}\right)(x) &= \sum_{\tau \in \calT} \sum_{a \in \{\pi(\hat{\fL}(x)),\pi(\bh)\}^{[\tau]}} \lim_{m \rightarrow \infty} \fM_{\tau}^{\otimes}(\Phi_{m}(a_1) \otimes \ldots \otimes \Phi_{m}(a_{[\tau]})) \\
        &= \sum_{\tau \in \calT} \sum_{a \in \{x,h\}^{[\tau]}} \lim_{m \rightarrow \infty} \fM_{\tau}^{\otimes}(\Phi_{m}(a_1) \otimes \ldots \otimes \Phi_{m}(a_{[\tau]})) 
    \end{align}

    holds. On the other hand, by Proposition Proposition \ref{prop:L2L1}
    
    \begin{equation}
        \hat{\fL} \circ T_h = \left( \lim_{m \rightarrow \infty}^{\mu} \fL \circ \Phi_m \right) \circ T_h = \lim_{m \rightarrow \infty}^{\mu} \left( \fL \circ \Phi_m \circ T_h \right) 
    \end{equation}

    and by $\calT$-multi-linearity of $\fL$, for a fixed $m \in \bbN$,

    \begin{equation}
        \fL(\Phi_m(T_h (x))) = \fL(\Phi_m(x+h)) = \sum_{\tau \in \calT} \sum_{a \in \{x,h\}^{[\tau]}} \fM_{\tau}^{\otimes}(\Phi_{m}(a_1) \otimes \ldots \otimes \Phi_{m}(a_{[\tau]})) .
    \end{equation}
    
    Therefore, for $\mu$-a.e. $x \in E$ the equality 

    \begin{equation}
        \left( \hat{\fL} \circ T_h \right)(x) = \sum_{\tau \in \calT} \sum_{a \in \{x,h\}^{[\tau]}} \lim_{m \rightarrow \infty} \fM_{\tau}^{\otimes}(\Phi_{m}(a_1) \otimes \ldots \otimes \Phi_{m}(a_{[\tau]}))
    \end{equation}

    holds.
\end{proof}

\begin{rem}
  The existence of a lifted shift operator (a.k.a. translation operator on rough path or model space) is non-trivial and typically relies on some reconstruction (or sewing) arguments on mixed Sobolev and Hölder (or variation) scales; it then comes with continuity (and further regularity) properties. See e.g. \cite{frizMultidimensionalStochasticProcesses2010}, \cite{frizCourseRoughPaths2020} for the rough path case, \cite{choukSupportTheoremSingular2018}, \cite{cannizzaroMalliavinCalculusRegularity2017}, \cite{frizPreciseLaplaceAsymptotics2022} for the gPAM model, with applications to support theory, Malliavin calculus and Laplace asymptotics, respectively, and \cite{tsatsoulisSpectralGapStochastic2018} for the $\Phi^4_2$-stochastic quantization equation with applications to support theory. In the context of general singular SPDEs, within regularity structures, M. Hairer and P. Sch\"onbauer (\cite{schonbauerMalliavinCalculusDensities2023}, \cite{hairerSupportSingularStochastic2022}) make use of a ``weak translation operator'', with application to Malliavin calculus and support theory, defined by an elegant doubling of noise argument. At this moment, we do not see a counterpart of this construction in the generality of AWMS. It is conceivable that a AWMS with additional properties, allowing for an abstract doubling-of-noise, can accommodate their construction but at this moment this is pure speculation.
\end{rem}

\section{Examples} \label{sec:examples}
\subsection{Gaussian Rough Paths} \label{subsec:GaussianRP}

\paragraph{General Setup:} Throughout this section let $T > 0$ be fixed and let $P([0,T]) \label{symb:partition}$ denote the set of partitions of $[0,T]$. Let us write arguments of functions as a subscript and for a function $f$ of a single parameter with values in a vector space let us make the convention $f_{s,t} := f_t - f_s$. \\

Consider a continuous and centered $d$-dimensional Gaussian process $X = (X^1, \ldots, X^d)$ with independent components which is of finite $\rho$-variation\label{symb:rho_var}, in the sense that there exists a $\rho \in [1,2)$ such that 

\begin{equation}
    \Vert R \Vert_{\rho-\var; [0,T]^2}^{\rho} := \sup_{Q, Q' \in P([0,T])} \sum_{t_i \in Q,t_j \in Q'} \left\Vert \bbE \left[ \left(X_{t_{i+1}} - X_{t_{i}} \right) \otimes \left( X_{t_{j+1}} - X_{t_{j}} \right) \right] \right\Vert^{\rho} < \infty . \label{eq:rho_var}
\end{equation}

Define for any $p \in [1, \infty)$ the separable Banach space $C^{0,p-\var}([0,T]; \bbR) \label{symb:variation_smooth}$ as the closure of the set of smooth functions on $[0,T]$ and $[0,T]^2$ respectively w.r.t. the $p$-variation norm

\begin{align}
    \Vert x \Vert_{p-\var; [0,T]}^p &:= \vert x_0 \vert + \sup_{Q \in P([0,T])} \sum_{t_i \in Q} \vert x_{t_i, t_{i+1}} \vert^p \\
    \Vert x \Vert_{p-\var; [0,T]^2}^p &:= \sup_{Q \in P([0,T])} \sum_{t_i,t_j \in Q} \vert x_{t_i, t_j} \vert^{p},
\end{align}

for $1$-parameter and $2$-parameter functions, respectively. 

\begin{rem}
While here we work in the $p$-variation setting, assuming that $\Vert R \Vert_{\rho-\var; [0,T]^2}^{\rho}$ is controlled by a 2D control $\omega$ s.t. $\omega([0,T]^2) < \infty$ allows switching to the Hölder setting. See \cite[Chap. 5 \& 15]{frizMultidimensionalStochasticProcesses2010}.
\end{rem}

\begin{rem}
    Assumption \eqref{eq:rho_var} ensures the existence of a Gaussian rough path lift. By recent work of P. Gassiat and T. Klose \cite{gassiatGaussianRoughPaths2023} one may replace this assumption by that of ``controlled complementary Young regularity (cCYR)'' - see \cite[Thm. 2.7]{gassiatGaussianRoughPaths2023}. Given that cCYR is (slightly) stronger than complementary Young regularity (as in the sense of Definition \ref{defn:CYR}), with a view towards Proposition \ref{prop:KL_is_approx} and Proposition \ref{prop:PL_is_approx}, cCYR is a natural alternative assumption in this section. 
\end{rem}

\subsubsection{As an AWMS with Approximation}  \label{subsubsec:GRP_with_approx}

\paragraph{Ambient Space and AWS:} Our goal is to obtain an AWMS with approximation such that the full lift coincides with the Gaussian rough path lift $\fX$ associated to $X$ (as defined in \cite[Chap. 15]{frizMultidimensionalStochasticProcesses2010}) and the enhanced measure $\bmu^{\fX}$ is the distribution of $\fX$ on $\bE$. We will do so with two different choices of approximation, by specifying two different Bottom-Up data and applying Theorem \ref{thm:bottom_up_construction}. \\

Define the ambient space $(\calT, \bE, [\cdot], \calN)$ by

\begin{equation}
    \calN = \calT^{(1)} = \{ 1, \ldots, d\}, \quad \calT^{(2)} = \{ ij : 1 \leq i,j \leq d \}, \quad \calT^{(3)} = \{ ijk : 1 \leq i,j,k \leq d \},
\end{equation}

and 

\begin{equation} 
    E_{i} = C^{0,p-\var}([0,T]; \bbR), \quad E_{ij} = C^{0,\frac{p}{2}-\var}([0,T]^2; \bbR), \quad E_{ijk} = C^{0,\frac{p}{3}-\var}([0,T]^2; \bbR),
\end{equation}

for some $p > 2 \rho$ (cf. \cite[Def. 8.6]{frizMultidimensionalStochasticProcesses2010}). Fix $\rho \in [1,2)$ and $p > 2 \rho$ throughout the section. Define the abstract Wiener space $(E, \sH, i, \mu)$ by

\begin{align}
E &= \bigoplus_{i = 1}^d E_i = \bigoplus_{i = 1}^d C^{0,p-\var}([0,T]; \bbR), \quad \Vert x \Vert_{E} = \sum_{i = 1}^{d} \Vert x^i \Vert_{p-\var}, \quad p > 2\rho, \\
\sH &= \bigoplus_{i = 1}^d {\sH}_i, \quad \Vert h \Vert^2_{\sH} = \sum_{i = 1}^{d} \Vert h^i \Vert^2_{\sH_i}, 
\end{align}

where $\sH_i$ is the Cameron--Martin space associated to the law of $X^i$ and $\mu$ is the Gaussian measure on $E$ associated to $X$. Examples of processes satisfying condition \eqref{eq:rho_var} include Brownian motion, (with $\rho = 1$), Gaussian martingales (with $\rho = 1$), Ornstein--Uhlenbeck processes (with $\rho = 1$), and fractional Brownian motion with Hurst parameter $H$ (with $\rho = \frac{1}{2H}$).\footnote{See \cite[Chap. 15]{frizLargeDeviationPrinciple2007} for $d = 1$. Due to the assumption that $X$ is centered and has independent components the case for general $d \geq 1$ is immediate.} \\

The above gives an ambient space $(\calT, \bE, [\cdot], \calN)$ and an AWS $(E, \sH, i, \mu)$, such that $E_{\calN} = E$. Let us now turn to specifying what is missing to give Bottom-Up data; that is, an intermediate space $\sK$, a $\sK$-skeleton lift $\fM$ and an admissible approximation $(\Phi_m)_{m \in \bbN}$.

\paragraph{Approximation Scheme:} Let $\sK$ be an intermediate space which will be specified further after the definition of the $\sK$-skeleton lift. Define the $\sK$-skeleton lift $\fM^{\GRP}: \sK \rightarrow \bE$ by

\begin{align}
    \left[\fM^{\GRP}_i (h) \right]_t &= h^i_t, \quad \left[\fM^{\GRP}_{ii} (h)\right]_{s,t} = \frac{1}{2} (h^i_{s,t})^2, \quad \left[\fM^{\GRP}_{iii} (h)\right]_{s,t} = \frac{1}{6} (h^i_{s,t})^3, \label{eq:GRP_lift_first_line}\\
    \left[\fM^{\GRP}_{ij} (h)\right]_{s,t} &= \int_s^t h^i_{s,r} \dd h^j_r, \quad \left[\fM^{\GRP}_{iij} (h)\right]_{s,t} = \int_s^t (h^i_{s,r})^2 \dd h^j_r, \quad i \neq j , \\
    \left[\fM^{\GRP}_{ijk} (h)\right]_{s,t} &= \int_s^t \int_s^r h^i_{s,u} \dd h^j_u \dd h^k_r, \quad i \neq j \neq k, \\
    \left[\fM^{\GRP}_{iji} (h)\right]_{s,t} &= [\fM^{\GRP}_{ij} (h)]_{s,t} \cdot [\fM^{\GRP}_{i} (h)]_{s,t} - 2 [\fM^{\GRP}_{iij} (h)]_{s,t}, \quad i \neq j, \label{eq:GRP_shuffle_1}\\
    \left[\fM^{\GRP}_{jii} (h)\right]_{s,t} &= [\fM^{\GRP}_{ii} (h)]_{s,t} \cdot [\fM^{\GRP}_j (h)]_{s,t} - [\fM^{\GRP}_{iji} (h)]_{s,t} - [\fM^{\GRP}_{iij} (h)]_{s,t}, \quad i \neq j, \label{eq:GRP_shuffle_2}
\end{align}

for $h \in \sK$, $s,t \in [0,T]$.\footnote{Line \eqref{eq:GRP_shuffle_1} and \eqref{eq:GRP_shuffle_2} come from the shuffle relations \label{symb:shuffle} $ij \shuffle i = iji + 2 iij$ and $ii \shuffle j = iij + iji + jii$ and the weak geometricity of Gaussian rough paths, cf. \cite[Chap. 15]{frizLargeDeviationPrinciple2007}.} The definition of the $\sK$-skeleton lift $\fM^{\GRP}$ suggests that $\sK$ should be contained in a space of functions for which the iterated integral makes sense canonically and is continuous i.e. $\sK \subseteq C^{\beta-\var}$ for $\beta < 2$. Depending on the admissible approximation $(\Phi_m)_{m \in \bbN}$, we will choose either

\begin{equation}
    \sK^1 = \sH \quad \text{or} \quad \sK^2 = \bigoplus_{i = 1}^d C^{0,\rho-\var}([0,T]; \bbR). \label{eq:definition_of_K_for_GRP}
\end{equation}

\begin{prop} \label{prop:intermediate_spaces}
In the context of the current section 

\begin{enumerate}
    \item $\sK^1$ is a compact intermediate space if and only if $\sH$ is finite dimensional,
    \item $\sK^2$ is a compact intermediate space if $\rho > 1$, 
    \item (for $\sK \in \{\sK^1, \sK^2 \}$) $\fM^{\GRP}$ is a $\sK$-skeleton lift.
\end{enumerate}    
\end{prop}
\begin{proof}
    \begin{enumerate}
         \item Since the identity $\id_{\sH}: \sH \rightarrow \sH$ is compact if and only if $\sH$ is finite dimensional the proposition follows. 

         \item Let $\rho' \in [1,2)$ be arbitrary. Then by \cite[Prop. 15.8]{frizMultidimensionalStochasticProcesses2010} we have $\Vert h \Vert^2_{\rho'-\var; [0,T]} \leq \Vert h \Vert_{\sH} \Vert R \Vert_{\rho'-\var;[0,T]^2}$ and thus $\sH \subseteq C^{\rho'-\var}$ for every $\rho' \in [1,2)$. Since $C^{\rho'-\var} \subseteq C^{0,\rho-\var}$ compactly for any $\rho' < \rho$ we have $\sH \subseteq C^{0,\rho-\var}$ compactly for every $\rho \in (1,2)$. Hence $\sK^2$ is a compact intermediate space. 

         \item The first equation in \eqref{eq:GRP_lift_first_line} confirms the lifting property while continuity is guaranteed by the Young--Loève estimates \cite[Sec. 6.2]{frizMultidimensionalStochasticProcesses2010}. Defining $(\fM^{\GRP})^{\otimes}$ as iterated Young integrals on $\sK$ gives the $\calT$-multi-linearity. 
    \end{enumerate}
\end{proof}

Recall the following definition from the theory of Gaussian rough paths. 

\begin{defn}[{Complementary Young Regularity; \cite[Condition 15.60]{frizMultidimensionalStochasticProcesses2010}}] \label{defn:CYR}
    Let $X$ be a Gaussian process. Then $X$ is said to satisfy \textbf{complementary Young regularity} (CYR) if there exists a $q \geq 1$ such that 

            \begin{equation}
                \sH \hookrightarrow \calC^{q-\var}([0,T]; \bbR), \quad \text{and} \quad \frac{1}{p} + \frac{1}{q} > 1 . \label{eq:CYR}
            \end{equation}
\end{defn}

For example, fraction Brownian motion with Hurst parameter $H \in (0,1)$ satisfies CYR if $H > \frac{1}{4}$.

\paragraph{Karhunen--Loève Approximation} \label{subsubsec:Karhunen_Loève_approx}

Let $(E, \sH, i, \mu)$ be an abstract Wiener space with an ONB $(e_k)_{k \in \bbN}$ of $\sH$ contained in $E^{\ast}$. Then define the Karhunen--Loève approximation\footnote{Other names include spectral-Galerkin approximation and $L^2$-approximation.} by 

\begin{equation}
    \Phi_m(x) = \sum_{k = 1}^{m} \langle e_k, x \rangle e_k, \quad x \in E, m \in \bbN.
\end{equation}

Since the Karhunen--Loève approximation maps (by definition) into $\sH$, the natural choice in \eqref{eq:definition_of_K_for_GRP} is $\sK^1 = \sH$. Due to the structure of $\sH = \bigoplus_{i = 1}^d \sH_i$ as a direct sum the Karhunen--Loève expansion takes the form 

\begin{equation}
    \Phi^{\KL}_m(x) = \sum_{i = 1}^{d} \sum_{k = 1}^{m} \langle e^i_k, x^i \rangle_{\sH_i} e^i_k, \quad x \in E,
\end{equation}

where $(e^i_k)_{k \in \bbN}$ is an ONB of $\sH_i$ contained in $E_i^{\ast}$ and $x^i$ is the $i$-th component of $x \in E$. 

\begin{prop}\label{prop:KL_is_approx}
    If $X$ satisfies complementary Young regularity in the sense of Definition \ref{defn:CYR}, then $\Phi^{\KL}$ is an admissible approximation w.r.t. the $\sK^1$-skeleton lift $\fM^{\GRP}$.
\end{prop}
\begin{proof}
    Fix $m \in \bbN$. Since $x^i \mapsto \langle e^i_k, x^i \rangle_{\sH_i}$ lies in $E_i^{\ast}$ for every $i \in \bbN, 1 \leq k \leq d$, the linear operator $\Phi^{\KL}_m$ is bounded. Assumption \eqref{eq:A2.I} is satisfied by general abstract Wiener space theory (see \cite[App. D.3]{frizMultidimensionalStochasticProcesses2010}) while \eqref{eq:A2.II} is satisfied since $\cup_{i= 1}^d (e^i_k)_{k \in \bbN}$ is an ONB of $\sH$. By \cite[Thm. 15.51]{frizMultidimensionalStochasticProcesses2010} we deduce that for every $\tau \in \calT$  

        \begin{equation}
            \left\Vert \Vert \pi_{\tau} \circ \fM^{\GRP} \circ \Phi^{\KL}_m - \pi_{\tau} \circ \fX \Vert_{\frac{p}{[\tau]}} \right\Vert_{L^2(E, \mu; \bbR)} \rightarrow 0, \label{eq:KL_limit}
        \end{equation}
    
    where $\fX$ is the Gaussian rough paths lift associated to $X$ in the sense of \cite[Chap. 15]{frizMultidimensionalStochasticProcesses2010}. That is, $\pi_{\tau} \circ \fM^{\GRP} \circ \Phi^{\KL}_m$ converges in probability in $E_{\tau}$, which in turn implies \eqref{eq:A2.III}. Thus $(\Phi^{\KL}_m)_{m \in \bbN}$ is an admissible approximation and the limit of $(\fM^{\GRP} \circ \Phi^{\KL}_m)_{m \in \bbN}$ can be identified with the Gaussian rough path lift $\fX$. \\

    Since $(\Phi^{\KL}_m)_{m \in \bbN}$ satisfies \eqref{eq:A2.I} and \eqref{eq:A2.II}, condition \eqref{eq:sK_compatibility} is satisfied by \cite[Thm. 9.35 (ii)]{frizMultidimensionalStochasticProcesses2010} (which requires complementary Young regularity).
\end{proof}

\begin{prop} 
    If $X$ satisfies complementary Young regularity in the sense of Definition \ref{defn:CYR}, then the data $(\calT, \bE, [\cdot], \calN), (E, \sH, i, \mu), (\sK^1, \fM^{\GRP}), (\Phi^{\KL}_m)_{m \in \bbN}$ is Bottom-Up data in the sense of Theorem \ref{thm:bottom_up_construction} and induces an AWMS with approximation such that the full lift coincides with $\fX$ and the enhanced measure $\bmu^{\fX}$ coincides with the distribution of the Gaussian rough path lift associated to $X$. 
\end{prop}
\begin{proof}
    In light of Propositions \ref{prop:intermediate_spaces} and \ref{prop:PL_is_approx} the only thing left to show is \eqref{eq:WIC_assumption_BU}. 

    Let $m \in \bbN$, $\tau \in \calT$, and $x \in E$ be arbitrary and let
    
    \begin{equation}
        I^{[\tau]}_{m} := \Big\{ (\alpha, \beta): \{1, \ldots, [\tau]\} \rightarrow \{1, \ldots, m\} \times \{1, \ldots, d\} \Big\}.
    \end{equation} 
    
    Then since $\fM^{\GRP}$ is $\calT$-multi-linear
    
    \begin{align}
        \fM^{\GRP}_{\tau} \left( \Phi^{\KL}_m(x) \right) &= \fM_{\tau}^{\GRP \otimes}\left( \Phi^{\KL}_m \left( x^{\otimes [\tau]} \right) \right) \\
        &= \fM_{\tau}^{\GRP \otimes} \left( \left( \sum_{i = 1}^{d} \sum_{k = 1}^{m} \left\langle e^i_k, x^i \right\rangle_{\sH_i} e^i_k \right)^{\otimes [\tau]} \right) \\
        &= \sum_{(\alpha, \beta) \in I^{[\tau]}_{m}} \left( \left\langle e_{\alpha_1}^{\beta_1}, x^1 \right\rangle_{\sH_{\beta_1}} \cdot \ldots \cdot \left\langle e_{\alpha_{[\tau]}}^{\beta_{[\tau]}}, x^{\beta_{[\tau]}} \right\rangle_{\sH_{\beta_{[\tau]}}} \right) \fM_{\tau}^{\GRP \otimes} \left( e_{\alpha_1}^{\beta_1} \otimes \ldots \otimes e_{\alpha_{[\tau]}}^{\beta_{[\tau]}} \right).
    \end{align} 
    
    Evaluating the above at $s,t \in [0,T]$ yields 
    
    \begin{equation}
        \left[ \fM_{\tau}^{\GRP}(\Phi^{\KL}_m(x)) \right]_{s,t} = \sum_{(\alpha, \beta) \in I^{[\tau]}_{m}} \left( \prod_{i = 1}^{[\tau]}\langle e_{\alpha_i}^{\beta_i}, x^{\beta_i} \rangle_{\sH_{\beta_i}} \right) \underbrace{\left[\fM_{\tau}^{\GRP \otimes} \left( e_{\alpha_1}^{\beta_1} \otimes \ldots \otimes e_{\alpha_{[\tau]}}^{\beta_{[\tau]}} \right)\right]_{s,t}}_{\in \bbR},
    \end{equation}
    
    i.e. $x \mapsto \left[ \fM_{\tau}^{\GRP}(\Phi^{\KL}_m(x)) \right]_{s,t}$ is a linear combination of $[\tau]$-fold products of bounded linear functionals $x \mapsto \langle e_{\alpha_i}^{\beta_i}, x^{\beta_i} \rangle_{\sH_{\beta_i}}$ and thus lies in $\calP^{(\leq [\tau])}(E, \mu; \bbR)$ by Lemma \ref{lem:polynomial_in_WIC}. Thus $\fM_{\tau}^{\GRP}(\Phi^{\KL}_m) \in \calP^{(\leq [\tau])}(E, \mu; E_{\tau})$ Proposition \ref{prop:WIC_equiv_characterizations}(iv).

    Since the full lift coincides with the limit of $(\fM^{\GRP} \circ \Phi^{\KL}_m)_{m \in \bbN}$, the fact that the full lift can be identified with the Gaussian rough path lift $\fX$ follows from the proof of Proposition \ref{prop:KL_is_approx}.
\end{proof}

\paragraph{Piecewise Linear Approximation} \label{subsubsec:piecewise_linear_approx}

For the reasons given in Subsection \ref{subsec:intermediate_space} we choose $\sK^2 = C^{0,\rho - \var}$ for the piecewise linear approximation. Let $\Phi^{\PL} = (\Phi^{\PL}_m)_{m \in \bbN}$ be the piecewise linear approximation on the dyadic dissection of $[0,T]$ with mesh size $T \cdot 2^{-m}$

\begin{equation}
    D_m := \left\{ \frac{T\cdot k}{2^m} : 0 \leq k \leq 2^m \right\} . \label{symb:dyadic_partition}
\end{equation}

\begin{prop}\label{prop:PL_is_approx}
    If $X$ satisfies complementary Young regularity in the sense of Definition \ref{defn:CYR}, then $\Phi^{\PL}$ is an admissible approximation w.r.t. the $\sK^2$-skeleton lift $\fM^{\GRP}$. 
\end{prop}
\begin{proof}
    Recall that $\rho \in [1, 2)$. Since piecewise linear functions have finite variation, the variation embedding theorems guarantee that $(\Phi^{\PL}_m)_{m \in \bbN}$ indeed maps into $\sK^2$. To see the continuity of $(\Phi^{\PL}_m)_{m \in \bbN}$, notice that the total variation norm of a function which is piecewise linear on some partition is attained for that partition. In particular, $\Vert \Phi^{\PL}_m (x) \Vert_{1-\var} \leq 2m \Vert x \Vert_{\infty}$ for any $x \in C^{0, p-\var}$. Thus we have 

        \begin{equation}
            \Vert \Phi^{\PL}_m(x) \Vert_{\rho-\var} \Vert \leq \Vert \Phi^{\PL}_m (x) \Vert_{1-\var} \leq 2m \Vert x \Vert_{\infty} \lesssim \Vert x \Vert_{p-\var}, \quad m \in \bbN.
        \end{equation}

    Since $\sK^2 = C^{0,\rho-\var}$ and $E = C^{0,p-\var}$ (as opposed to $C^{\rho-\var}$ and $C^{p-\var}$ \label{symb:variation}), both \eqref{eq:A2.I} and \eqref{eq:A2.II} follow from \cite[Thm. 5.33 (i.3)]{frizMultidimensionalStochasticProcesses2010} since piecewise linear approximations are nothing but geodesic approximations in $\bbR$ with the standard Riemannian metric. Finally, \eqref{eq:A2.III} is guaranteed by \cite[Thm. 15.34]{frizMultidimensionalStochasticProcesses2010}. \\

    In a similar fashion as for the Karhunen--Loève approximation, Condition \eqref{eq:sK_compatibility} is satisfied by \cite[Thm. 9.35 (ii)]{frizMultidimensionalStochasticProcesses2010} (which requires complementary Young-regularity) since, as shown above, $(\Phi^{\PL}_m)_{m \in \bbN}$ satisfies \eqref{eq:A2.I} and \eqref{eq:A2.II}.
\end{proof}

\begin{prop} \label{prop:GRP_gives_AWMS_with_approx}
    If $X$ satisfies complementary Young regularity in the sense of Definition \ref{defn:CYR}, then the data associated to a Gaussian rough path lift $(\calT, \bE, [\cdot], \calN), (E, \sH, i, \mu), (\sK^2, \fM^{\GRP})$, $(\Phi^{\PL}_m)_{m \in \bbN}$ is Bottom-Up data in the sense of Theorem \ref{thm:bottom_up_construction} and induces an AWMS with approximation such that the full lift coincides with $\fX$ and the enhanced measure $\bmu^{\fX}$ coincides with the distribution of the Gaussian rough path lift associated to $X$. 
\end{prop}
\begin{proof}
    In light of Propositions \ref{prop:intermediate_spaces} and \ref{prop:PL_is_approx} the only thing left to show is \eqref{eq:WIC_assumption_BU}. This follows immediately from \cite[Prop. 15.20]{frizMultidimensionalStochasticProcesses2010}.
\end{proof}

\subsubsection{As an AWMS (Without Approximation)}

Note that in order to produce an AWMS with approximation, in Proposition \ref{prop:KL_is_approx} and \ref{prop:PL_is_approx} we assumed $X$ to satisfy complementary Young regularity. However, this was only necessary to ensure that the approximations $(\Phi_m)_{m \in \bbN}$ were compatible with $\fM^{\GRP}$; i.e. that they satisfy \eqref{eq:sK_compatibility}. In order to produce an AWMS without approximation, and in particular in order to show an LDP and a Fernique estimate, this is not necessary: 

\begin{prop} \label{prop:AWMS_GRP}
    Let $\fX:E \rightarrow \bE$ denote the Gaussian rough path lift of $X$ in the sense of \cite[Chap. 15]{frizMultidimensionalStochasticProcesses2010} and let $\bmu^{\fX}$ denote its distribution on $\bE$. Then $(\calT, \bE, [\cdot], \calN), \bmu^{\fX}, \fX$ is Top-Down data in the sense of Theorem \ref{thm:top_down_construction}. 
\end{prop}
\begin{proof}
    As shown previously in the section $(\calT, \bE, [\cdot], \calN)$ is an ambient space and by construction $\pi_{\ast} \bmu^{\fX} = \mu$ and $\fX_{\ast} \mu = \bmu^{\fX}$. By \cite[Prop. 15.20]{frizMultidimensionalStochasticProcesses2010} we have $\fX \in \calP^{[\calT]}(E, \mu; \bE)$. 
\end{proof}

\subsubsection{Application of Theorems}

\paragraph{Large Deviations}
By Proposition \ref{prop:AWMS_GRP} and the Top-Down Construction Theorem \ref{thm:top_down_construction} we may apply Theorem \ref{thm:LDP_for_AWMS} to conclude that for a given centered Gaussian process $X$ as in the previous section the family of measures defined by 

\begin{equation}
\bmu^{\fX}_{\varepsilon}(A) = \mu( \delta_{\varepsilon} \fX \in A ) = \mu \left( \left(\varepsilon \fX^{(1)}, \varepsilon^{2} \fX^{(2)}, \varepsilon^{3} \fX^{(3)} \right) \in A \right), \quad A \in \sB_{\bE}, \varepsilon > 0,
\end{equation}

satisfies an LDP with good rate function 

\begin{equation}
    \sJ^{\fX}(\bx) = 
    \begin{cases}
        \frac{1}{2} \Vert \pi (\bx) \Vert^2_{\sH} & \quad \bx \in \sbH^{\fX} \\
        + \infty & \quad \text{else}, 
    \end{cases} \label{eq:GRP_rate_function}
\end{equation}

where $\sbH^{\fX}$ is nothing but 

\begin{equation}
    \sbH^{\fX} = \fM^{\GRP}(\sH) = \left\{ \left( h, \int h \otimes \dd h, \int \int h \otimes \dd h \otimes \dd h \right) : h \in \sH \right\} .
\end{equation}

In the case where $X$ is Brownian motion $\fX$ is the Stratonovich lift in which case

\begin{equation}
    \sJ^{\fX}(\bx) = \begin{cases}
    \frac{1}{2} \sum_{i=1}^d \int_0^T \vert [\pi_{i}(\bx)]'(s) \vert^2 ds & \quad \bx \in \sbH^{\fX} \\ 
    + \infty & \quad \text{else}.
    \end{cases} \label{eq:Strat_rate_function}
\end{equation}

Note that in the case of Brownian motion, since $\rho = 1$, we could have neglected the third level increment, which would have lead to the same value of $\sJ^{\fX}(\bx)$ since the third level increment is determined through the first and second. We note that LDPs in this setting were already obtained in \cite{frizLargeDeviationPrinciple2007}. 

\paragraph{Fernique Estimate}
By Proposition \ref{prop:AWMS_GRP} and the Top-Down Construction Theorem \ref{thm:top_down_construction} we may apply Theorem \ref{thm:exp_int} and conclude that the measure $\bmu^{\fX}$ satisfies a Fernique estimate. That is, let  

\begin{equation}
   \eta_0 := \inf \left\{ \frac{1}{2} \Vert h \Vert_{\sH}^2 : h \in \sH, \left\Vert h \right\Vert_{p-\var} + \left\Vert \int h \dd h \right\Vert_{\frac{p}{2}-\var}^{\frac{1}{2}} + \left\Vert \int \int h \dd h \dd h \right\Vert_{\frac{p}{3}-\var}^{\frac{1}{3}} =1 \right\} .
\end{equation}

Then for any $\eta < \eta_0$ and $t \geq 0$

\begin{equation}
    \bmu \left( \bx \in \bE : \vertiii{\bx} \geq t \right) = \mu \left( \left\Vert \fX^{(1)} \right\Vert_{p-\var} + \left\Vert \fX^{(2)} \right\Vert_{\frac{p}{2}-\var}^{\frac{1}{2}} + \left\Vert \fX^{(3)} \right\Vert_{\frac{p}{3}-\var}^{\frac{1}{3}} \geq t \right) \lesssim \exp\left( - \eta t^2 \right). 
\end{equation}

\paragraph{Cameron--Martin Theorem} Assume the process satisfies complementary Young regularity in the sense of Definition \ref{defn:CYR}. Then by Proposition \ref{prop:KL_is_approx} or Proposition \ref{prop:PL_is_approx} we may apply Theorem \ref{thm:shift_operator_consistency} (for either approximation $\Phi^{\KL}$ or $\Phi^{\PL}$) and Theorem \ref{thm:lifted_CM_thm} to obtain that $\bmu^{\fX}$ is quasi-invariant under transformations of the form 

\begin{equation}
    \bT_{\bh}: \begin{pmatrix}
        x \\
        \int x \dd x \\
        \int \int x \dd x \dd x
    \end{pmatrix}
    \mapsto 
    \begin{pmatrix}
        x + h \\
        \int x \dd x + \int h \dd x + \int x \dd h + \int h \dd h \\
        \int \int x \dd x \dd x + \int \int x \dd x \dd h + \ldots + \int \int x \dd h \dd h + \int \int h \dd h \dd h
    \end{pmatrix}, 
\end{equation}

with Radon--Nikodým density given by 

\begin{equation}
    \frac{\dd \bmu^{\fX}_{\bh}}{\dd \bmu^{\fX}}(\bx) = \exp \left( \underline{\pi (\bh)} (\pi(\bx)) - \frac{1}{2} \Vert \pi(\bh) \Vert_{\sH}^2 \right), \quad \bx \in \bE, 
\end{equation}

where $\bh = \left( h, \int h \dd h, \int \int h \dd h \dd h \right) \in \sbH^{\fX}$. 

\subsection{Ito Brownian Motion} \label{subsec:Ito}

\paragraph{General Setup}

Set $\calN = \calT^{(1)} = \{1, \ldots, d\}$, $\calT^{(2)} = \{ij : 1 \leq i, j \leq d \}$ and define $E_{i} = \calC^{0,\alpha}([0,T]; \bbR)$, $E_{ij} = \calC^{0,2\alpha}([0,T]^2; \bbR)$ as the closure of the set of smooth functions $X, \mathbb{X}$ on $[0,T]$ and $[0,T]^2$, respectively, w.r.t. the norms 

\begin{equation}
    \Vert X \Vert_{\alpha} := \sup_{\substack{s,t \in [0,T] \\ s \neq t}} \frac{\vert X_{s,t} \vert}{\vert t - s \vert^{\alpha}}, \quad \text{resp.}, \quad \Vert \bbX \Vert_{2 \alpha} := \sup_{\substack{s,t \in [0,T] \\ s \neq t}} \frac{\vert \mathbb{X}_{s,t} \vert}{\vert t - s \vert^{2\alpha}},
\end{equation}

for some $0 < \alpha < \frac{1}{2}$ which shall be fixed throughout the rest of the section. Consider a $d$-dimensional Brownian motion $B = (B^1, \ldots, B^d)$ defined on $E:= E_{\calN} \cong C^{0,\alpha}([0,T]; \bbR^d)$ and its associated iterated Ito integral $\bbB^{ij} := \int B^i \dd B^j$. Define $\fB := (B, \bbB)$.

\subsubsection{As an AWMS}

Our goal is to obtain an AWMS such that the full lift coincides with the Ito lift $\fB$ associated to a Brownian motion $B$ and the enhanced measure $\bmu^{\fB}$ is the distribution of $\fB$ on $\bE$. We will do so by specifying Top-Down data and computing the proxy-restriction. \\

The Cameron--Martin space associated to $(E, \mu)$ is given by 

\begin{align}
    \sH &= \bigoplus_{i=1}^d \sH_i, \quad \Vert h \Vert^2_{\sH} = \sum_{i = 1}^d \Vert h \Vert^2_{\sH_i}, \quad \langle h, k \rangle_{\sH_i} = \int (h^i)'_s (k^i)'_s \dd s, \\
    \sH_i &= \left\{ h \in L^2([0,T]; \bbR) : \exists h' \in L^2([0,T]; \bbR) ~ \text{s.t.} ~ h_t = \int_0^t h'_s \dd s ~ \text{for every} ~ t \in [0,T] \right\}.
\end{align}

\begin{prop} \label{prop:Ito_is_Top_Down}
    Let $\fB: E \rightarrow \bE$ denote the Ito lift of a Brownian motion $B$ as defined above and let $\bmu^{\fB}$ be the distribution of $\fB$ on $\bE$. Then $(\calT, \bE, [\cdot], \calN), \bmu^{\fB}, \fB$ is Top-Down data in the sense of Theorem \ref{thm:top_down_construction}.
\end{prop}
\begin{proof}
    The fact that $(\calT, \bE, [\cdot], \calN)$ is an ambient space is immediate. Since $\fB$ is Borel-measurable $\bmu^{\fB}$ is a Borel probability measure on $\bE$. By construction $\mu := \pi_{\ast} \bmu^{\fB}$ is a Gaussian measure (the classical Wiener measure) on the space $E \cong C^{0,\alpha}([0,T]; \bbR^d)$, $\pi \circ \fB = \id_E$ $\mu$-a.s.\!, and $\fB_{\ast} \mu = \bmu^{\fB}$. To see that $\fB_{\tau} \in \calP^{(\leq [\tau])}(E, \mu; E_{\tau})$ for any $\tau \in \calT$, note that by Proposition \ref{prop:WIC_equiv_characterizations}(iv) it is enough to show that $\left[\fB_{\tau}(\cdot)\right]_{s,t} \in \calP^{(\leq [\tau])}(E, \mu; \bbR)$ for $s,t \in [0,T]$ and recall that $\calP^{([\tau])}(E, \mu; \bbR)$ is the subspace of $L^2(E, \mu; \bbR)$ generated by $[\tau]$-fold Ito integrals; see \cite[Prop. 1.1.4]{nualartMalliavinCalculusRelated2006}. Therefore, in fact, not only
    
    \begin{equation}
        \fB_{\tau} \in \calP^{(\leq [\tau])}(E, \mu; E_{\tau}), \quad \text{but} \quad \fB_{\tau} \in \calP^{([\tau])}(E, \mu; E_{\tau}). \label{eq:better_yet}
    \end{equation}
\end{proof}

\paragraph{Computation of the Proxy-Restriction}

By Proposition \ref{prop:Ito_is_Top_Down} we may apply the Top-Down Construction Theorem \ref{thm:top_down_construction} and compute the skeleton lift as the proxy-restriction of $\fB$.  

\begin{prop} \label{prop:proxy_restr_of_Ito}
    In the context of this section, for every $h \in \sH$, $s,t \in [0,T]$, $1 \leq i,j \leq d$ 

    \begin{equation}
        \left[\overline{\fB_{i}}(h) \right]_t = h^i_t, \quad [\overline{\fB_{ij}}(h)]_{s,t} = \int_s^t h^i_{s,r} \dd h^j_r ,
    \end{equation}

    as iterated integrals in the Young-sense. In particular, the proxy-restriction of the Ito lift $\fB$ and the proxy-restriction of the Gaussian rough path lift $\fX$ of a Brownian motion $B$ coincide. 
\end{prop}
\begin{proof}
    By \eqref{eq:better_yet} we immediately have $\fB^{\circ} = \fB$. To compute the proxy-restriction let $h \in \sH$, $s,t \in [0,T]$ and $1 \leq i,j \leq d$ be arbitrary. Then since $\mu$ is the distribution of a Brownian motion starting at $0$ we obtain

\begin{equation}
    \left[\overline{\fB_{i}}(h) \right]_t = \int_E \left[\fB_{i}(x + h)\right]_t \dd \mu(x) = \int_E x^i_t + h^i_t \dd \mu(x) = h^i_t
\end{equation}

and 

\begin{align}
    [\overline{\fB_{ij}}(h)]_{s,t} &= \int_E \left[\fB_{ij}(x + h)\right]_{s,t} \dd \mu(x) \\
    &= \int_E \int_s^t (x + h)^i_{s,r} \dd (x + h)^j_r \dd \mu(x) \\
    &= \underbrace{\int_E \int_s^t x^i_{s,r} \dd x^j_r \dd \mu(x)}_{=: I} + \underbrace{\int_E \int_s^t x^i_{s,r} \dd h^j_r \dd \mu(x)}_{=: II} \\
    &+ \underbrace{\int_E \int_s^t h^i_{s,r} \dd x^j_r \dd \mu(x)}_{=: III} + \int_E \int_s^t h^i_{s,r} \dd h^j_r \dd \mu(x).
\end{align}

$I$ vanishes due to independence of increments, vanishing expectation of Brownian motion, and the martingale property of Ito integrals. $II$ and $III$ vanish due to Fubini and vanishing of expectation of Brownian motion. The proxy-restriction of the lift is therefore $\overline{\fB}(h) = \left( h, \int h  \otimes \dd h \right)$. Since $\sH \subseteq C^{0, 1- \var}([0,T]; \bbR)$ the iterated integral is well defined in the Young-sense and thus the proxy-restriction of the Ito lift coincides with that of the Stratonovich/GRP skeleton lift \eqref{eq:GRP_lift_first_line} - \eqref{eq:GRP_shuffle_2}.
\end{proof}

\subsubsection{Application of Theorems}

\paragraph{Large Deviations} We may apply Theorem \ref{thm:LDP_for_AWMS} to conclude that the family of measures $(\bmu^{\fB}_{\varepsilon})_{\varepsilon > 0}$ defined by 

\begin{equation}
    \bmu^{\fB}_{\varepsilon} (A) =  \mu \left( ( \delta_{\varepsilon} \fB) \in A \right) = \mu \left( \left( \varepsilon B, \varepsilon^2 \bbB \right) \in A \right), \quad A \in \sB_{\bE}, \varepsilon > 0,
\end{equation}

satisfies an LDP with good rate function given by 

\begin{equation}
    \sJ^{\fB}(\bx) = \begin{cases}
    \frac{1}{2} \sum_{i=1}^d \int_0^T \vert [\pi_{i}(\bx)]'_s \vert^2 \dd s & \quad \bx \in \sbH^{\fB} \\ 
    + \infty & \quad \text{else}, 
    \end{cases} \label{eq:Ito_rate_function}
\end{equation}

where $\sbH^{\fB}$ is nothing but 

\begin{equation}
    \sbH^{\fB} = \overline{\fB}(\sH) = \left\{ \left( h, \int h \otimes \dd h \right) : h \in \sH \right\} .
\end{equation}

As a consequence of Proposition \ref{prop:proxy_restr_of_Ito}, if the centered Gaussian process in section \ref{subsec:GaussianRP} is a Brownian motion and $\fX$ consequently the Stratonovich lift, then $\sJ^{\fB} = \sJ^{\fX}$ and 

\begin{equation}
    \sbH^{\fB} = \overline{\fB}(\sH) = \overline{\fX}(\sH) = \sbH^{\fX} .
\end{equation}

This is to be expected because the Ito- and Stratonovich enhancement only differ in a bracket term, which lies in $\calP^{(\leq 1)}$ and thus does not contribute to the proxy-restriction of the lift on level $2$; cf. Remark on p. \pageref{rem:replacing_condition}.

\paragraph{Fernique Estimate}

Furthermore, according to Theorem \ref{thm:exp_int} the measure $\bmu^{\fB}$ satisfies a Fernique estimate. That is, let  

\begin{equation}
    \eta_0 := \inf \left\{ \frac{1}{2} \Vert h \Vert_{\sH}^2 : h \in \sH, \left(\left\Vert h \right\Vert_{\alpha} + \left\Vert \int h \dd h \right\Vert_{2 \alpha}^{\frac{1}{2}} \right) =1 \right\} .
\end{equation}

Then for any $\eta < \eta_0$ and $t \geq 0$

\begin{equation}
    \bmu^{\fB} \left( \bx \in \bE : \vertiii{\bx} \geq t \right) = \mu \left( \Vert B \Vert_{\alpha} + \Vert \bbB \Vert_{2 \alpha}^{\frac{1}{2}} \geq t \right) \lesssim \exp\left( - \eta t^2 \right).
\end{equation}

\subsection{Rough Volatility Regularity Structure}

\paragraph{General Setup}
Consider the regularity structure associated to rough volatility as defined in \cite{bayerRegularityStructureRough2020}. Fix a Hurst parameter $0 < H < \frac{1}{2}$ and $\kappa \in (0, H)$ throughout the section and define $M$ as the smallest integer s.t. $(M+1)(H - \kappa) - \frac{1}{2} - \kappa > 1$. Define the regularity structure consisting of the structure space which is the $\bbR$-linear span of $\{ \Xi, \Xi \mathcal{I}(\Xi), \ldots, \Xi \mathcal{I}(\Xi)^m, 1, \mathcal{I}(\Xi), \ldots, \mathcal{I}(\Xi)^m \}$, and the structure group $G:= \{ \Gamma_h : h \in (\bbR, +) \}$ with $\Gamma_h 1 = 1, \Gamma_h \Xi = \Xi, \Gamma_h \mathcal{I}(\Xi) = \mathcal{I}(\Xi) + h1$ extended to linear operators on the structure space via linearity and multiplicativity. The indexing set is given by the homogeneities of the symbols, which are $\vert \Xi \vert = - \frac{1}{2} - \kappa, ~ \vert 1 \vert = 0$ extended to the rest of the symbols via the rules $\vert \tau \cdot \tau' \vert = \vert \tau \vert + \vert \tau' \vert$ and $\vert \mathcal{I}(\tau) \vert = \vert \tau \vert + (H + \frac{1}{2})$.

\subsubsection{As an AWMS}

As in the previous subsections, it is our goal to define an AWMS such that the full lift coincides with the lift defined in \cite{bayerRegularityStructureRough2020}. We will do so by specifying Top-Down data. 

\paragraph{Definition of the Ambient Space} Define $\calN = \{ \Xi \}$, $\calT^{(i)} = \left\{\Xi \mathcal{I}(\Xi)^{i-1}, \mathcal{I}(\Xi)^{i} \right\}$ with $1 \leq i \leq M$, i.e. the degree $[\tau]$ of a symbol $\tau$ counts the number of (multiplicative) appearances of $\Xi$ in $\tau$. Furthermore, for every $\tau$ define $E_{\tau}$ as the closure of the smooth two parameter functions on $[0,T]$, which we will write as $(s,t) \mapsto f_s(t)$, under the norm

\begin{equation}
    \Vert f \Vert_{E_{\tau}} := \sup_{\lambda \in (0,1]} \sup_{\varphi \in \mathcal{B}} \sup_{s \in [0,T]} \lambda^{- \vert \tau \vert} \left\vert f_s (\varphi_s^{\lambda}) \right\vert.
\end{equation}

Here $\varphi_s^{\lambda}(t) := \frac{1}{\lambda} \varphi(\frac{t - s}{\lambda})$ and, given $s \in [0,T]$ and $\lambda \in (0,1]$, $f_s (\varphi_s^{\lambda})$ denotes the application of $f_s$ to $\varphi_s^{\lambda}$ in the sense of a distribution, and $\mathcal{B}$ denotes the set of smooth space-time functions which are compactly supported in the unit ball and whose value and the value of its derivatives up to order $1$ are bounded by $1$. 

\paragraph{Definition of the Full Lift} Let $\xi$ be a white noise on $[0,T]$ defined on $E$. Let $B$ be a Brownian motion defined by $B_t = \xi(1_{[0,t]})$. Recall that $0 < H < \frac{1}{2}$ denotes the Hurst parameter and define the associated Volterra kernel by $K^{H}(t) := \sqrt{2H} ~ t^{H- \frac{1}{2}} 1_{\{ t > 0\}}$ and a fractional Brownian motion $W^H$ by 

\begin{equation}
    W^H(t) = \sqrt{2H} \int_0^t \vert t - r \vert^{H - \frac{1}{2}} \dd B_r \quad (= K^{H} \ast \xi) .
\end{equation}

Define 

\begin{equation}
    \bbW_{s,t}^H := \int_s^t \left( W^H_{s,r} \right)^m \dd B_r 
\end{equation}

in the sense of an Ito integral and define the full lift $\fV: E \rightarrow \bE$ as

\begin{equation}
    \left[\fV_{\Xi}(\cdot) \right](s, \cdot) = \dot{B}, \quad \left[\fV_{\mathcal{I}(\Xi)^m}(\cdot) \right](s, \cdot) = (W^H_{s, \cdot})^m, \quad \left[\fV_{\Xi \mathcal{I}(\Xi)^m}(\cdot) \right](s, \cdot) = \partial_t \bbW^H \label{eq:RV_lift_3}
\end{equation}

and the enhanced measure $\bmu^{\fV}$ as the distribution of $\fV$. The abstract Wiener space associated to $(E, \mu)$ is 

\begin{equation}    
    E = \mathcal{C}^{0, \alpha}([0,T]; \bbR), \quad \sH = L^2([0,T]; \bbR), \quad \text{and} \quad \mu = \text{Law}(\xi)
\end{equation}

for $\alpha:= \vert \Xi \vert < - \frac{1}{2}$.

\begin{prop} \label{prop:rough_vol_is_top_down}
    Let $(\calT, \bE, [\cdot], \calN)$ be as defined as above, let $\fV$ be defined as in \eqref{eq:RV_lift_3} and let $\bmu^{\fV}$ be its distribution on $\bE$. Then $(\calT, \bE, [\cdot], \calN), \bmu^{\fV}, \fV$ is Top-Down data in the sense of Theorem \ref{thm:top_down_construction}.
\end{prop}
\begin{proof}
    Via a Stone--Weierstrass argument, the spaces $E_{\tau}$ can be seen to be separable (see \cite{hairerLargeDeviationsWhiteNoise2015} for a sketch and Lemma 4.2 in the master's thesis of T. Klose for a full argument). Thus $(\calT, \bE, [\cdot], \calN)$ constitutes an ambient space. By definition, $\hat{\fL}$ is a measurable lift, $\mu := \pi_{\ast} \bmu$ is a centred Gaussian measure (the white noise measure) on $E$, and $\bmu^{\fV} = \fV_{\ast} \mu$. Condition $\pi_{\tau} \circ \fV \in \calP^{(\leq [\tau])}(E, \mu; E_{\tau})$ is satisfied by \cite[Sec. 10.2]{hairerTheoryRegularityStructures2014}. 
\end{proof}

\subsubsection{Computation of the Proxy-Restriction}

\begin{prop}
    In the context of this section, for every $h \in \sH$ and $s,t \in [0,T]$
    
    \begin{equation}
        \left[ \overline{\fV}_{\Xi}(h) \right]_t = h_t, \quad \left[\overline{\fV}_{\mathcal{I}(\Xi)}(h) \right]_{s,t} = \int_0^{s \vee t} \left( K^H(t-u) - K^H(s-u) \right) h(u) \dd u,
    \end{equation}
    
    extended to all of $\mathcal{T}$ by $\overline{\fV}_{\tau \tau'}(h) = \overline{\fV}_{\tau}(h) \cdot \overline{\fV}_{\tau'}(h)$.
\end{prop}
\begin{proof}
    See \cite[Lem. 4.1]{bayerRegularityStructureRough2020} and \cite[Lem. B1]{bayerRegularityStructureRough2020}.
\end{proof}

In other words, the proxy-restriction of the full lift coincides with the formal application to elements in $\sH$. 

\subsubsection{Application of Theorems}

\paragraph{Large Deviations}

Theorem \ref{thm:LDP_for_AWMS}, in conjunction with Proposition \ref{prop:rough_vol_is_top_down}, implies that the family of measures $\left( \bmu^{\fV}_{\varepsilon} \right)_{\varepsilon > 0}$ satisfies an LDP with good rate function  

\begin{equation}
    \sJ^{\fV}(\bx) = \begin{cases}
    \frac{1}{2} \Vert \pi (\bx) \Vert_{L^2([0,T]; \bbR)}^2  & \quad \bx \in \sbH^{\fV} \\ 
    + \infty & \quad \text{else}, 
    \end{cases}
\end{equation}

where $\sbH^{\fV}$ is nothing but 

\begin{equation}
    \overline{\fV}(\sH) = \left\{ \left(h, K^H \ast h, h \left(K^H \ast h\right), \left(K^H \ast h\right)^2, h \left(K^H \ast h\right)^2, \ldots \right) : h \in L^2([0,T]; \bbR) \right\}.
\end{equation}

\begin{rem}
    Such LDP results have proven very useful in volatility modelling. See e.g. \cite{tsatsoulisSpectralGapStochastic2018}.
\end{rem}

\paragraph{Fernique Estimate}

According to Theorem \ref{thm:exp_int} the measure $\bmu^{\fV}$  satisfies a Fernique estimate. That is, let 

\begin{align}
    \eta_0 := \inf \Bigg\{ \frac{1}{2} \Vert h \Vert_{\sH}^2 : h \in \sH, &\Vert h \Vert_{E_{\Xi}} + \Vert K^H \ast h \Vert_{E_{\calI(\Xi)}} \\
    + &\Vert h(K^H \ast h) \Vert_{E_{\Xi \calI(\Xi)}}^{\frac{1}{2}} + \Vert (K^H \ast h)^2 \Vert_{E_{ \calI(\Xi)^2}}^{\frac{1}{2}} \\
    + &\Vert h(K^H \ast h)^2 \Vert_{E_{\Xi \calI(\Xi)}}^{\frac{1}{3}} + \Vert (K^H \ast h)^3 \Vert_{E_{ \calI(\Xi)^3}}^{\frac{1}{3}} + \ldots =1 \Bigg\} .
\end{align}

Then for any $\eta < \eta_0$ and $t \geq 0$ we have $\bmu^{\fV} \left( \bx \in \bE : \vertiii{\bx} \geq t \right) \lesssim \exp\left( - \eta t^2 \right)$.

\subsection{\texorpdfstring{$\Phi^4_d$ for $d = 2, 3$}{Phi4d for d = 2, 3}} \label{subsec:Phi}

In the following we want to apply the developed machinery to the case of the $\Phi^4_d$-model with $d \in \{2,3\}$, which is the object of study in \cite{hairerLargeDeviationsWhiteNoise2015}. 

\paragraph{General Setup}
Let $T > 0$ be fixed throughout the section, let $\mathbb{T}^d$ denote the $d$-dimensional torus and consider a regularity structure associated to the problem 

\begin{equation}
    \partial_t \phi = \Delta \phi + C \phi - \phi^3 + \varepsilon \xi, \tag{$\Phi^4_d$} \label{eq:Phi4}
\end{equation}

where $\phi$ is a scalar field on $[0,T] \times \mathbb{T}^d$, $C \in \bbR$, $\xi$ is space-time white noise, and $\varepsilon > 0$. 

Recall that a model for a regularity structure is uniquely determined by its minimal model, which need only be defined on the set of trees with negative homogeneities (see \cite[Prop. 3.31, Thm. 5.14]{hairerTheoryRegularityStructures2014}). In the case of $\Phi^4_d$ this set is given by\footnote{Our tree notation is identical to that of \cite{hairerLargeDeviationsWhiteNoise2015}.}

\begin{equation}
    \calT_{2} = \left\{ \Xi, \<1>, \<2>, \<3> \right\} \quad \text{and} \quad \calT_{3} = \left\{ \Xi, \<1>\ , \<2>\ , \<3>\ , \<32>\ , \<22>\ , \<31> \right\} . \label{eq:W_for_Phi}
\end{equation}

Define the homogeneities by the usual rules (i.e. $\vert \Xi \vert = - \frac{d+2}{2} - \kappa$ for some $\kappa \in (0, \frac{1}{14})$ which shall be fixed throughout this section, $\vert 1 \vert = 0$, adding an edge increases the homogeneity by $2$ and multiplying two symbols adds the homogeneities)\footnote{As explained in \cite[Sec. 9.1, Sec 9.3]{hairerTheoryRegularityStructures2014} the upper bound on $\kappa$ serves only to ensure that $\calT$ as defined in \eqref{eq:W_for_Phi} contains all symbols of negative homogeneity.} and the structure group by the usual Hopf-algebraic construction (see \cite[Sec. 2.2]{hairerLargeDeviationsWhiteNoise2015} or \cite[Sec. 8.1]{hairerTheoryRegularityStructures2014}). 

\subsubsection{As an AWMS}

As before we want to construct an AWMS for which the full lift coincides with the lift defined in \cite{hairerLargeDeviationsWhiteNoise2015}. 

\paragraph{Definition of the Ambient Space} Fix $d \in \{2,3\}$. Define the set of symbols $\calT_d$ as in \eqref{eq:W_for_Phi}, $\calN_2 = \calN_3 = \{ \Xi \}$, and for any symbol in $\calT_{d}$ define the space $E_{\tau}$ as the closure of the smooth two-parameter functions on $[0,T] \times \mathbb{T}^d$, which we will write as $(z, z') \mapsto f_z(z')$ under the norms

\begin{align}
    \Vert f \Vert_{E_{\Xi}} &:= \sup_{\lambda \in (0,1]} \sup_{\varphi \in \mathcal{B}} \sup_{s \in \bbR} \sup_{z \in [0,T] \times \mathbb{T}^d} \lambda^{- \vert \Xi \vert} \left\vert 1_{[0,t]}(s) f_z (\varphi_z^{\lambda}) \right\vert, \\
    \Vert f \Vert_{E_{\<1>}} &:= \sup_{\lambda \in (0,1]} \sup_{\psi \in \mathcal{B}_0} \sup_{z \in [0,T] \times \mathbb{T}^d} \lambda^{- \vert \<1> \vert} \left\vert f_z (t, \cdot) (\psi_x^{\lambda}) \right\vert, \\
    \Vert f \Vert_{E_{\tau}} &:= \sup_{\lambda \in (0,1]} \sup_{\varphi \in \mathcal{B}} \sup_{z \in [0,T] \times \mathbb{T}^d} \lambda^{- \vert \tau \vert} \left\vert f_z (\varphi_z^{\lambda}) \right\vert, \quad \tau \in \calT_d\setminus \{ \Xi, \<1> \},
\end{align}

respectively, where 

\begin{equation}
    \varphi_z^{\lambda}(z') := \lambda^{-d+2} \varphi \left( \lambda^{-2}(t' - t), \lambda^{-1} (x' - x) \right) , \quad  z = (t,x), z' = (t',x') \in \bbR \times \bbR^d,
\end{equation}

and $\calB_0$ is a space of test functions defined analogously to $\calB$, but only in the spatial variable $x$. Functions in $\calB_0$ are rescalled analogously to those in $\calB$. For a fixed $z \in [0,T] \times \bbT^d$, $f_z (\varphi)$ denotes the application of $f_z$, viewed as a distribution, to the test function $\varphi$, in the second variable of $f$. Define the integer degree $[\tau]$ of a symbol $\tau \in \calT_d$ as the number of leaves in the tree representing $\tau$, e.g. $[\<2>] = 2$, $[\<31>] = 4$.

\paragraph{Definition of the Full Lift}

Let $\mathfrak{s} \label{symb:parabolic_scaling}$ denote the parabolic scaling on $[0,T] \times \mathbb{T}^d$, let $\xi$ be space-time white noise on $[0,T] \times \mathbb{T}^d$ defined on $E := E_{\calN} = E_{\Xi} \cong \mathcal{C}_{\mathfrak{s}}^{0, - \frac{d+2}{2} - \kappa}([0,T] \times \mathbb{T}^d; \bbR)$, let $\mu$ be the distribution of $\xi$, and let $\tau \in \calT_d$ be fixed throughout the section. The abstract Wiener space associated to $(E, \mu)$ is 

\begin{equation}    
    E = \mathcal{C}_{\mathfrak{s}}^{0, \alpha} ([0,T] \times \mathbb{T}^d; \bbR), \quad \sH = L^2([0,T] \times \bbT^d; \bbR), \quad \text{and} \quad \mu = \text{Law}(\xi) .
\end{equation}

for $\alpha := - \frac{d+2}{2} - \kappa < - \frac{d+2}{2}$. Let $\rho$ be a smooth compactly supported function on $[0,T] \times \bbT^d$ with $\int \rho = 1$ and define 

\begin{equation}
\rho^{\delta}(t,x) := \delta^{-(d+2)} \rho( \delta^{-2} t, \delta^{-1} x ) \quad \text{and} \quad \fP_{\tau, \delta} := \fP( \cdot \ast \rho_{\delta}).
\end{equation}

Let $\Ren_{\delta} \fP_{\tau, \delta} \label{symb:renormalization}$ denote the renormalized minimal model lifts at correlation length $\delta > 0$ as defined in \cite[Sec. 2.6, Eq. (2.21), Eq. (2.22)]{hairerLargeDeviationsWhiteNoise2015}. By \cite[Thm 10.7, Thm. 10.22]{hairerTheoryRegularityStructures2014} the family $(\Ren_{\delta} \fP_{\tau, \delta})_{\delta > 0}$ converges in $L^2(E, \mu; E_{\tau})$ to a limit $\fP_{\tau}$. Define this limit as the full lift and let $\bmu^{\fP}$ denote the distribution of $\fP$ on $\bE$.

\begin{prop} \label{prop:Phi_4_is_top_down}
    The data associated to the $\Phi^4_d$-model $(\calT, \bE, [\cdot], \calN), \bmu^{\fP}, \fP$ is Top-Down data in the sense of Theorem \ref{thm:top_down_construction} and $\fP$ can be identified with the renormalized minimal model lift in the sense of \cite{hairerLargeDeviationsWhiteNoise2015}.
\end{prop}
\begin{proof}
    As in Proposition \ref{prop:rough_vol_is_top_down}, a Stone--Weierstrass argument shows that the spaces $E_{\tau}$ are separable. By construction $\fP$ is Borel-measurable and thus $\bmu^{\fP}$ is a Borel probability measure. Since $\lim_{\delta \downarrow 0} \Ren_{\delta} \fP_{\Xi, \delta} = \lim_{\delta \downarrow 0} \fP_{\Xi, \delta}= \fP_{\Xi}$ the measure $\mu := \pi_{\ast} \bmu^{\fP}$ is centred Gaussian (the white noise measure) and $\pi \circ \fP = \id_{E}$ $\mu$-a.s.\! Lastly, condition $\pi_{\tau} \circ \fP \in \calP^{(\leq [\tau])}(E, \mu; E_{\tau})$ is satisfied by \cite[Sec. 10.2]{hairerTheoryRegularityStructures2014}.
\end{proof}

\paragraph{Computation of the Proxy-Restriction}

The computation of $\overline{\fP}$ has already been done \cite[Sec. 4]{hairerLargeDeviationsWhiteNoise2015} and in much greater detail in Section 4.2.3 in the master's thesis of T. Klose. For the reader's convenience we spell out the argument in detail and adapted to our notations in the supplement \cite[Sec. 7.4.]{chiusoleAbstractWienerModel2024}, but quote the result here:

Let $\tau \in \calT_d$, $y \in [0,T] \times \bbT^d$, $\varphi \in C_c^{\infty}([0,T] \times \bbT^d)$, $0 \leq k \leq [\tau]$ be arbitrary and let $(z_1, \ldots, z_k) \mapsto \langle \calW_{\tau, k}(\cdot, y; z_1, \ldots, z_{k}), \varphi \rangle$ be the kernels in $\sH^{\otimes k}$ such that 

\begin{equation}
    \left[ \fP_{\tau} \right](y,\varphi) = \sum_{0 \leq k \leq [\tau]} I_k \left( \langle \calW_{\tau, k}(\cdot, y), \varphi \rangle \right),
\end{equation}

where $I_k$ is the $k$-th level Ito isometry $L^2(([0,T] \times \bbT^d)^{k}) \rightarrow \calP^{(k)}(E, \mu; \bbR)$. Then, for any $h \in \sH$, the proxy restriction takes the form

\begin{equation}
   \left[\overline{\fP_{\tau}}(h)\right](y, \varphi) = \int \ldots \int \langle \calW_{\tau, [\tau]}(\cdot, y; z_1, \ldots, z_{[\tau]}), \varphi \rangle h(z_1) \ldots h(z_{[\tau]}) \dd z_1 \ldots \dd z_{[\tau]}.
\end{equation}

\subsubsection{Application of Theorems}

\paragraph{Large Deviations}

Theorem \ref{thm:LDP_for_AWMS} implies that the family of measures $\left( \bmu^{\fP}_{\varepsilon} \right)_{\varepsilon > 0}$ satisfies an LDP on $\bE$ with good rate function  

\begin{equation}
    \sJ^{\fP}(\bx) = \begin{cases}
    \frac{1}{2} \Vert \pi (\bx) \Vert_{L^2([0,T] \times \mathbb{T}^d; \bbR)}^2  & \quad \bx \in \sbH^{\fP} \\ 
    + \infty & \quad \text{else}. 
    \end{cases}
\end{equation}

where $\sbH^{\fP} = \overline{\fP}(L^2([0,T] \times \mathbb{T}^d; \bbR))$. This is an abstract version of \cite[Thm. 4.3]{hairerLargeDeviationsWhiteNoise2015}, which is the main result in \cite{hairerLargeDeviationsWhiteNoise2015}.

\paragraph{Fernique Estimate}

According to Theorem \ref{thm:exp_int} the measure $\bmu^{\fP}$ satisfies a Fernique estimate. That is, let 

\begin{equation}
    \eta_0 := \inf \left\{ \frac{1}{2} \Vert h \Vert_{\sH}^2 : h \in L^2([0,T] \times \mathbb{T}^d; \bbR), \vertiii{\overline{\fP}(h)} = 1 \right\} .
\end{equation}

Then for any $\eta < \eta_0$ and $t \geq 0$

\begin{equation}
    \bmu^{\fP} \left( \bx \in \bE : \vertiii{\bx} \geq t \right) = \mu \left( \vertiii{\fP} \geq t \right) \lesssim \exp\left( - \eta_0 t^2 \right).
\end{equation}

\begin{rem}
    In principle, there is nothing in the way of applying the framework in this article to more complicated equations and/or with a higher number of trees/closer to criticality or with other types of (Gaussian) noise, as for instance in \cite{chandraPrioriBounds$Phi2023}, where the authors treat the $\Phi^4$ equation in the full subcritical regime $d < 4$, modeling fractional dimension via a slight spatial coloring of the (Gaussian) driving noise. However, note that if the noise is defined on an unbounded region of spacetime, the norms of the $E_{\tau}$ spaces typically need to be weighted in some form or replaced by families of semi-norms. While the former is within the scope of this article, the latter is not. However, given that the classical theory of abstract Wiener spaces generalizes well to separable Fréchet spaces, we are confident that AWMS could be generalized in a similar manner to the case where $\bE$ is replaced by a separable and graded Fréchet space.
\end{rem}

\subsection{\texorpdfstring{Parabolic Anderson Model, $d = 2$}{Parabolic Anderson Model, d = 2}}

In the following we treat an example where the driving noise is spatial, not space-time. 

\paragraph{General Setup}

Fix $T > 0$ throughout the section and consider a regularity structure which is associated to the problem 

\begin{equation}
    \partial_t u = \Delta u + u \varepsilon \zeta, \tag{PAM} \label{eq:PAM}
\end{equation}

where $u$ is a scalar field on the $2$-dimensional torus $\mathbb{T}^2$, $\zeta$ is spatial white noise, and $\varepsilon > 0$. Again, since minimal models carry enough information to recover a full regularity structure to solve \eqref{eq:PAM} we define $\calN = \calT^{(1)} = \{ \Xi \}$, $\calT^{(2)} = \{ \Xi \mathcal{I}(\Xi) \}$. Define the homogeneities by the usual rules (i.e. $\vert \Xi \vert = - 1 - \kappa$, for some $\kappa \in (0, \frac{1}{3})$ which shall be fixed throughout this section, $\vert 1 \vert = 0$, $\mathcal{I}(\tau) = \vert \tau \vert + 2$ and multiplying two symbols adds their homogeneities) and the structure group by the usual Hopf-algebraic construction (see \cite[Sec. 2.2]{hairerLargeDeviationsWhiteNoise2015} or \cite[Sec. 8.1]{hairerTheoryRegularityStructures2014}). 

\subsubsection{As an AWMS}

\paragraph{Definition of the Ambient Space}

Define the set of symbols $\calT$ as above and for any symbol $\tau \in \calT$ define the space $E_{\tau}$ as the closure of the smooth two-parameter functions on $[0,T] \times \mathbb{T}^d$, which we will write as $(z, z') \mapsto f_z(z')$ under the norms

\begin{equation}
    \Vert f \Vert_{E_{\tau}} := \sup_{\lambda \in (0,1]} \sup_{\varphi \in \mathcal{B}} \sup_{z \in [0,T] \times \mathbb{T}^d} \lambda^{- \vert \tau \vert} \left\vert f_z (\varphi_z^{\lambda}) \right\vert .
\end{equation}

Let $\alpha:= -1 - \kappa  < - \frac{1}{2}$ and let $\zeta$ be spatial white noise defined on $E = E_{\calN} \cong \mathcal{C}_{\mathfrak{s}}^{0, \alpha}([0,T] \times \mathbb{T}^d; \bbR)$. The abstract Wiener space associated to $(E, \mu)$ is 

\begin{equation}    
    E = \mathcal{C}^{0, - 1 - \kappa}([0,T] \times \mathbb{T}^d; \bbR), \quad \sH = L^2(\bbT^d; \bbR), \quad \text{and} \quad \mu = \text{Law}(\zeta) .
\end{equation}

\paragraph{Definition of the Full Lift}

As in Subsection \ref{subsec:Phi} define the full lift $\fA_{\tau}$ as the $L^2(E, \mu; E_{\tau})$-limit of the renormalized minimal admissible models $\Ren_{\delta} \fA_{\tau, \delta}$ and the enhanced measure $\bmu^{\fA}$ as the distribution of $\fA$.

\begin{prop}
    The data associated to the parabolic Anderson model $(\calT, \bE, [\cdot], \calN), \bmu^{\fA}, \fA$ is Top-Down data in the sense of Theorem \ref{thm:top_down_construction} and $\fA$ can be identified with the renormalized minimal model lift in the sense of \cite{frizPreciseLaplaceAsymptotics2022}.
\end{prop}
\begin{proof}
    See Proposition \ref{prop:Phi_4_is_top_down}.
\end{proof}

\paragraph{Computation of the Proxy-Restriction}

The computation of the proxy-restriction $\overline{\fA}$ of the full lift $\fA$ follows in the same way as in Subsection \ref{subsec:Phi}.

\subsubsection{Application of Theorems}

\paragraph{Large Deviations}

Applying Theorem \ref{thm:LDP_for_AWMS} shows that the family of measures $\left( \bmu^{\fA} \right)_{\varepsilon > 0}$ satisfies a large deviation principle on $\bE$ with good rate function given by 

\begin{equation}
    \sJ^{\fA}(\bx) = \begin{cases}
    \frac{1}{2} \Vert \pi (\bx) \Vert_{L^2(\mathbb{T}^d; \bbR)}^2  & \quad \bx \in \sbH^{\fA} \\ 
    + \infty & \quad \text{else}, 
    \end{cases}
\end{equation}

where $\sbH^{\fA} = \overline{\fA}(L^2(\bbT^d; \bbR))$. See also e.g. \cite[Thm. C.3]{frizPreciseLaplaceAsymptotics2022}.

\paragraph{Fernique Estimate}

As for $\Phi^4_d$, by Theorem \ref{thm:exp_int} the measure $\bmu^{\fA}$ satisfies a Fernique estimate. That is, let 

\begin{equation}
    \eta_0 := \inf \left\{ \frac{1}{2} \Vert h \Vert_{L^2(\bbT^d; \bbR)}^2 : h \in \sH, \vertiii{\overline{\fA}(h)} = 1 \right\} .
\end{equation}

Then for any $\eta < \eta_0$ we have $\bmu^{\fA} \left( \bx \in \bE : \vertiii{\bx} \geq t \right) = \mu \left( \vertiii{\fA} \geq t \right) \lesssim \exp\left( - \eta_0 t^2 \right)$.

%----------------------------
%APPENDIX
\appendix

\section*{Appendix}

In the appendix we will collect some notation and background.

\section{Abstract Wiener Spaces} \label{sec:AWS}

Let $E$ be a separable Banach space and let $\mu$ be a centred Gaussian Borel probability measure on $E$. Then define 

\begin{equation}
    \fq_{\mu}(f,g) := \int_E f(x) g(x) \dd \mu(x), \quad f,g \in E^{\ast}, 
\end{equation}

as the \textbf{covariance form} of $\mu$ on $E$. Define the assignment 

\begin{equation}
    j: E^{\ast} \rightarrow L^2(E, \mu; \bbR), \quad f \mapsto [f]_{\mu}, 
\end{equation}

where $[f]_{\mu}$ denotes the $\mu$-a.s.\! equivalence class of $f$. By Fernique's Theorem \ref{thm:fernique} the assignment $j$ is a well-defined function and by definition it is injective precisely when $\mu$ is non-degenerate. Define the \textbf{reproducing kernel Hilbert space} \label{symb:reproducing_kernel_hilbert_space} of $\mu$ in $E$ as 

\begin{equation}
    \sR(\mu) := \overline{j(E^{\ast})}^{L^2(E, \mu; \bbR)},
\end{equation}

and the covariance operator \label{symb:covariance_operator} $\fC_{\mu}: E^{\ast} \rightarrow E \subseteq E^{\ast \ast}$ by

\begin{equation}
    \left[ \fC_{\mu} f \right] (g) = \fq_{\mu}(f,g) .
\end{equation}

It is not clear a priori that $\fC_{\mu}$ maps $E^{\ast}$ into the canonical inclusion of $E$ into its double dual $E^{\ast \ast}$ (instead of just into $E^{\ast \ast}$). However, this can be shown to be the case even if $E$ is replaced by a Fréchet space - see \cite[Thm. 3.2.1.]{bogachevGaussianMeasures1998}. Recall that the Cameron--Martin space $\sH_{\mu}$ is the isometric image of $\sR(\mu)$ under the operator $\fC_{\mu}$ - see e.g. \cite[Chap. 4]{hairerIntroductionStochasticPDEs2023}. In particular,

\begin{equation}
    \underline{h} := \fC_{\mu}^{-1} h \label{symb:underline_h}
\end{equation}

is (as an $L^2$-limit of Gaussian random variables) Gaussian with distribution $\sN(0, \Vert h \Vert_{\sH}^2)$. In particular, the moment generating function of $\underline{h}$ has the form 

\begin{equation}
    \bbE \left[ \exp \left( \lambda \underline{h} \right) \right] = \exp \left( \Vert h \Vert_{\sH}^2 \frac{\lambda^2}{2} \right), \quad \lambda \in \bbR.
\end{equation}

\begin{prop}[Integrability of Cameron--Martin Density] \label{prop:int_of_CM_density}
    Let $(E, \sH, i, \mu)$ be an abstract Wiener space and let $h \in \sH$. Let 

    \begin{equation}
        f_h(x) := \exp \left( \underline{h}(x) - \frac{1}{2} \Vert h \Vert_{\sH}^2 \right), \quad x \in E.
    \end{equation}

    Then for every $1 \leq p < \infty$

    \begin{equation}
        \Vert f_h \Vert_{L^p(E, \mu; \bbR)} \leq \exp \left( \Vert h \Vert_{\sH}^2 \frac{p^2}{2} \right) < \infty.
    \end{equation}
\end{prop}
\begin{proof}
    Let $1 \leq p < \infty$ be arbitrary. Then 

    \begin{align}
        &\Vert f_h \Vert_{L^p(E, \mu; \bbR)}^p = \int_E \left\vert \exp \left( \underline{h}(x) - \frac{1}{2} \Vert h \Vert_{\sH}^2 \right) \right\vert^p \dd \mu(x) \\
        = &\int_E \exp \left( p \underline{h}(x) - \frac{p}{2} \Vert h \Vert_{\sH}^2 \right) \dd \mu(x) \lesssim \int_E \exp \left( p \underline{h} \right) \dd \mu = \exp \left( \Vert h \Vert_{\sH}^2 \frac{p^2}{2} \right).
    \end{align}
\end{proof}

\begin{prop} \label{prop:L2L1} 
    Let $(E, \sH, i, \mu)$ be an abstract Wiener space, let $B$ be a separable Banach space, let $h \in \sH$, and let $X_n \rightarrow X \in L^2(E, \mu; B)$. Then $X_n \rightarrow X \in L^1(E, \mu_h; B)$, where $\mu_h := (T_h)_{\ast} \mu$. In other words, 

    \begin{equation}
        X_n(\cdot + h) \rightarrow X(\cdot + h) \quad \text{in} ~ L^1(E, \mu; B) .
    \end{equation}
\end{prop}
\begin{proof}
    Let $h \in \sH$ be arbitrary. Then by the classical Cameron--Martin Theorem \ref{thm:classical_CM} with $f_h := \frac{\dd \mu_h}{\dd \mu}: E \rightarrow [0, \infty)$

    \begin{align}
        \Vert X_n - X \Vert_{L^1(E, \mu_h; B)} &= \int_E \Vert X_n - X \Vert_B \dd \mu_h = \int_E \Vert X_n - X \Vert_B f_h \dd \mu \\
        &\leq \underbrace{\Vert X_n - X \Vert_{L^2(E, \mu; B)}}_{\rightarrow 0} \underbrace{\Vert f_h \Vert_{L^2(E, \mu; \bbR)}}_{< \infty} .
    \end{align}

    where we used the Cauchy--Schwarz inequality in the last line. The latter of the two terms is finite by Proposition \ref{prop:int_of_CM_density}. Thus $\Vert X_n - X \Vert_{L^1(E, \mu_h; \bbR)} \rightarrow 0$. 
\end{proof}

\begin{prop} \label{prop:prob_prob_h}
    Let $(E, \sH, i, \mu)$ be an abstract Wiener space, let $B$ be a separable Banach space, let $h \in \sH$, and let $X_n \rightarrow X$ in probability w.r.t. $\mu$. Then $X_n \rightarrow X$ in probability w.r.t. $\mu_h$, where $\mu_h := (T_h)_{\ast} \mu$.
\end{prop}
\begin{proof}
    Let $\eta > 0$. Then 

    \begin{align}
        &\mu \left( \Vert X_n \circ T_h - X \circ T_h \Vert_B > \eta \right) = \mu_h \left( \Vert X_n - X \Vert_B > \eta \right) \\
        &= \int_E 1_{\left\{ \Vert X_n - X \Vert_B > \eta \right\}} f_h \dd \mu \leq \mu \left( \Vert X_n - X \Vert_B > \eta \right)^{\frac{1}{2}} \underbrace{\Vert f_h \Vert_{L^2(E, \mu; \bbR)}}_{< \infty}, \label{eq:prob_prob_h}
    \end{align}

    where the latter term is finite by Proposition \ref{prop:int_of_CM_density}. Since $X_n \rightarrow X$ w.r.t. $\mu$ the expression in \eqref{eq:prob_prob_h} goes to $0$ as $n \rightarrow \infty$, which shows the claim. 
\end{proof}

\begin{lem}[Convergence of Proxy-Restriction]\label{lem:L2_convergence_implies_pointwise_convergence_of_homogeneous_part} %DONE 
    Let $(\calT, \bE, [\cdot], \calN)$ be an ambient space, let $(E, \sH, i, \mu)$ be an abstract Wiener space, and let $\Psi_{\delta}, \Psi \in \calP^{(\leq [\calT])}(E, \mu; \bE)$ for every $\delta > 0$ such that $\Psi_{\delta} \rightarrow \Psi$ in $L^2(E, \mu; \bE)$. Then $\overline{\Psi_{\delta}}(h) \rightarrow \overline{\Psi}(h)$ uniformly on bounded sets of $\sH$. 
\end{lem}
\begin{proof}
    Let $A \subseteq \sH$ be bounded and let $h \in A$ be arbitrary. Then 

    \begin{equation}
       \overline{\Psi_{\delta}}(h) - \overline{\Psi}(h) = \bbE \left[ \left( \Psi_{\delta} - \Psi \right)^{\circ}( \cdot + h) \right] = \int_E \left( \Psi_{\delta} - \Psi \right)^{\circ} (x) \dd \mu_{h}(x)
    \end{equation}

where $\mu_h (\cdot) = \mu(\cdot - h)$. Via the Cameron--Martin Theorem and the Cauchy--Schwarz inequality we obtain 

\begin{align}
    \sup_{h \in A}  \left\Vert \overline{\Psi_{\delta}}(h) - \overline{\Psi}(h) \right\Vert_{\bE} &= \sup_{h \in A} \left\Vert \int \left( \Psi_{\delta} - \Psi \right)^{\circ} (x) f_h(x) \dd \mu (x) \right\Vert_{\bE} \\
    &\leq \sup_{h \in A} \int \underbrace{\Vert \left( \Psi_{\delta} - \Psi \right)^{\circ} (x)\Vert_{\bE}}_{\in L^2(E, \mu; \bbR)} \underbrace{\vert f_h(x) \vert}_{\in L^2(E, \mu; \bbR)} \dd \mu (x) \\
    &\leq \underbrace{\Vert \left( \Psi_{\delta} - \Psi \right)^{\circ} \Vert_{L^2(E, \mu; \bE)}}_{ \rightarrow 0} \underbrace{\sup_{h \in A} \Vert f_h \Vert_{L^2(E, \mu; \bbR)}}_{<\infty} \rightarrow 0
\end{align} 

where $f_h$ denotes the Cameron--Martin density, the square-norm of which is bounded uniformly on bounded subsets of $\sH$ by Proposition \ref{prop:int_of_CM_density}.
\end{proof}

\begin{prop} \label{prop:containment_of_image_CM}
    Let $(E, \sH, i, \mu)$ be an abstract Wiener space, let $B$ be a separable Banach space and let $\Phi: E \rightarrow B$ be a bounded linear operator. Then $\nu := \Phi_{\ast} \mu$ is a Gaussian measure and 

    \begin{equation}
        \Phi(\sH) \subseteq \texttt{CM}(B, \nu),
    \end{equation}

    where $\texttt{CM}(B, \nu)$ denotes the Cameron--Martin space of $\nu$ in $B$. 
\end{prop}
\begin{proof}
    We use the characterization of $\texttt{CM}(B, \Phi_{\ast} \mu)$ as the subspace of $B$ consisting precisely of those elements $g \in B$ such that $(T_{g})_{\ast}\nu$ and $\nu$ are equivalent - see Theorem \ref{thm:classical_CM}. Let $h \in \sH$ and let $A \subseteq B$ be measurable. Then we have 

    \begin{align}
        [(T_{\Phi(h)})_{\ast}\nu](A) &= \nu (A - \Phi(h)) = \mu( x \in E : \Phi(x) \in A - \Phi(h)) \\
        &= \mu( x \in E : \Phi(x+h) \in A) = \mu ( x \in E: x + h \in \Phi^{-1}(A) ) \\
        &= \mu_{h} ( \Phi^{-1}(A)) 
    \end{align}

    Thus, by Theorem \ref{thm:classical_CM} the above is zero precisely when $\mu( \Phi^{-1}(A)) = \nu(A)$ is zero. Hence $\Phi(h) \in \texttt{CM}(B, \nu)$.
\end{proof}

\section{Banach Valued Wiener--Ito Chaos} \label{sec:WIC}

Throughout this section let $(E, \sH, i, \mu)$ be an abstract Wiener space, let $(e_i)_{i \in \bbN}$ be an ONB of $\sH$ which is contained in $E^{\ast}$,\footnote{Such a choice of ONB is always possible, even if the $E$-closure of $\sH$ does not coincide with $E$; i.e. if $\mu$ does not have full support.} and let

\begin{equation}
\calF_N := \sigma \left( \langle e_k, \cdot \rangle : 1 \leq k \leq N \right).
\end{equation}

Let $(\calT, \bE, [\cdot], \calN)$ be an ambient space (in the sense of Definition \ref{defn:ambient_space}) and recall the notation $A^{k}_n$ and $A^{\leq k}_n$ from \eqref{symb:Akn} on p.~\pageref{symb:Akn}. Let $(h_k)_{k \geq 0}$ be the family of Hermite polynomials defined by 

%% One can always choose an ONB of $\sH$ which is contained in $E^{\ast}$. To see this, note that $\sH$ is isometrically isomorphic to the RKHS $\sR(\mu)$, which is the $L^2$-closure of $j(E^{\ast})$, which can be identified as a quotient space of $E^{\ast}$. Thus, we may choose an ONB of $\sR(\mu)$ (and thus in $\sH$). This is possible by density. Then, for each $k \in \bbN$ choose an element $j^{-1}(e_k)$. Those elements constitute an ONB of $\sH$ (by construction) and are contained in the preimage of $i^{\ast}$ (which is what we really mean when we say ``lies in $E^{\ast}$''). To see the last part, recall that $i^{\ast} = \fC \circ j$. 

\begin{equation}
    h_k(x) = (-1)^k e^{\frac{x^2}{2}} \partial_x^k \left( e^{-\frac{x^2}{2}} \right), \quad k \geq 1, \label{symb:hermite_polynomial}
\end{equation}

and $h_0 \equiv 1$. In particular, 

\begin{align}
    h_0(x) = 1, \quad h_1(x) = x, \quad h_2(x) = x^2 - 1, \quad h_3(x) = x^3 - 3x, \quad \ldots,
\end{align}

With this convention we have that $h_k' = k h_{k-1}$, that each $h_k$ is monic, i.e. that the leading coefficient is $1$, and that the family $\left\{ \frac{h_k}{\sqrt{k!}} : k \geq 0 \right\}$ forms an ONB of $L^2(\bbR, \sN(0,1))$. 

\begin{prop}[Binomial Theorem for Hermite Polynomials] \label{prop:hermite_polynomial_binomial}
Let $h_n$ denote the $n$-th Hermite polynomial. Then for any $x,y \in \bbR$ we have the following identity

\begin{equation}
    h_n(x + y) = \sum_{k = 0}^n \binom{n}{k} h_k(x) y^{n-k} . 
\end{equation}
\end{prop}
\begin{proof}
    Taylor expanding the left-hand-side around $x$ gives 

    \begin{equation}
        h_n(x+y) = \sum_{k = 0}^n \frac{1}{k!} h^{(k)}_{n}(x) y^{k} = \sum_{k = 0}^n \underbrace{\frac{1}{k!} \frac{n!}{(n-k)!}}_{= \binom{n}{k}} h_{n-k}(x) y^{k}
    \end{equation}

    where $h^{(k)}_{n}$ denotes the $k$-th derivative of the $n$-th Hermite polynomial. Applying the identity $h_n' = n h_{n-1}$ a total of $k$ times and changing the summation index ($k \mapsto n-k$) gives the result.  
\end{proof}

Let $\alpha: \bbN \rightarrow \bbN_0$ be a multi-index with $\vert \alpha \vert := \sum_{i = 0}^{\infty} \alpha_i < \infty$ and define the multi-dimensional Hermite polynomial $H_{\alpha}$ with index $\alpha$ as the non-linear functional
    
\begin{equation}
x \mapsto H_{\alpha}(x) := \prod_{i \in \bbN} h_{\alpha_i}(\langle e_i, x \rangle), \quad x \in E. \label{eq:Hermite_poly}
\end{equation}
    
For every $k \in \bbN$ define the \textbf{$k$-th homogeneous $\bbR$-valued Wiener--Ito chaos on $(E, \mu)$}, $\calP^{(k)}(E, \mu; \bbR)$, as the closed linear subspace of $L^2(E, \mu; \bbR)$ generated by 
    
\begin{equation}
\{H_{\alpha} : \vert \alpha \vert = k\}  
\end{equation}

and the \textbf{$n$-th inhomogeneous $\bbR$-valued Wiener--Ito chaos on $(E, \mu)$}, $\calP^{(\leq n)}(E, \mu; \bbR)$, as 

\begin{equation}
    \calP^{(\leq n)}(E, \mu; \bbR) := \bigoplus_{k = 0}^n \calP^{(k)}(E, \mu; \bbR) \label{eq:decomposition_of_WIC} .
\end{equation}

\begin{lem} \label{lem:ortho_of_Hermite}
    Let $\alpha \neq \beta$. Then 

    \begin{equation}
        \int_E H_{\alpha} H_{\beta} \dd \mu = 0 .
    \end{equation}
\end{lem}
\begin{proof}
    Since $\alpha \neq \beta$, there exists an $i' \in \bbN$ s.t. $\alpha_{i'} \neq \beta_{i'}$. Furthermore, $e_i$ and $e_j$ are orthogonal in $\sH$ and thus $x \mapsto \langle e_i, x \rangle$ and $x \mapsto \langle e_j, x \rangle$ are uncorrelated w.r.t. $\mu$ for any $i \neq j$. Hence 

     \begin{align}
        \int_E H_{\alpha} H_{\beta} \dd \mu &= \int_E \prod_{i \in \bbN} h_{\alpha_i}(\langle e_i, x \rangle) \prod_{i \in \bbN} h_{\beta_i}(\langle e_i, x \rangle) \dd \mu(x) = \int_E \prod_{i \in \bbN} h_{\alpha_i}(\langle e_i, x \rangle) h_{\beta_i}(\langle e_i, x \rangle) \dd \mu(x) \\
        &= \left( \prod_{i \neq i'} \int_E  h_{\alpha_i}(\langle e_i, x \rangle)  h_{\beta_i}(\langle e_i, x \rangle) \dd \mu(x) \right) \underbrace{\int_E  h_{\alpha_{i'}}(\langle e_{i'}, x \rangle)  h_{\beta_{i'}}(\langle e_{i'}, x \rangle) \dd \mu(x)}_{=0}
    \end{align}

    where the latter term is $0$ since $\alpha_{i'} \neq \beta_{i'}$ and Hermite polynomials of different degree are orthogonal as mentioned in the beginning of Appendix \ref{sec:WIC}.
\end{proof}

In particular, this implies that the decomposition in \eqref{eq:decomposition_of_WIC} is orthogonal. \\

Now let $B$ be a separable Banach space and define for every $1 \leq p < \infty$ the Banach space $L^p(E, \mu; B)$ as the space of ($\mu$-equivalence classes of) measurable functions $\Psi: E \rightarrow B$ s.t. 
    
\begin{equation}
\int_E \Vert \Psi \Vert_B^p \dd \mu < \infty,  \label{eq:Bochner_Lebesgue_norm}
\end{equation}

with the norm induced by \eqref{eq:Bochner_Lebesgue_norm}. Define the \textbf{$k$-th (homogeneous) $B$-valued Wiener--Ito chaos on $(E, \mu)$}, $\calP^{(k)}(E, \mu; B) \label{symb:homogeneous_WIC}$, as the closed linear subspace in $L^2(E, \mu; B)$ generated by 
    
\begin{equation}
\{H_{\alpha} y : \vert \alpha \vert = k, y \in B\},  \label{eq:generators_of_hWIC_Banach}
\end{equation}

and the \textbf{$n$-th inhomogeneous $B$-valued Wiener--Ito chaos on $(E, \mu)$}, $\calP^{(\leq n)}(E, \mu; B) \label{symb:inhomogeneous_WIC}$, as 

\begin{equation}
    \calP^{(\leq n)}(E, \mu; B) := \bigoplus_{k = 0}^n \calP^{(k)}(E, \mu; B). \label{eq:decomposition_of_BV_WIC} 
\end{equation}

\begin{sloppy}
Note that as opposed to \eqref{eq:decomposition_of_WIC}, this Banach space-valued construction cannot provide an orthogonal decomposition of $L^2(E, \mu; B)$, because $L^2$-spaces with values in a Banach space do not, in general, have a Hilbert space structure. However, there are the following two propositions.
\end{sloppy}

\begin{prop}[Conditional Expectation for Elements in Finite BV WIC] \label{prop:conditional_expectation_of_finite_WIC} 
    Let $k \geq 0$, let $X \in \calP^{(\leq k)}(E, \mu; B)$, and let $N \in \bbN$. Then 

    \begin{equation}
        \bbE[X \vert \calF_N] = \sum_{\alpha \in A^{\leq k}_N} \bbE[X H_{\alpha}] H_{\alpha}. 
    \end{equation}
\end{prop}
\begin{proof}
    Let $f \in B^{\ast}$ be arbitrary. Fernique's Theorem \ref{thm:fernique} implies that $f \circ X \in L^2(E, \mu; \bbR)$ and thus the conditional expectation of $f \circ X$ w.r.t. $\calF_N$ is nothing but the $L^2$-projection in $L^2(E, \mu; \bbR)$ onto the subspace $L^2(E, \calF_N, \mu; \bbR) \subseteq L^2(E, \mu; \bbR)$ i.e. 
    
    \begin{equation}
        \bbE \left[ (f \circ X) \vert \calF_N \right] = \sum_{\alpha_i = 0, i > N} \bbE[(f \circ X) H_{\alpha}] H_{\alpha} = \sum_{\alpha \in A^{\leq k}_N} \bbE[(f \circ X) H_{\alpha}] H_{\alpha},
    \end{equation}

    where in the second equality we used the orthogonal decomposition of $L^2(E, \mu; \bbR) = \bigoplus_{k = 0}^{\infty} \calP^{(k)}(E, \mu; \bbR)$ and the fact that $X \in \calP^{(\leq k)}(E, \mu; B)$. Since the sum is finite and $f$ is linear we obtain 

    \begin{equation}
        f\left(\bbE \left[ X \vert \calF_N \right]\right) = \bbE \left[ (f \circ X) \vert \calF_N \right] = \sum_{\alpha \in A^{\leq k}_N} \bbE[(f \circ X) H_{\alpha}] H_{\alpha} = f \Bigg( \sum_{\alpha \in A^{\leq k}_N} \bbE[X H_{\alpha}] H_{\alpha} \Bigg).
    \end{equation}
    
    Since $f \in B^{\ast}$ was arbitrary, this gives the result.
\end{proof}

\begin{prop}[{\cite[Prop. V-2-6]{neveuDiscreteParameterMartingales1975}}] \label{prop:convergence_of_conditional_exp}
    Let $(\Omega, \calF, \bbP)$ be a probability space, $(\calF_N)_{N \in \bbN}$ be a discrete filtration of $\calF$ such that $\calF = \sigma \left( \bigcup_{N \in \bbN} \calF_N \right)$, $B$ be a separable Banach space, $p \in [1, \infty)$, and $X \in L^p(\Omega, \bbP; B)$. Then 

    \begin{equation}
        \bbE \left[ X \vert \calF_N\right] \rightarrow X, \quad \bbP-a.s.\! ~ \text{and} ~ \text{in} ~ L^p(\Omega, \bbP; B) .
    \end{equation}
\end{prop}

\begin{prop} (Characterization of BV WIC) \label{prop:WIC_equiv_characterizations}
Let $X \in L^2(E, \mu; B)$ and $k \geq 0$. Then the following are equivalent: 

\begin{enumerate}[(i)]
    \item $X \in \calP^{(k)}(E, \mu; B)$ %% this is the standard definition used in hairerLargeDeviationsWhiteNoise2015, Borell, ...
    \item $\forall \vert \alpha \vert \neq k: \int_E X(x) H_{\alpha}(x) \dd \mu(x) = 0$ %% this is needed to connect it to the other characterizations
    \item $\forall f \in B^{\ast}: f \circ X \in \calP^{(k)}(E, \mu; \bbR)$ %% this is needed to handle the examples easier
    \item $\forall f \in F: f \circ X \in \calP^{(k)}(E, \mu; \bbR)$ for a point separating subset $F \subseteq B^{\ast}$ i.e. $\left( \forall f \in F: f(x) = 0 \right) \Rightarrow x = 0$.
\end{enumerate}
\end{prop}
\begin{proof}
    (i) $\Rightarrow$ (ii): By assumption there exist elements $(y_{n_j})_{j \in \bbN} \subseteq B$ and multi-indices $(\beta_{n_j})_{j \in \bbN}$ with $\vert \beta_j \vert = k$ such that $\sum_{j = 1}^{m_n} H_{\beta_{n_j}} y_{n_j} \rightarrow X$ in $L^2(E, \mu; B)$ as $n \rightarrow \infty$. Let $\vert \alpha \vert \neq k$. Then by Lemma \ref{lem:ortho_of_Hermite}

    \begin{equation}
        \int_E X(x) H_{\alpha}(x) \dd \mu(x) = \lim_{n \rightarrow \infty} \sum_{j = 1}^{m_n} \underbrace{\int_E H_{\beta_{n_j}}(x) H_{\alpha}(x) \dd \mu(x)}_{=0} y_{n_j} = 0 .
    \end{equation}

    (ii) Since $X \in L^2(E, \mu; B)$, by Proposition \ref{prop:convergence_of_conditional_exp}

    \begin{equation}
        \bbE \left[ X \vert \calF_N\right] \rightarrow X \quad \text{in} ~ L^2(E, \mu; B) . \label{eq:A15}
    \end{equation}

    The left-hand side of \eqref{eq:A15} lies in $\calP^{(k)}(E, \mu; B)$ since by Proposition \ref{prop:conditional_expectation_of_finite_WIC}

    \begin{equation}
        \bbE \left[ X \vert \calF_N\right] = \sum_{\substack{\vert \alpha \vert = k \\ \alpha_i = 0, i > N}} \bbE[X H_{\alpha}] H_{\alpha} , \label{eq:conditional_expectation_of_WIC}
    \end{equation} 

    and therefore the right-hand side does too. 

    (ii) $\Rightarrow$ (iii): Let $f \in B^{\ast}$ and let $\vert \alpha \vert \neq k$. Then since bounded linear operators can be pulled into Bochner integrals 

    \begin{equation}
        0 = f\Bigg(\underbrace{\int_E X H_{\alpha} \dd \mu}_{=0} \Bigg) = \int_E f\left(X\right) H_{\alpha} \dd \mu .
    \end{equation}

    (iii) $\Rightarrow$ (iv): clear since $B^{\ast}$ separates points if $B$ is separable Banach. 

    (iv) $\Rightarrow$ (ii): Let $f \in B^{\ast}$ and let $\vert \alpha \vert \neq k$. Then 

    \begin{equation}
        0 = \int_E f\left(X\right) H_{\alpha} \dd \mu = f\left(\int_E X H_{\alpha} \dd \mu \right)
    \end{equation}

    Since $F$ separate points, this implies that $\int_E X H_{\alpha} \dd \mu = 0$.
\end{proof}

\begin{rem} \label{rem:point_evaluation_enough_for_WIC}
  (iv) of Proposition \ref{prop:WIC_equiv_characterizations} implies that if $B$ is a space of functions s.t. the point evaluation functionals $\ev_x: f \mapsto f(x)$ are continuous, then it is enough to check the condition on the point evaluation functions. 
\end{rem}

\begin{lem}(Sequential Completeness in Probability of Homogeneous BV WIC) \label{lem:complete_in_probability}
    Let $(E, \sH, i, \mu)$ be an abstract Wiener space, let $k \geq 0$, $(X_n)_{n \in \bbN}$ be a sequence in $\calP^{(k)}(E, \mu; B)$ and $X \in L^2(E, \mu; B)$ s.t. $X_n \rightarrow X$ in probability w.r.t. $\mu$. Then $X \in \calP^{(k)}(E, \mu; B)$. 
\end{lem}
\begin{proof}
    Let $f \in B^{\ast}$ be arbitrary. Then since $f$ is continuous, $f \circ X_n \rightarrow f \circ X$ in probability w.r.t. $\mu$. Since $(f \circ X_n)_{n \in \bbN} \subseteq \calP^{(k)}(E, \mu; \bbR)$ by Proposition \ref{prop:WIC_equiv_characterizations}\textit{(iii)} and $\calP^{(k)}(E, \mu; \bbR)$ is closed under convergence in probability w.r.t. $\mu$ (see \cite{borellTaylorSeriesWiener1984}) $f \circ X \in \calP^{(k)}(E, \mu; \bbR)$. As $f \in B^{\ast}$ was arbitrary, this proves the claim by Proposition \ref{prop:WIC_equiv_characterizations}\textit{(iii)}.
\end{proof}

\begin{lem} \label{lem:polynomial_in_WIC}
Let $P_k$ be a polynomial in $m$ variables of degree $k$. Then for any $\varphi_1, \ldots, \varphi_m \in E^{\ast}$ we have $P_k(\varphi_1, \ldots, \varphi_m) \in \calP^{(\leq k)}(E, \mu; \bbR)$.
\end{lem}
\begin{proof} 
    Let $\varphi \in E^{\ast}$. Let $\sR(\mu) \subseteq L^2(E, \mu; \bbR)$ be the reproducing kernel Hilbert space of $\mu$ and recall from Appendix \ref{sec:AWS} that $\sH$ may be characterized as the isometrically isomorphic image of $\sR(\mu)$ under the covariance operator $\fC_{\mu}: E^{\ast} \rightarrow E$. Now let $(e_n)_{n \in \bbN}$ be an ONB of $\sH$ contained in $E^{\ast}$. Then by the above, there exist coefficients $(\alpha_k)_{k \in \bbN} \subseteq \bbR$ s.t. $\sum_{k = 1}^n \alpha_k \langle e_k, \cdot \rangle \rightarrow \varphi$ in $L^2(E, \mu; \bbR)$ as $n \rightarrow \infty$. Therefore, since $\sum_{k = 1}^n \alpha_k \langle e_k, \cdot \rangle \in \calP^{(\leq 1)}(E, \mu; \bbR)$ for every $n \in \bbN$, by Lemma \ref{lem:complete_in_probability} we obtain $\varphi \in \calP^{(\leq 1)}(E, \mu; \bbR)$. \\

    Now let $P_k$ be a polynomial in $m$ variables of degree $k$ and $\varphi_1, \ldots, \varphi_m \in E^{\ast}$. Then by the above there are coefficients $\alpha_k^i$ s.t. 

    \begin{equation}
       \sum_{k = 1}^{n_i} \alpha_k^i e_k^i \rightarrow \varphi_i \quad \text{in} \quad L^2(E, \mu; \bbR), \quad i = 1, \ldots, m
    \end{equation}

    hence
    
    \begin{equation}
       \sum_{k = 1}^{n_i} \alpha_k^i e_k^i \rightarrow \varphi_i \quad \text{in probability w.r.t.} ~ \mu, \quad i = 1, \ldots, m
    \end{equation}

    and thus 
    
    \begin{equation}
       P_k \left( \sum_{k = 1}^{n_1} \alpha_k^1 e_k^1, \ldots, \sum_{k = 1}^{n_m} \alpha_k^m e_k^m \right)  \rightarrow P_k \left( \varphi_1, \ldots, \varphi_m \right) \quad \text{in probability w.r.t.} ~ \mu
    \end{equation}

    by the continuity of $P_k$. Since the Hermite polynomials of degree $\leq k$ span the space of polynomials of degree $\leq k$, the left-hand side is contained in $\calP^{(\leq k)}(E, \mu; \bbR)$. Hence the claim follows from Lemma \ref{lem:complete_in_probability}.    
\end{proof}

Finite Wiener--Ito chaos has the remarkable property that the topology induced by any Bochner--Lebesgue $p$-norms ($1 < p < \infty$) and that of convergence in probability w.r.t. $\mu$ all coincide. 

\begin{lem}[Equivalence of Bochner--Lebesgue $p$-Norms in BV WIC]\label{lem:equiv_of_norms} 
    Let $k \geq 0$. Then for any $X \in \calP^{(k)}(E, \mu; B)$ 

    \begin{equation}
        \Vert X \Vert_{L^p(E, \mu; B)} \leq \Vert X \Vert_{L^q(E, \mu; B)} \leq \left( \frac{q-1}{p-1}\right)^{\frac{k}{2}} \Vert X \Vert_{L^p(E, \mu; B)}, \quad 1 < p \leq q < \infty, \label{eq:equiv_of_norms_ext}
    \end{equation}

    and there exists a constant $C(k)$ s.t. for any $X \in \calP^{(\leq k)}(E, \mu; B)$

    \begin{equation}
        \Vert X \Vert_{L^p(E, \mu; B)} \leq \Vert X \Vert_{L^q(E, \mu; B)} \leq C(k) (q - 1)^{\frac{k}{2}} \Vert X \Vert_{L^p(E, \mu; B)}, \quad 2 \leq p \leq q < \infty.\label{eq:equiv_of_norms}
    \end{equation}
    
    That is, on $\calP^{(\leq k)}(E, \mu; B)$ all Bochner--Lebesgue $p$-norms are equivalent for $1 < p < \infty$ and convergence in probability w.r.t. $\mu$ is equivalent to convergence in $p$-norm for any $p \in (0, \infty)$.
\end{lem}
\begin{proof}
    For \eqref{eq:equiv_of_norms_ext} and \eqref{eq:equiv_of_norms} see \cite[Lem. 2]{frizLargeDeviationPrinciple2007} and \cite[Lem. 3]{frizLargeDeviationPrinciple2007}. Note that, at the expense of restricting the set of admissible parameters $p,q$ and switching to the inhomogeneous chaos, the constant in \eqref{eq:equiv_of_norms} does not depend on $p$. See also \cite{borellTaylorSeriesWiener1984}. 
    
    The proof of the consequence regarding convergence in probability is almost verbatim that of \cite[Thm. D.9]{frizMultidimensionalStochasticProcesses2010}, replacing the absolute value with $\Vert \cdot \Vert_{B}$ and the $L^p(E,\mu; \bbR)$-norms by $L^p(E,\mu; B)$-norms. %in conjunction with \cite[Lem. D.10]{FV10}
\end{proof}

Note that the constant in \eqref{eq:equiv_of_norms} grows like $q^{\frac{k}{2}}$ as $q \rightarrow \infty$. A naive estimate for $X \in \calP^{( \leq [\calT])}(E, \mu; \bE)$ would thus lead to $q^{\frac{k}{2}}$, where $N := \max [\calT]$. Using the homogeneous distance, we can refine this estimate to have a constant of order $q^{\frac{1}{2}}$ (albeit at the expense of the estimate only holding for $p \leq q$ large enough). 

\begin{lem}[Equivalence of $p$-Norms in Homogeneous Distance]\label{lem:equiv_of_norm_homogeneous} 
    Let $X,Y: E \rightarrow \bE$ be measurable functions s.t. $X,Y \in \calP^{( \leq [\calT])}(E, \mu; \bE)$. Then there exists a constant $C(N) < \infty$ depending only on $N := \max [\calT]$, i.e. the highest degree occurring in $\calT$, s.t. for any $2 \leq p \leq q < \infty$

    \begin{equation}
    \Vert \vertiii{Y-X} \Vert_{L^{pN}(E, \mu; \bbR)} \leq \Vert \vertiii{Y-X} \Vert_{L^{qN}(E, \mu; \bbR)} \leq C'(N) \sqrt{q} \Vert \vertiii{Y-X} \Vert_{L^{pN}(E, \mu; \bbR)}, \label{eq:equiv_of_pnorms_with_distance} 
    \end{equation}

    where, with regards to \eqref{eq:equiv_of_norms}, $C'(N) := C(N) \sqrt{N}$. 
\end{lem}
\begin{proof}
    The first inequality is true since $(E, \mu)$ is a probability space. For the second inequality of \eqref{eq:equiv_of_pnorms_with_distance}

    \begin{align}
    & \Vert \vertiii{Y-X} \Vert_{L^{qN}(E, \mu; \bE)} = \left( \int_E \vertiii{Y-X}^{qN} \dd \mu \right)^{\frac{1}{qN}} \\
    &= \left( \int_E \left( \sum_{\tau \in \calT} \Vert X_{\tau} - Y_{\tau} \Vert_{E_{\tau}}^{\frac{1}{\vert \tau\vert}} \right)^{qN} \dd \mu \right)^{\frac{1}{qN}} \leq \left( \int_E (\# \calT)^{qN - 1} \sum_{\tau \in \calT} \Vert X_{\tau} - Y_{\tau} \Vert_{E_{\tau}}^{\frac{qN}{\vert \tau\vert}}  \dd \mu \right)^{\frac{1}{qN}} \\
    &\leq \underbrace{(\# \calT)^{\frac{qN - 1}{qN}}}_{\leq \# \calT} \sum_{\tau \in \calT} \left( \int_E \Vert X_{\tau} - Y_{\tau} \Vert_{E_{\tau}}^{\frac{qN}{\vert \tau\vert}}  \dd \mu \right)^{\frac{1}{qN}} \lesssim \sum_{\tau \in \calT} \Bigg( \underbrace{\left( \int_E \Vert X_{\tau} - Y_{\tau} \Vert_{E_{\tau}}^{\frac{qN}{\vert \tau\vert}}  \dd \mu \right)^{\frac{[\tau]}{qN}}}_{= \Vert X_{\tau} - Y_{\tau} \Vert_{L^{\frac{qN}{[\tau]}}(E, \mu; E_{\tau})}} \Bigg)^{\frac{1}{[\tau]}} \\
    &= \sum_{\tau \in \calT} \Vert X_{\tau} - Y_{\tau} \Vert_{L^{\frac{qN}{[\tau]}}(E, \mu; E_{\tau})}^{\frac{1}{[\tau]}}
    \end{align}

    Thanks to working with $pN$ and $qN$ instead of just $p$ and $q$, we may apply the second estimate of \eqref{eq:equiv_of_norms} in Lemma \ref{lem:equiv_of_norms} with $2 \leq \frac{pN}{[\tau]} \leq \frac{qN}{[\tau]} < \infty$ to obtain

    \begin{equation}
    \sum_{\tau \in \calT} \Vert X_{\tau} - Y_{\tau} \Vert_{L^{\frac{qN}{[\tau]}}(E, \mu; E_{\tau})}^{\frac{1}{[\tau]}} \leq \sum_{\tau \in \calT} \underbrace{C([\tau])^{\frac{1}{[\tau]}} \left(\frac{qN}{[\tau]} - 1 \right)^{\frac{1}{2}}}_{\leq C(N) \sqrt{N} \sqrt{q}} \Vert X_{\tau} - Y_{\tau} \Vert_{L^{\frac{pN}{[\tau]}}(E, \mu; E_{\tau})}^{\frac{1}{[\tau]}}.
    \end{equation}

    Write $C'(N) := C(N) \sqrt{N}$. Note how, due to using the homogeneous distance, the exponent $\frac{1}{[\tau]}$ cancels the exponent $[\tau]$ from the scaling in Lemma \ref{lem:equiv_of_norms}. We can further estimate

    \begin{align}
    &\leq C'(N) \sqrt{q} \sum_{\tau \in \calT} \Vert X_{\tau} - Y_{\tau} \Vert_{L^{\frac{pN}{[\tau]}}(E, \mu; E_{\tau})}^{\frac{1}{[\tau]}} = C'(N) \sqrt{q} \sum_{\tau \in \calT} \left( \left( \int_E \Vert X_{\tau} - Y_{\tau} \Vert_{E_{\tau}}^{\frac{pN}{[\tau]}} \dd \mu \right)^{\frac{[\tau]}{pN}} \right)^{\frac{1}{[\tau]}} \\
    &= C'(N) \sqrt{q} \sum_{\tau \in \calT} \left( \int_E \Vert X_{\tau} - Y_{\tau} \Vert_{E_{\tau}}^{\frac{pN}{[\tau]}} \dd \mu \right)^{\frac{1}{pN}} \leq C'(N) \sqrt{q} \underbrace{(\# \calT)^{1-\frac{1}{pN}}}_{\leq (\# \calT)} \left( \int_E \sum_{\tau \in \calT}\Vert X_{\tau} - Y_{\tau} \Vert_{E_{\tau}}^{\frac{pN}{[\tau]}} \dd \mu \right)^{\frac{1}{pN}} \\
    &\lesssim C'(N) \sqrt{q} \left( \int_E  \left( \sum_{\tau \in \calT} \Vert X_{\tau} - Y_{\tau} \Vert_{E_{\tau}}^{\frac{1}{[\tau]}} \right)^{pN} \dd \mu \right)^{\frac{1}{pN}} = C'(N) \sqrt{q} \Vert \vertiii{Y-X} \Vert_{L^{pN}(E, \mu; \bE)},
    \end{align}

    which gives the result.
\end{proof}

\section{Large Deviation Principle}

Let us recall some basic notions about large deviations to fix notation. 

\begin{defn}[Large Deviation Principle]
Let $\mathcal{X}$ be a Hausdorff topological space. Then a family of measures $(\mu_{\varepsilon})_{\varepsilon > 0}$ is said to \textbf{satisfy an LDP on $\mathcal{X}$} with rate $\varepsilon^2$ and good\footnote{The adjective ``good'' refers to the property of having compact sublevel sets.} rate function $\sI: \mathcal{X} \rightarrow [0,\infty]$ if $\sI$ is lower semi-continuous, not identically $\infty$, and has compact level sets, i.e. for every $0 \leq a < \infty$ the set $\{\sI \leq a \}$ is compact in $\mathcal{X}$, and 

\begin{align*}
    \limsup_{\varepsilon \downarrow 0} \varepsilon^2 \log \mu_{\varepsilon}(A) &\leq - \inf_{x \in A} \sI(x) , \quad \text{for every closed} ~ A \subseteq \mathcal{X}, \\
    \liminf_{\varepsilon \downarrow 0} \varepsilon^2 \log \mu_{\varepsilon}(U) &\geq - \inf_{x \in U} \sI(x) , \quad \text{for every open} ~ U \subseteq \mathcal{X}. 
\end{align*}
\end{defn}

\begin{rem}
    In the entirety of the paper, all LDPs will be assumed to have speed $\varepsilon^2$ without further comment. 
\end{rem}

One of the cornerstones of large deviations for abstract Wiener model spaces is the fact that the large deviations for (classical) abstract Wiener spaces are well understood. 

\begin{thm}[{Generalized Theorem of Schilder, see e.g. \cite[Thm. 3.4.12]{deuschelLargeDeviations1989}}] \label{thm:generalized_Schilder}
    Let $(E, \sH, i, \mu)$ be an abstract Wiener space and let $\mu_{\varepsilon}(\cdot) = \mu ( \varepsilon^{-1} (\cdot))$. Then the family $(\mu_{\varepsilon})_{\varepsilon > 0}$ satisfies an LDP with good rate function $\sI: E \rightarrow [0, \infty]$ given by 

\begin{equation}
    \sI(x) = \begin{cases}
    \frac{1}{2} \Vert x \Vert^2_{\sH} & x \in \sH \\
    + \infty & \text{else}.
    \end{cases} \label{eq:Schilder_rate_app}
\end{equation}
\end{thm}

Note that the asymptotic behaviour of $\mu_{\varepsilon}(A)$ is determined exclusively by $A \cap \sH$, despite the fact that $\mu(\sH) = 0$ whenever $\dim \sH = \infty$.

\begin{lem} \label{lem:log_equiv} 
    Let $(a_{\varepsilon})_{\varepsilon > 0}, (b_{\varepsilon})_{\varepsilon > 0} \subseteq [0, \infty)$. Then 

    \begin{equation}
        \limsup_{\varepsilon \rightarrow 0} \varepsilon^2 \log (a_{\varepsilon} + b_{\varepsilon}) = \max \left(\limsup_{\varepsilon \rightarrow 0} \varepsilon^2 \log a_{\varepsilon}, \limsup_{\varepsilon \rightarrow 0} \varepsilon^2 \log b_{\varepsilon} \right) . \label{eq:asymp_equiv_raw}
    \end{equation}
\end{lem}
\begin{proof}
    Let $\varepsilon > 0$. Then since $a_{\varepsilon}, b_{\varepsilon} \geq 0$ 

    \begin{equation}
        \max\{ a_{\varepsilon}, b_{\varepsilon} \} \leq a_{\varepsilon} + b_{\varepsilon} \leq 2 \max\{ a_{\varepsilon}, b_{\varepsilon} \}
    \end{equation}

    By monotonicity, applying $\log$ yields  

    \begin{equation}
        \max\{ \log(a_{\varepsilon}), \log(b_{\varepsilon}) \} \leq \log(a_{\varepsilon} + b_{\varepsilon}) \leq \log(2) + \max\{ \log(a_{\varepsilon}), \log(b_{\varepsilon}) \}.
    \end{equation}

    Multiplying by $\varepsilon^2$ and taking $\limsup_{\varepsilon \rightarrow 0}$ gives the result. 
\end{proof}

\begin{thm}[{Extended Contraction Principle; \cite[Lem. 3.3]{hairerLargeDeviationsWhiteNoise2015}}]\label{thm:HW_extended_contraction_principle} 
Let 
    
    \begin{enumerate}[(i)]
        \item $(\mathcal{X}, d)$ and $(\mathcal{Y}, d')$ be separable metric spaces,
        \item $( \mu_{\varepsilon})_{\varepsilon > 0}$ a family of probability measures on $\mathcal{X}$ satisfying an LDP with good rate function $\sI$, and
        \item $(\Psi_{\varepsilon})_{\varepsilon \geq 0}$ a family of functions $\mathcal{X} \rightarrow \mathcal{Y}$ which are continuous on neighborhoods of $\{x \in \mathcal{X}: \sI(x) < \infty \}$ s.t. for every $C \in \bbR$ there exists a neighborhood $O_C$ of $\{x \in \mathcal{X}: \sI(x) \leq C \}$ s.t. 

        \begin{equation}
            \limsup_{\varepsilon \downarrow 0} \sup_{x \in O_C} d'(\Psi_{\varepsilon}(x), \Psi_0(x)) = 0 . \label{eq:convergence_condition_in_contr_principle}
        \end{equation} 
    \end{enumerate}

    Then the family $\left( (\Psi_{\varepsilon})_{\ast} \mu_{\varepsilon} \right)_{\varepsilon > 0}$ satisfies an LDP on $\mathcal{Y}$ with good rate function

    \begin{equation}
        \sI'(y) := \inf \{ \sI(x) : x \in \mathcal{X}, \Psi_0(x) = y \}, \quad y \in \mathcal{Y},
    \end{equation}

    with the convention that $\inf \emptyset = \infty$.
\end{thm}

\section{Symbolic Index} \label{app:symbolic_index}

% ---- list of symbols: -------

\begin{tabular}{l l l}
\textbf{Symbol} & \textbf{Meaning} & \textbf{Ref.} \\ \hline %%spaces
$C(A;\bbR)$ & space of continuous $\bbR$-valued functions on a space $A$ & p.~\pageref{symb:continuous_functions}\\
$\calC^{\alpha}(A)$ & space of $\alpha$-Hölder continuous functions on a space $A$ & p.~\pageref{symb:hoelder_space_closure_of_smooth}\\
$\calC^{p-\var}(A)$ & space of continuous functions on a space $A$ with finite $p$-variation & p.~\pageref{symb:variation}\\
$\calC^{0,\alpha}(A)$ & closure of smooth functions in the $\calC^{\alpha}$-norm & p.~\pageref{symb:hoelder_space_closure_of_smooth} \\
$\calC^{0,p-\var}(A)$ & closure of smooth functions in the $p-\var$-norm & p.~\pageref{symb:variation_smooth} \\ 
$\Vert \cdot \Vert_{\rho-\var; [0,T]^2}$ & $2$-dimensional $\rho$-variation norm & p.~\pageref{symb:rho_var}\\ 
$W^{1,2}_0$ & $(1,2)$-Sobolev space of functions with $x(0) = 0$ & p.~\pageref{symb:sobolev_space_starting_at_0} \\
$\dot{H}_0^1(U)$ & homogeneous Dirichlet--Sobolev--Hilbert space on an open set $U$ & p.~\pageref{symb:sobolev_space_starting_at_0} \\
$\sR(\mu)$ & reproducing kernel Hilbert space of a Gaussian measure $\mu$ & p.~\pageref{symb:reproducing_kernel_hilbert_space} \\ 
$\texttt{CM}(E, \mu)$ & Cameron--Martin space of a Gaussian measure $\mu$ on $E$ & p.~\pageref{symb:CM_space} \\ 
$\Vert \cdot \Vert_{\bE}$ & norm on $\bE$ & p.~\pageref{symb:norm_on_bE}\\
$\vertiii{\cdot}$ & homogeneous distance on $\bE$ & p.~\pageref{symb:homogeneous_distance}\\
$(E, \sH, i, \mu)$ & abstract Wiener space & p.~\pageref{defn:AWS}\\ 
$E_{\calN}$ & direct summand of $\bE$ associated to the distinguished symbol $\calN$ & p.~\pageref{symb:E1}\\ 
$\pi := \pi_{\calN}$ & projection onto $E_{\calN}$, $\calN \subseteq \calT$ & p.~\pageref{symb:pi}\\ 
$\calP^{(k)}(E, \mu; B)$ & $k$-th homogeneous $B$-valued Wiener--Ito chaos  & p.~\pageref{symb:homogeneous_WIC}\\ 
$\calP^{(\leq k)}(E, \mu; B)$ & $k$-th inhomogeneous $B$-valued Wiener--Ito chaos  & p.~\pageref{symb:inhomogeneous_WIC}\\ \hline 
%%Operators/functions
$\rmm_{\varepsilon}$ & scalar multiplication by $\varepsilon > 0$ & p.~\pageref{symb:m_scaling}\\ 
$\delta_{\varepsilon}$ & dilation on $\bE$ by $\varepsilon > 0$ & p.~\pageref{eq:definition_of_dilation}\\
$T_h$ & classical translation operator w.r.t. $h$ & p.~\pageref{symb:ordinary_shift}\\
$\bT_{\bh}$ & enhanced translation operator w.r.t. $\bh$ & p.~\pageref{eq:shift_defn}\\
$\mu_h$ & measure $\mu$ shifted in the direction of $h$ & p.~\pageref{symb:shifted_measure}\\
$f_h$ & Cameron--Martin density/Radon--Nikodým derivative $\frac{\dd \mu_h}{\dd \mu}$ & p.~\pageref{symb:f_h}\\ 
$\overline{\Psi}$ & proxy restriction of $\Psi$ & p.~\pageref{eq:definition_of_bar_lift}\\
$\fC_{\mu}$ & covariance operator associated to a Gaussian measure $\mu$ & p.~\pageref{symb:covariance_operator}\\
$\rmP_m$ & Projection onto the subspace of $E$ spanned by $e_1, \ldots, e_m$ & p.~\pageref{symb:projection_onto_finite_basis}\\  \hline 
%%general symbols 
$\underline{h}$ & inverse image of $h \in \sH$ under covariance operator $\fC_{\mu}$ & p.~\pageref{symb:underline_h}\\
$\sN(a,\sigma^2)$ & normal distribution with expectation $a \in \bbR$ and variance $\sigma^2 \geq 0$ & p.~\pageref{symb:normal_distr}\\
$h_n$ & $n$-th Hermite polynomial & p.~\pageref{symb:hermite_polynomial}\\
$H_{\alpha}$ & $\alpha$-th multi-dimensional Hermite polynomial & p.~\pageref{eq:Hermite_poly}\\
$\sB_{A}, \sB_{A}^{\mu}$ & Borel $\sigma$-algebra on a space $A$ (completed w.r.t. a measure $\mu$) & p.~\pageref{symb:Borel_sigma}\\ 
$[\Psi]_{\mu}$ & $\mu$-a.s.\! equivalence class of a measurable function $\Psi$ & p.~\pageref{symb:equiv_class}\\ 
$A^k_n, A^{\leq k}_n$ & set of multi-indices of degree $k/\leq k$ supported inside $\{1, \ldots, n\}$ & p.~\pageref{symb:Akn}\\ 
$\supp(\mu)$ & topological support of a Borel measure $\mu$ & p.~\pageref{symb:support}\\ 
$\mu \approx \nu$ & equivalence of measures $\mu$ and $\nu$; i.e. existence of $\frac{\dd \mu}{\dd \nu}$ and $\frac{\dd \nu}{\dd \mu}$ & p.~\pageref{symb:equivalence_of_measures}\\ 
$\mathfrak{s}$ & parabolic scaling & p.~\pageref{symb:parabolic_scaling}\\
$D_m$ & dyadic partition at scale $m$ & p.~\pageref{symb:dyadic_partition}\\
$P([0,T])$ & set of partitions on $[0,T]$ & p.~\pageref{symb:partition}\\
% $\calS(n)$ & symmetric group of order $n$ & p.~\pageref{symb:symmetric_group}\\
$\calG$ & renormalization group & p.~\pageref{symb:renormalization_group}\\
% $\texttt{Sym}$ & symmetrization operator on tensors & p.~\pageref{symb:symmetrization}\\
$\Ren_{\delta}$ & renormalization at correlation length $\delta > 0$ & p.~\pageref{symb:renormalization}\\
$\otimes_{A}$ & algebraic tensor product & p.~\pageref{symb:algebraic_tensor}\\
% $\hat{\otimes}$ & symmetric tensor product & p.~\pageref{symb:symmetric_tensor}\\
$\shuffle$ & shuffle product & p.~\pageref{symb:shuffle}\\ 
\end{tabular}

\newpage

\printbibliography

\end{document}